\documentclass[11pt,oneside,reqno]{amsart} 

\usepackage[english]{babel}
\usepackage[T1]{fontenc}
\usepackage{lmodern}

\usepackage{fancyhdr}
\usepackage[usenames]{color}
\usepackage{bbm}
\usepackage[noadjust]{cite}
\usepackage{amsmath,amssymb,latexsym, amsfonts, amscd, amsthm}

\usepackage{mathrsfs}

\usepackage{enumitem}

\usepackage{mathtools}

\usepackage[color=blue!10]{todonotes}

\usepackage{graphicx}

\usepackage[mathscr]{euscript} 

\usepackage{tikz-cd}
\usetikzlibrary{arrows}
\usepackage{aliascnt}

\usepackage{relsize}
\usepackage[bbgreekl]{mathbbol}
\DeclareSymbolFontAlphabet{\mathbb}{AMSb} 
\DeclareSymbolFontAlphabet{\mathbbl}{bbold}
\newcommand{\Prism}{{\mathlarger{\mathbbl{\Delta}}}}

\usepackage{hyperref}
\hypersetup{pdftoolbar=true, pdftitle={OnExactness}, pdffitwindow=true, colorlinks=true, citecolor=blue, filecolor=black, linkcolor=red, urlcolor=red, hypertexnames=false}

\numberwithin{equation}{subsection}

\newtheoremstyle{thmlike}
{8pt}
{3pt}
{\slshape}
{}
{\bfseries}
{.}
{1em}
{}

\newtheoremstyle{deflike}
{8pt}
{3pt}
{}
{}
{\bfseries}
{.}
{1em}
{}

\theoremstyle{thmlike}
\newtheorem{theorem}{Theorem}[section]
\newtheorem{introtheorem}{Theorem}

\newtheorem{proposition}[theorem]{Proposition}
\newtheorem{lemma}[theorem]{Lemma}
\newtheorem{corollary}[theorem]{Corollary}

\theoremstyle{deflike}
\newtheorem{definition}[theorem]{Definition}

\newtheorem{remark}[theorem]{Remark}
\newtheorem*{remark*}{Remark} 
\newtheorem{example}[theorem]{Example}

\renewcommand{\ker}{\operatorname{Ker}}
\newcommand{\coker}{\operatorname{Coker}}

\renewcommand{\hom}{\operatorname{Hom}}
\newcommand{\ext}{\operatorname{Ext}}
\newcommand{\CH}{\operatorname{CH}}

\newcommand{\coh}{\operatorname{\mathsf{Coh}}}
\newcommand{\perf}{\operatorname{\mathsf{Perf}}}
\newcommand{\vect}{\operatorname{\mathsf{Vect}}}
\newcommand{\Mod}{\operatorname{\mathsf{Mod}}}
\renewcommand{\mod}{\operatorname{\mathsf{mod}}}
\newcommand{\Rep}{\operatorname{\mathsf{Rep}}}

\newcommand{\gr}{\operatorname{gr}}

\newcommand{\ev}{\operatorname{ev}}

\newcommand{\spf}{\operatorname{Spf}}
\newcommand{\Hom}{\operatorname{Hom}}

\newcommand{\Ext}{\operatorname{Ext}}
\newcommand{\Tor}{\operatorname{Tor}}
\newcommand{\img}{\operatorname{im}}

\newcommand{\oh}{\operatorname{\mathcal{O}}}

\newcommand{\Gm}{\operatorname{\mathbb{G}_{m}}}

\newcommand{\tor}{\operatorname{tor}}
\newcommand{\Fil}{\operatorname{Fil}}
\newcommand{\FI}{\operatorname{FI}}

\newcommand{\res}{\operatorname{res}}
\newcommand{\PD}{\operatorname{PD}}

\newcommand{\Gal}{\operatorname{Gal}}
\newcommand{\st}{\operatorname{st}}
\newcommand{\cris}{\operatorname{cris}}

\newcommand{\et}{\mathrm{\acute{e}t}}
\newcommand{\dr}{\mathrm{dR}}
\newcommand{\HT}{\mathrm{HT}}

\newcommand{\syn}{\mathrm{syn}}
\newcommand{\refl}{\mathrm{refl}}

\newcommand{\Es}{\operatorname{\mathfrak{S}}}
\newcommand{\Em}{\operatorname{\mathfrak{M}}}
\newcommand{\En}{\mathcal{N}}
\newcommand{\Ef}{\operatorname{\mathfrak{F}}}
\newcommand{\ainf}{\operatorname{A_{inf}}}

\renewcommand{\AA}{\mathbb{A}}
\newcommand{\ZZ}{\mathbb{Z}}
\newcommand{\NN}{\mathbb{N}}
\newcommand{\QQ}{\mathbb{Q}}

\newcommand{\CC}{\mathbb{C}}

\renewcommand{\H}{\mathrm{H}}

\newcommand{\ModPrime}{{}^{\backprime}\hspace{-2pt}\Mod}

\newcommand{\lin}{\mathrm{lin}}
\newcommand{\fr}{\mathrm{fr}}
\newcommand{\BK}{\mathrm{BK}}

\title[Ex. Property of Breuil--Kisin Functors and Bloch--Kato Selmer Groups]{Exactness property of Breuil--Kisin functors and Bloch--Kato Selmer groups}

\begin{document}

\begin{abstract} 
\noindent Let $K$ be a $p$-adic field and $T$ a lattice in a semistable representation of $\Gal(\overline{K}/K)$ with Hodge-Tate weights in $[0, r]$. Assuming $0\leq r<p-1$, we prove that for a semistable extension of $\ZZ_p$  by  $T$, the corresponding sequence of strongly divisible modules is exact.  Analogous statements are proved for Breuil--Kisin modules and for prismatic $F$-crystals for all $r\geq 0$. 
In the crystalline case, we deduce that the integral Bloch-Kato Selmer group $H^1_f(K, T)$ is computed by $\Ext^1$ in the category of crystalline strongly divisible modules. Using further exactness results, we define a tensor product of strongly divisible modules, which commutes with the functors to Galois representations. As an application, we show that for abelian varieties $A_1, A_2$ over $K$ with good reduction, the  cup product map $\delta_1\cup\delta_2:A_1(K)\otimes A_2(K)\rightarrow H^2(K, T_p(A_1)\otimes T_p(A_2))$ induced by the Kummer sequences of $A_1, A_2$ factors through an $\Ext^2$ group of strongly divisible modules.
In the appendix we generalize further the main exactness results  
to coherent reflexive $F$-gauges.
\end{abstract}

\author[P. \v{C}oupek]{Pavel \v{C}oupek}
\address{Pavel \v{C}oupek, Department of Mathematics, University of Virginia, Kerchof Hall, 141 Cabell Dr., Charlottesville, VA 22903, USA}
\email{kym4jc@virginia.edu}

\author[E. Gazaki]{Evangelia Gazaki}
\address{Evangelia Gazaki, Department of Mathematics, University of Virginia, Kerchof Hall, 141 Cabell Dr., Charlottesville, VA 22903, USA}
\email{eg4va@virginia.edu}

\author[A. Marmora]{Adriano Marmora}
\address{Adriano Marmora, Institut de Recherche Math\'{e}matique Avanc\'{e}e, 7 rue Ren\'{e} Descartes, 67084 Strasbourg, France}
\email{marmora@math.unistra.fr}

\maketitle

\section{Introduction}

In a celebrated paper \cite{BlochKato}, Bloch and Kato defined the local component of Bloch--Kato Selmer groups of a $p$-adic Galois representation of a number field at places above $p$. These groups should not only control the order of vanishing of the corresponding $L$-function but in their integral form dictate the exact special values of the $L$-function (see for example \cite[C \S11]{FontaineBourb}).

Let $K$ be a $p$-adic field, $\oh_K$ its ring of integers, and $G_K=\mathrm{Gal}(\overline{K}/K)$ its absolute Galois group. A lattice in a $p$-adic $G_K$-representation is a finite free $\ZZ_p$-module $T$, endowed with a continuous action of $G_K$. Such a lattice is called crystalline if the $p$-adic representation 
$V=T[1/p]$      
is crystalline in the sense of Fontaine \cite{Fontaine82}. For a crystalline lattice $T$, the  Bloch--Kato (local) Selmer group $H^1_f(K, T)$ is the subgroup of the continuous Galois cohomology $H^1(K,T)$ containing the classes of crystalline lattices
\begin{equation*}
\begin{tikzcd}
0\ar[r]& T \ar[r] & L \ar[r] & \ZZ_p \ar[r] & 0.
\end{tikzcd}
\end{equation*}

 Integral $p$-adic Hodge Theory aims to give a more explicit description of categories of representations coming from motives. After works of  Faltings, Kato, Fontaine, Breuil, Kisin, \emph{et al.}, Bhatt and Scholze \cite{BhattScholze2} gave a geometric incarnation of the category of crystalline lattices as the category of prismatic 
 $F$-crystals on $\spf \oh_K$.
 Precisely, they construct an equivalence of categories, called the  étale realization functor, from the category of prismatic $F$-crystals on $\spf \oh_K$ to the category of crystalline lattices of $G_K$-representations. 
It is tempting to try to identify $H^1_f(K, T)$ to a Yoneda $\Ext^1$ group in the category of prismatic 
 $F$-crystals, 
 however  even 
 though the étale realization functor is exact, its quasi-inverse is not.  
 From a geometric point of view, this is analogous to the fact that
 the category of holomorphic vector bundles on a punctured complex surface is equivalent to the category of bundles on the whole surface, but while the restriction functor is exact, its quasi-inverse (i.e., the direct image functor) is not.  
 In abstract terms this is not surprising either. The categories are not abelian,  
 but only exact (in the sense of Quillen \cite{Quillen}).  
 Since the datum of an exact category is an additional structure rather than a property of the category, an equivalence between such categories often fails to be exact in one or both directions. 
 What is more surprising is the following exactness result which we prove.

\begin{introtheorem}[Theorem~\ref{thm:ExactnessPrismatic}]\label{thm:MainThmPrismatic}
Let $T$ be a lattice in a crystalline $G_K$-representation with all Hodge--Tate weights non-negative. 
\begin{enumerate}\item Given  a short exact sequence of crystalline $G_K$-lattices of the form 
\begin{equation}\tag{$\tau$}
\begin{tikzcd}
0\ar[r]& T \ar[r] & L \ar[r] & \ZZ_p \ar[r] & 0,
\end{tikzcd}
\end{equation}
the corresponding sequence of prismatic $F$-crystals on $\spf \oh_K$
\begin{equation}\tag{$\delta$}
\begin{tikzcd}
0\ar[r]& D_{\Prism}(T) \ar[r] & D_{\Prism}(L) \ar[r] & \oh_{\Prism} \ar[r] & 0,
\end{tikzcd}
\end{equation}
is again short exact.
\item Consequently, we have a natural isomorphism $H^1_f(K, T)\simeq  \Ext^1(\oh_{\Prism}, D_{\Prism} (T))$, where $\Ext^1$ denotes 
the Yoneda $\Ext$ group in the category of prismatic $F$-crystals.
\end{enumerate}
\end{introtheorem}

Our strategy to prove Theorem \ref{thm:MainThmPrismatic} is to work with Breuil--Kisin modules \cite{Kisin1}. Since the Breuil--Kisin prism covers the final object of the prismatic topos of $\spf(\oh_K)$, we are 
reduced to proving the following theorem, which is the technical core of the paper.

\begin{introtheorem}[Theorem~\ref{thm:ExactKisinModules}, Remark~\ref{rem:ExactKisinModlesComments}]\label{thm:MainThmKisin} 
Given a short exact sequence of lattices in semistable $G_K$-representations with non-negative Hodge-Tate weights 
of the form  \begin{equation}\tag{$\tau$}
\begin{tikzcd}
0\ar[r]& T \ar[r] & L \ar[r] & \ZZ_p \ar[r] & 0,
\end{tikzcd}
\end{equation}
the sequence of associated Breuil--Kisin modules
\begin{equation}\tag{$\mu$}
\begin{tikzcd}
0\ar[r]& \underline{\Em}(T) \ar[r] & \underline{\Em}(L) \ar[r] & \Es \ar[r] & 0
\end{tikzcd}
\end{equation}
corresponding to $(\tau)$  
under the covariant \'{e}tale realization functor is a short exact sequence of Breuil--Kisin modules. 
\end{introtheorem}

We emphasize that the assumption that $T$ has non-negative Hodge--Tate weights is crucial for our results (a counterexample for $T=\mathbb{Z}_p(-1)$ is given as Example~\ref{ex:KeyCounterexample} in the main text). Let us point out that the convention  
we use is that the cyclotomic character is of Hodge--Tate weight $1$.

\medskip
Coming back to the original motivation of computing (integral, local) Bloch--Kato Selmer groups, 
we need to formulate variants of the above theorems in a more computable setting. By using Theorem \ref{thm:MainThmKisin}
  we  prove the following theorem, formulated
 in terms of strongly divisible modules of Breuil \cite{Breuil99}. By a theorem of Liu \cite{Liu08}, fixing $r<p-1$, the category $\Mod(S)_{r, \log}^{\varphi, N}$ of all weight $r$ strongly divisible modules is equivalent to the category of lattices in semistable representations with Hodge--Tate weights in $[0, r]$, with the subcategory $\Mod(S)_{r}^{\varphi, N}$ corresponding to crystalline lattices 
 (see \S\ref{subsec:BreuilModules} in the main text for a recollection of the context). 
 For a lattice $T$ in a semistable representation of $G_K$, let us  denote by  $H^1_{\st}(K, T)$ the semistable version of Bloch--Kato cohomology introduced in \cite{FontainePerrinRiou, Nekovar}. We then have the following.

\begin{introtheorem}[Corollary~\ref{cor:ExactBreuilModules}, Theorem~\ref{Th:isomorphism_ext}, Theorem~\ref{Th:isomorphism_ext_st}]\label{thm:MainThmBreuil}
Fix $r \in \mathbb{N}$ with $r<p-1$. Given a short exact sequence of lattices in semistable $G_K$-representations with Hodge-Tate weights in $[0, r]$ of the form  \begin{equation}\tag{$\tau$}
\begin{tikzcd}
0\ar[r]& T \ar[r] & L \ar[r] & \ZZ_p \ar[r] & 0,
\end{tikzcd}
\end{equation}
the corresponding sequence 
\begin{equation}\tag{$\mu$}
\begin{tikzcd}
0\ar[r]& \mathcal{M}_{\st}(T) \ar[r] & \mathcal{M}_{\st}(L) \ar[r] & \mathbf{1} \ar[r] & 0
\end{tikzcd}
\end{equation} 
of strongly divisible modules is short exact.
Consequently:
\begin{enumerate}
\item{For every lattice $T$ as above, there is a natural isomorphism $$H^1_{\st}(K, T) \simeq \Ext^1_{\Mod(S)_{r, \log}^{\varphi, N}}(\mathbf{1}, \mathcal{M}_{\st}(L)), $$
where 
the group on the right is the Yoneda $\Ext$ in the category of strongly divisible modules.
} \label{thm:MainThmBreuilpart1}

\item{If, additionally, $T[1/p]$ is crystalline, then there is a natural isomorphism $$H^1_f(K, T) \simeq \Ext^1_{\Mod(S)_{r}^{\varphi, N}}(\mathbf{1}, \mathcal{M}_{\st}(L)), $$ where 
the group on the right is the Yoneda $\Ext$ in the category of crystalline strongly divisible modules.
} \label{thm:MainThmBreuilpart2}
\end{enumerate}
\end{introtheorem}

\begin{remark*} Several comments about the above theorems are in order.
\begin{enumerate}
    \item 
Part \eqref{thm:MainThmBreuilpart2} of Theorem~\ref{thm:MainThmBreuil} fills a gap in the paper of Iovita--Marmora \cite{Iovita/Marmora} where crystalline strongly divisible modules were employed to define an explicit complex computing such Selmer groups, and an exactness result as in Theorem~\ref{thm:MainThmBreuil} was implicitly assumed (see Remark \ref{remarkGap} in the main text). 

\item \label{RemarkIntro2} 
The statements \eqref{thm:MainThmBreuilpart1}, \eqref{thm:MainThmBreuilpart2} of Theorem~\ref{thm:MainThmBreuil} also have counterparts in the Breuil--Kisin context. 
  There are multiple ways to enhance the category of free Breuil--Kisin modules to one equivalent with the category of $G_K$-lattices in semistable (resp. crystalline) representations, cf.
 \cite{LiuPhiGHat, DuLiuPhiGHat, GaoBKGK, YaoSemistableLattices},\cite[Appendix F]{EmertonGee}. 
 Analogues of \eqref{thm:MainThmBreuilpart1}, \eqref{thm:MainThmBreuilpart2} of Theorem~\ref{thm:MainThmBreuil} are then valid in these enhancements; moreover, the requirement $r<p-1$ is not needed. We formulate this result explicitly in the case of Breuil--Kisin $G_K$-modules of Gao \cite{GaoBKGK} (see Theorem~\ref{thm:BKGKComputesCohomology} for the precise statement). Regarding the semistable variant, let us point out that Theorem \ref{thm:MainThmPrismatic} likewise admits a logarithmic version (see Remark
\ref{rem:PrismaticComments} \eqref{rem:PrismaticComments2} in the main text). 
\end{enumerate}
\end{remark*}

It is worth pointing out that Theorem~\ref{thm:MainThmPrismatic} leads to further explicit variants. In fact, a key feature of the prismatic formalism is certain ``coordinate independence'' which allows one to propagate the exactness result further, for example, to the theory of Wach modules \cite{Wach, Wach2, ColmezWachM, BergerWachM}. If $K/\QQ_p$ is unramified and $T$
is a crystalline $G_K$-lattice, Abhinandan \cite{AbhinandanBlochKato} constructs a complex $\mathcal{S}^{\bullet}(M)$, defined in terms of the Wach module $M$ corresponding to $T$, computing the Yoneda $\Ext^1$ in that category.
 He can then deduce a natural isomorphism  $H^i(\mathcal{S}^{\bullet}(M))[1/p] \simeq H^i_f(K, T[1/p])$.  
As a consequence of Theorem~\ref{thm:MainThmPrismatic}, we show that  $H^1(\mathcal{S}^{\bullet}(M))$ already agrees with the integral Bloch--Kato Selmer group $H^1_f(K, T)$, assuming that all Hodge--Tate weights of $T$ are non-negative (see Corollary~\ref{cor:WachModulesSyntomic} in the main text).

After finishing the first version of this paper, we were asked by Bhargav Bhatt whether Theorem~\ref{thm:MainThmPrismatic} can be strengthened to the case of reflexive prismatic $F$-gauges, a theory initiated by Drinfel'd \cite{DrinfeldPrismatization} and further developed by Bhatt--Lurie \cite{BhattLurie1, BhattLurie2, BhattFGauge}, enhancing that of prismatic $F$-crystals. This is indeed the case, and we present the argument separately in an appendix to this paper (see Theorem~\ref{thm:ExactnessFGauges} 
 and Corollary \ref{cor:BlochKatoBySyntomic}).

\medskip
 Theorem \ref{thm:MainThmBreuil} is deduced from Theorem \ref{thm:MainThmKisin} by using the (exact) comparison functor $\mathcal{M}_{\Es}^{(r)}: \Mod_{\Es, \fr}^{\varphi,\leq r}\to\Mod(S)_{r}^\varphi$ between free Breuil-Kisin modules of height $\leq r$ and strongly divisible modules of weight $\leq r$ without monodromy, which is an equivalence (cf.~\cite{CarusoLiu}). We show compatibility of this functor with the covariant \'{e}tale realization functors. Additionally, we prove that its quasi-inverse, denoted by $\underline{\Em}_S^{(r)}$, is  exact (see Theorem \ref{thm:KisinToBreuilIsBiExact}). Using these results, we define for a pair of integers $0\leq r_1, r_2<p-1$ a tensor product of Breuil modules, 
 \[\otimes: \Mod(S)_{r_1,\log}^{\varphi, N}\times \Mod(S)_{r_2,\log}^{\varphi, N}\to \Mod(S)_{r_1+r_2,\log}^{\varphi, N},\] 
 (see Definitions \ref{def:TensorProductQuasiSDMod}, \ref{def:TensorProductSDMod}). As expected, this tensor product commutes  with the \'{e}tale realization functors. Namely, there is an isomorphism of $\ZZ_p[G_K]$-representations 
 \[T_{\st,\star}(M_1\otimes M_2)\simeq T_{\st,\star}(M_1)\otimes_{\ZZ_p}T_{\st,\star}(M_2),\]
 for $M_i\in \Mod(S)_{r_i,\log}^{\varphi, N}$. 

Combining the tensor product with Theorem \ref{thm:MainThmBreuil}, we obtain an interesting application to abelian varieties. Let $A_1, A_2$ be abelian varieties over $K$ with good reduction. For $i=1,2,$ let $\delta_i:A_i(K)\to H^1(K, T_p(A_i))$ be the homomorphism induced by the connecting homomorphisms of the Kummer sequences for $A_i$. It follows by \cite{BlochKato} that the image of $\delta_i$ is precisely the Bloch--Kato Selmer group $H^1_f(K, T_p(A_i))$. 
\begin{introtheorem} [Theorem \ref{thm:cupproduct}, Corollary \ref{cor:abelianvars}]\label{thm:abvarsintro}
The cup product map  \[\delta_1\cup\delta_2: A_1(K)\otimes A_2(K)\to H^2(K, T_p(A_1)\otimes T_p(A_2))\] factors through $\Ext^2_{\Mod(S)_2^{\varphi, N}}(\mathbf{1}, \mathcal{M}_{\st}(T_p(A_1))\otimes \mathcal{M}_{\st}(T_p(A_2)))$. 
\end{introtheorem}
Let us emphasize that we don't expect an analog of Theorem \ref{thm:MainThmBreuil} to hold for $2$-extensions. 
To prove Theorem \ref{thm:abvarsintro},  we define instead a cup product structure on $\Ext$ groups of Breuil modules (Theorem \ref{thm:cupproduct}), and use it to prove such a factoring.

The reason we are interested in the cup product map $\delta_1\cup\delta_2$ is because it is related to questions about algebraic cycles and $K$-theory. As an example, when $X$ is either an abelian surface or a product $C_1\times C_2$ of smooth projective curves over $K$, the map $\delta_1\cup\delta_2$ plays a significant role in understanding the behavior of the $p$-adic cycle class map to \'{e}tale cohomology 
\[cl_p:\varprojlim_n\CH_0(X)/p^n\to H^4_{\et}(X, \ZZ_p(2)),\] where $\CH_0(X)$ is the Chow group of $0$-cycles on $X$. Namely, when restricted to the kernel of the Abel--Jacobi map, $cl_p$ is built up from cup products as above. When the abelian varieties involved are elliptic curves both having either good ordinary or split multiplicative reduction, the map $cl_p$ has been shown to be injective (\cite{GL21, Yam05}). However, when $X=E_1\times E_2$ is a product of two elliptic curves with good supersingular reduction, we suspect that this might not be the case, especially when the ramification index $e_K$ of $K$ is large enough. Corollary \ref{cor:abelianvars} is a first step towards a strategy to show this. Namely, we suspect that in this case the map $$\Ext^2_{\Rep_{\ZZ_p,\cris}(G_K)}(\ZZ_p, T_p(E_1)\otimes T_p(E_2))\to H^2(K, T_p(E_1)\otimes T_p(E_2))$$ is not injective. Working instead with $2$-extensions of strongly divisible modules might make this problem more tractable, and we are going to explore this in a future paper.

\subsection*{Outline} In Section~\ref{sec:background}, we recall the various categories of Breuil--Kisin and Breuil modules we employ. We also clarify the relationship of various versions of \'{e}tale realization functors from these categories. This is needed in order to effectively pass between Breuil--Kisin and Breuil modules and their attached Galois representations. The main results on exactness are then discussed in Section~\ref{sec:ExactnessResults}. Here we prove Theorem~\ref{thm:MainThmKisin} and deduce Theorems~\ref{thm:MainThmPrismatic}, \ref{thm:MainThmBreuil}   using the relations established in Section~\ref{sec:background}. We also discuss auxiliary results, such as the aforementioned exactness of the functor $\underline{\Em}^{(r)}_S$, or the fact that the functor $\mathcal{M}_{\st}$ from $G_K$--lattices to strongly divisible modules is not even left exact in general. In Section~\ref{sec:Applications}, we define tensor product of strongly divisible modules and prove that it is exact. Using the tensor product to define a cup product on Ext groups of strongly divisible modules, we establish Theorem~\ref{thm:abvarsintro}. Finally, in Appendix~\ref{appendix} we extend Theorem~\ref{thm:MainThmKisin} by proving an exactness statement on the associated Nygaard filtration of Breuil--Kisin modules, and deduce a variant of Theorem~\ref{thm:MainThmPrismatic} applicable to prismatic $F$-gauges.

\subsection*{Acknowledgment} We would like to thank Adrian Iovita for many helpful discussions; Theorem \ref{thm:cupproduct} in particular was largely inspired by his ideas. We also thank Tong Liu for his interest in our work, and for several discussions regarding (non-)exactness of various functors. His work \cite{Liu08} played a particularly important role in the genesis of this article. We are also grateful to Bhargav Bhatt and Preston Wake for contributing interesting comments that improved the quality of the paper, and to Dylan Pentland for several discussions regarding $F$-gauges and his paper \cite{Pentland}. 
The third named author would like to thank Arthur-C\'{e}sar Le Bras, Carlo Gasbarri, Pierre Guillot and Sacha Ikonicoff for helpful discussions.

Support for the first named author's research was provided by
an AMS--Simons Travel Grant. The second named author's research was partially supported by the NSF grant DMS-2302196.  
The third named author's  research was partially funded by the French National Research Agency (ANR) under the
project ANR-25-CE40-7869-01.
\vspace{3pt}

\section{Background and preliminary results}\label{sec:background}

\subsection{Review of period rings} \label{subsec:PeriodRings}

In what follows, let $k$ be a perfect field of characteristic $p>0$, and denote by $W=W(k)$ the ring of Witt vectors of $k$. Let $K$ be a finite totally ramified extension of $K_0=W[1/p]$ of degree $e \geq 1$. Recall that there is an isomorphism $\mathcal{O}_K\simeq W[u]/(E(u))$, where $E(u)$ is an Eisenstein polynomial of degree $e$, induced by the evaluation homomorphism $\ev_\pi:W[u]\to\mathcal{O}_K$, where $\pi$ is a fixed uniformizer of $K$. 

\subsubsection{The rings $\Es$ and $\ainf$}

Setting $\Es=W[[u]]$, $\ev_\pi$ extends to a map $\ev_{\pi}:\Es \rightarrow \oh_K$, and likewise, we have the induced isomorphism $\Es/(E(u))\simeq \oh_K$. We further make a choice of a fixed completed algebraic closure $\CC_K=\widehat{\overline{K}}$ of $K$, and of a system $(\pi_s)_s$ of compatible $p^s$-th roots of $\pi$ (taken from $\CC_K$). We let $K_{\infty}$ denote the field $\bigcup_s K(\pi_s)$ and consider the closed subgroup $G_{\infty}=\Gal(\overline{K}/K_{\infty})$ of the Galois group $G_K=\Gal(\overline{K}/K)$.

Let $\ainf$ denote Fontaine's infinitesimal period ring, $\ainf=W(\oh_{\CC_K}^{\flat})$ where $\oh_{\CC_K}^{\flat}$ is the tilt of $\oh_{\CC_K},$  
$$\oh_{\CC_K}^{\flat}=\varprojlim_{x\mapsto x^p} \oh_{\CC_K}/p.$$ 
$\oh_{\CC_K}^{\flat}$ is a rank one valuation ring, and we let $\CC_K^{\flat}$ denote its fraction field. The inclusion $\oh_{\CC_K}^{\flat}\rightarrow \CC_K^{\flat}$ induces the inclusion of Witt vector rings $\ainf \rightarrow W(\CC_K^{\flat})$. Both rings are naturally endowed with a $G_K$-action induced from the action on $\oh_{\CC_K}$ by functoriality.

The choice of the system $(\pi_s)_s$ outlined above determines the element $\underline{\pi}=(\pi_s)_s \in \oh_{\CC_K}^{\flat}$, hence the element $[\underline{\pi}]\in \ainf $ where $[\cdot ]$ denotes the Teichm\"{u}ller lift. We fix the inclusion of rings $\Es\rightarrow \ainf$ given by $u \mapsto [\underline{\pi}]$. The image of $\Es$ is a subring of $\ainf$ fixed by the $G_{\infty}$-action. 
 The ring maps $\Es \rightarrow \ainf$ and $\Es \rightarrow W(\CC_K^{\flat})$ are both flat, the former being faithfully flat. This follows from the fact that all three rings are $p$-adically complete and $p$-torsion free, $\Es$ is noetherian, and reducing modulo $p$ yields the faithfully flat map $k[[u]]\rightarrow \oh_{\CC_K}^{\flat}$ and the flat map $k[[u]]\rightarrow \CC_K^{\flat}$, resp., sending $u$ to $\underline{\pi}\neq 0$. 
 The ring $\ainf$ comes equipped with the map $\theta: \ainf \rightarrow \oh_{\CC_K}$ determined by $[t] \mapsto t^{\sharp}$ where $(\cdot)^{\sharp}: \oh_{\CC_K}^{\flat} \rightarrow \oh_{\CC_K}$ is the unique multiplicative map lifting the projection $\oh_{\CC_K}^{\flat}=\varprojlim_{x\mapsto x^p} \oh_{\CC_K}/p \stackrel{\mathrm{pr}_0}{\longrightarrow} \oh_{\CC_K}/p$. The kernel of $\theta$ is the principal ideal $(\xi)$, where $\xi \in \ainf$ can be taken as $E([\underline{\pi}])$. That is, modulo $E(u),$ the map $\Es \rightarrow \ainf$ becomes the inclusion $\oh_K \subseteq \oh_{\CC_K}$. 

The map $\theta: \ainf \rightarrow \oh_{\CC_K}$ is easily seen to be $G_K$-equivariant, and consequently, the kernel $(\xi) \subseteq \ainf$ is stable under the $G_K$-action. Similarly, the ideal $([\underline{\pi}])$ is $G_K$-stable --- in fact, one easily checks that for all $g \in G_K$, $g [\underline{\pi}]$ is equal to $[\underline{\pi}]$ times a unit; we denote this unit by $\epsilon(g),$ that is, we have $g [\underline{\pi}]=\epsilon(g)[\underline{\pi}]$ for all $g \in G_K$.

\subsubsection{The rings $S$, $A_{\cris}$ and $\widehat{A}_{\st}$}

We recall the definition of the ring $S=W[[u]]^{\PD}$ of divided powers of Faltings (\cite{Faltings99}) and Breuil (\cite{Breuil00, Breuil02}). The ring $S$ is the $p$-adic completion of the divided power envelope of $W[u]$ with respect to the ideal $(E(u))$, on which we have divided powers compatible with the canonical divided powers of the ideal $pW[u]$. 
The ring $S$ can be identified with the subring of $K_0[[u]]$ consisting of all power series of the form
$\displaystyle \alpha=\sum_{n=0}^\infty a_n\frac{u^n}{\lfloor n/e \rfloor !}$ with $a_n\in W$ such that the sequence $(a_n)_{n\geq 0}$ converges $p$-adically to $0$. In particular, $S$ is an $\Es$-algebra (in fact, it is the $p$-adic completion of $\Es[\{E(u)^k/k!\}_{k\geq 1}],$ the divided power envelope of $(E(u))$ in $\Es$). 

The ring $S$ is endowed with the following extra structures: 
\begin{enumerate}
    \item[(1)] A Frobenius homomorphism $\varphi: S\to S$, extending the Frobenius $\varphi$ on $\Es$. 
    \item[(2)] A descending filtration $\{\Fil^i S\}_{i\geq 0}$ of $S$-submodules, where for $i\geq 0,$ $\Fil^i S$ is the $p$-adic closure in $S$ of the ideal generated by $\left\{\displaystyle\frac{E(u)^j}{j!}, \;j\geq i\right\}$.  
    \item[(3)] A continuous $W$-linear derivation $N:S\to S$ such that $N(u)=-u$.  The derivation satisfies $N\varphi=p\varphi N$. Moreover, $N(\Fil^i S)\subset \Fil^{i-1} S$, for all $i\geq 1$. 
\end{enumerate}  
For every $0\leq i\leq p-1$, we have $\varphi(\Fil^i S)\subseteq p^i S$, which allows us to set $\displaystyle\varphi_i:=\frac{\varphi|_{\Fil^i S}}{p^i}$.

In parallel, the crystalline period ring $A_{\cris}$ is defined as the $p$-adic completion of the divided power envelope of the ideal $(\xi)$ in $\ainf$. Comparing universal properties, we have a map $S \rightarrow A_{\cris}, $ determined by $u \mapsto [\underline{\pi}]$ as before. Similarly to the ring $S,$ we have the concrete description
$$A_{\cris}=\left\{ \sum_{n=0}^{\infty} a_n \frac{\xi^n}{n!}\;\Big|\; a_n \in \ainf, a_n \rightarrow 0 \;p\text{-adically}\right\}\,.$$ The ring $A_{\cris}$ is endowed with the following structures:  

\begin{enumerate}
    \item A Frobenius map $\varphi: A_{\cris}\to A_{\cris}$, extending the Witt vector Frobenius on $\ainf$ (there is a unique way to do this, namely, we have $\varphi(\xi^k/k!)=\varphi(\xi)^k/k!,$ which is then extended continuously).
    \item A filtration $\{\Fil^iA_{\cris}\}_{i\geq 0}$ defined as
    \[\Fil^i A_{\cris}:=p\text{-adic closure of}\left(\left\{\frac{\xi^j}{j!}\;\Big|\; j\geq i\right\}\right) \subseteq A_{\cris}.\]
    \item A $G_K$-action extending the action on $A_{\inf}$ (note that this makes sense, this e.g. follows from the fact that the ideal $(\xi)\subseteq \ainf$ is $G_K$-equivariant). 
\end{enumerate}

Let $A_{\cris}\langle x\rangle=A_{\cris}[\{x^k/k!\}_{k\geq 0}]$ be the divided power envelope of the polynomial algebra $A_{\cris}[x]$ compatible with the canonical divided powers of the ideal $pA_{\cris}[x]$. The ring $\widehat{A}_{\st}$, defined by Kato in \cite[\S 3]{Kato94}, is the $p$-adic completion of $A_{\cris}\langle x\rangle$. It is endowed with the following structures: 
\begin{enumerate}
    \item A filtration $\{\Fil^i\widehat{A}_{\st}\}_{i\geq 0}$ defined as
    \[\Fil^i\widehat{A}_{\st}:=\left\{\sum_{n=0}^\infty a_n\frac{x^n}{n!}:\; a_n\in A_{\cris}, \lim\limits_{n\to\infty} a_n=0,\forall n\leq i, a_n\in\Fil^{i-n} A_{\cris}\right\}.\]
    \item A Frobenius $\varphi: \widehat{A}_{\st}\to\widehat{A}_{\st}$, extending that of $A_{\cris}$ such that $\varphi(x)=(1+x)^p-1$. 
    \item A $G_K$-action extending the action on $A_{\cris}$, determined by the formula $$g\cdot x=\epsilon(g)(x+1)-1\,,\;\; g \in G_K\,.$$ 
    \item \label{Monodromy_operatorAst}A $p$-adically continuous  $A_{\cris}$-derivation $N\colon\widehat{A}_{\st}\to\widehat{A}_{\st}$ such that $N(x)=x+1$.      
\end{enumerate}

We put on $\widehat{A}_{\st}$ an $S$-algebra structure by the injection $S\hookrightarrow \widehat{A}_{\st}$, $u\mapsto[\underline{\pi}](1+x)^{-1}$. By \cite[\S 4.2]{Breuil97}, this identifies $S$ with $(\widehat{A}_{\st})^{G_K}$ and it is compatible with all other structures. 

Consider an integer $i$ with $0 \leq i \leq p-1$ and let $A$ be $A_{\cris}$ or $\widehat{A}_{\st}$. Just like in the case of the ring $S$, we have $\varphi(\Fil^iA) \subseteq p^iA$, and we can therefore set $\displaystyle\varphi_i=\frac{\varphi|_{\Fil^iA}}{p^i}$.

While we have a map $A_{\cris}\rightarrow \widehat{A}_{\st}$, we warn the reader that the $S$-algebra structure obtained from $S\rightarrow A_{\cris}\rightarrow \widehat{A}_{\st}$ is different from the one we consider: the two $S$-algebra structures are related, but in the opposite direction. Namely, composing the map $S \rightarrow \widehat{A}_{\st}$ with the quotient map $q:\widehat{A}_{\st} \rightarrow A_{\cris}$ determined by $q(x)=0$, one recovers the map $S \rightarrow A_{\cris}$ given previously. 

\begin{remark} 
The derivation $N$ in \eqref{Monodromy_operatorAst} above is usually called the monodromy operator. Its datum is equivalent to a $p$-adically continuous connection $\nabla:\widehat{A}_{\st}\to\widehat{A}_{\st}\otimes_S\Omega^1_{\log}$, satisfying $\nabla(A_{\cris})=0$ and $\nabla(x)=-(1+x)\otimes u^{-1}du$. 
Here we denoted by $\Omega^1_{\log}$ the module of continuous logarithmic differential $1$-forms of $S$ over $W$ and $d\colon S \to \Omega^1_{\log}$ the canonical differential, see \cite[2.1.3]{Iovita/Marmora}. 
The module  $\Omega^1_{\log}$ is a free $S$-module of rank $1$ and choosing the base $-u^{-1}du$ allows to pass from the connection to the monodromy operator and conversely : we have $\nabla(f) = -N(f) \otimes u^{-1}du$. 

Same considerations apply also to the ring $S$ and its monodromy operator $N$, and the canonical differential $d: S \rightarrow \Omega^1_{\log}$. 
\end{remark}

\subsection{Kisin modules and functors to $\ZZ_p$-lattices}

We now introduce several categories of Breuil--Kisin modules. From now on, we refer to such modules as ``Kisin modules'' only, in order to distinguish them clearly from strongly divisible modules of Breuil (discussed in \S\ref{subsec:BreuilModules}).

\begin{definition}
A \textit{Kisin module} is a finitely generated $\Es$-module $\Em$ endowed with an isomorphism $$\varphi_{\Em, \lin}: (\varphi^* \Em)[1/E(u)] \rightarrow \Em[1/E(u)].$$ 

A morphism of Kisin modules $\Em \rightarrow \Em'$ is an $\Es$-linear map $f:\Em \rightarrow \Em'$ compatible with $\varphi_{\Em, \lin}$ and $\varphi_{\Em', \lin}$, i.e., such that the following diagram is commutative:
\begin{center}
\begin{tikzcd}
(\varphi^* \Em)[1/E(u)] \ar[r, "\varphi_{\Em, \lin}"] \ar[d, "\varphi^*(f)"] & \Em[1/E(u)] \ar[d, "f"] \\
 (\varphi^* \Em')[1/E(u)] \ar[r, "\varphi_{\Em', \lin}"] & \Em'[1/E(u)]
\end{tikzcd}
\end{center}
Denote by $\Mod_{\Es}^{\varphi}$ the category of all Kisin modules. Then $\Mod_{\Es, \fr}^{\varphi}$ denotes the category of \textit{free Kisin modules}, i.e., the full subcategory of $\Mod_{\Es}^{\varphi}$ consisting of all Kisin modules $\Em$ where $\Em$ is finite free as an $\Es$-module.
\end{definition}

While free Kisin modules are the classical notion (as defined e.g. in \cite{KisinShimura}), the more general category (considered, for example, in \cite{BMS1}) has the advantage of being abelian, with kernels, cokernels etc. inherited from the category of $\Es$-modules. Moreover, the category $\Mod_{\Es}^{\varphi}$  possesses the structure of abelian tensor category, which is well-known but rarely described explicitly (see e.g. \cite{LevinWangErickson} for a variant). Concretely, $\otimes$ and internal $\Hom$'s are given as follows. 

\begin{definition}
Given two Kisin modules $\Em_1, \Em_2,$ the tensor product $\Em_1 \otimes \Em_2$ is the Kisin module given as follows: 
\begin{enumerate}
\item The underlying $\Es$-module is $\Em_1 \otimes_{\Es}\Em_2.$ 
\item Under the canonical identification $$\varphi^*(\Em_1 \otimes_{\Es} \Em_2)[1/E(u)]=\varphi^*(\Em_1)[1/E(u)]\otimes_{\Es[1/E(u)]}\varphi^*(\Em_2)[1/E(u)],$$ $\varphi_{\Em_1\otimes \Em_2, \lin}$ is given by $\varphi_{\Em_1, \lin}\otimes \varphi_{\Em_2, \lin}$.
\end{enumerate}

Given two Kisin modules $\Em_1, \Em_2$, the internal Hom $\underline{\Hom}(\Em_1, \Em_2)$ is the Kisin module defined as follows:
\begin{enumerate}
\item The underlying $\Es$-module is $\Hom_{\Es}(\Em_1, \Em_2)$ (the abstract $\Es$-module homomorphisms). 
\item Under the canonical identification (given by flat base change) $$\varphi^*\Hom_{\Es}(\Em_1 , \Em_2)[1/E(u)]=\Hom_{\Es[1/E(u)]}(\varphi^*(\Em_1)[1/E(u)],\varphi^*(\Em_2)[1/E(u)]),$$ $\varphi_{\underline{\Hom}(\Em_1, \Em_2), \lin}$ is given by $$\varphi_{\underline{\Hom}(\Em_1, \Em_2), \lin}(g)= \varphi_{\Em_2, \lin}\circ g \circ \varphi_{\Em_1, \lin}^{-1}.$$
\end{enumerate}
In particular, for a Kisin module $\Em$, define the dual $\Em^{\vee}$ as $\underline{\Hom}(\Em, \Es),$ where $\Es$ is the Kisin module given by the ring $\Es$ and $\varphi_{\Es, \lin}$ being the linearization of $\varphi$ on $\Es$.
 \end{definition}
It is clear that if $\Em_1, \Em_2$ are free Kisin modules, then so are the modules  $\Em_1\otimes\Em_2$ and $\underline{\Hom}(\Em_1,\Em_2)$. 

Given a free Kisin module $\Em$, there is a natural, $\varphi$-semilinear map \[\Em \rightarrow \varphi^* \Em=\Es \otimes_{\Es, \varphi} \Em\] given by $m \mapsto 1 \otimes m. $ Composing with $\varphi_{\Em, \lin},$ one obtains the map $\varphi_{\Em}: \Em \rightarrow \Em[1/E(u)]$. We call $\Em$ \textit{of finite height} (or \textit{effective}) if $\varphi_{\Em}(\Em)\subseteq \Em$, i.e., if this map comes from a $\varphi$-semilinear map $\varphi_{\Em}: \Em \rightarrow \Em$. In this case, it is easy to see that $\varphi_{\Em, \lin}$ is the map obtained from $\varphi_{\Em}$ by linearization followed by inverting $E(u)$. Denote by $\Mod_{\Es, \fr}^{\varphi, \geq 0}$ the category of all free Kisin modules of finite height.

The central category of our interest is the following category.

\begin{definition}
Given $r \in \NN$, a free effective Kisin module $\Em$ is \textit{of height $\leq r$} if the $\Es$-submodule of $\Em$ generated by $\varphi_{\Em}(\Em)$ contains $E(u)^r\Em$. Equivalently, the linearization $\varphi_{\Em, \lin}=1 \otimes \varphi_{\Em}: \Es\otimes_{\Es, \varphi} \Em \rightarrow \Em$ (is injective and) has cokernel annihilated by $E(u)^r$. Denote the category of all free Kisin modules of height $\leq r$ by $\Mod_{\Es, \fr}^{\varphi, \leq r}$ (it would make sense to include also the superscript ``$\geq 0$'', but we leave it out in order to ease notation).  
\end{definition}

\begin{remark}\label{rem:Kisinexactcat}
The full subcategories $\Mod_{\Es, \fr}^{\varphi}, \Mod_{\Es, \fr}^{\varphi, \geq 0}$ and $\Mod_{\Es, \fr}^{\varphi, \leq r}$ of $\Mod_{\Es}^{\varphi}$ are closed under extensions. This is clear in the case of $\Mod_{\Es, \fr}^{\varphi}$ and $ \Mod_{\Es, \fr}^{\varphi, \geq 0}$. For $\Mod_{\Es, \fr}^{\varphi, \leq r}$, one needs to check that for a short exact sequence $0 \to \Em_1 \to \Em_2 \to \Em_3 \to 0$ in $\Mod_{\Es, \fr}^{\varphi, \geq 0},$ $\Em_2$ is of height $\leq r$ if $\Em_1, \Em_3$ are of height $\leq r$. 

To see this, given $\Em \in \Mod_{\Es, \fr}^{\varphi, \geq 0},$ consider the \textit{Nygaard filtration} on $\varphi^*\Em,$ $$\Fil^k \varphi^*\Em=\{x \in \varphi^*\Em\;|\; \varphi_{\Em, \lin}(x)\in E(u)^k\Em\}, \;\; k \in \ZZ.$$
One can now consider the filtered $K$-vector space $D_{\dr}(\Em)=(\varphi^*\Em/E(u)\varphi^*\Em)[1/p]$ and finally, the graded $K$-vector space $D_{\HT}(\Em)=\bigoplus_{k}\gr^k D_{\dr}(\Em)$. It can now be easily verified that $D_{\HT}$ is an exact functor into graded $K$-vector spaces, and that $\Em$ being of height $\leq r$ translates into the fact that $D_{\HT}(\Em)_k=0$ for all $k \geq r+1$. This proves the claim. 

Consequently, the categories $\Mod_{\Es, \fr}^{\varphi}, \Mod_{\Es, \fr}^{\varphi, \geq 0}$ and $\Mod_{\Es, \fr}^{\varphi, \leq r}$ are exact in the sense of Quillen (see \cite[\S A.3]{Positselski}
for the Axioms of an exact category). This follows from the characterization of exact categories as full subcategories of abelian categories closed under extension (see \cite[\S A.6]{Positselski}).
\end{remark}

\begin{definition}
Denote by $\Rep_{\ZZ_p}(G_{\infty})$ the category of finitely generated $\ZZ_p$-modules endowed with a continuous action of $G_{\infty}$. The \textit{\'{e}tale realization functor} is the functor $T_{\Es}: \Mod_{\Es}^{\varphi} \rightarrow \Rep_{\ZZ_p}(G_{\infty}) $ given by 
\[\Em \mapsto (\Em\otimes_{\Es}W(\CC_K^{\flat}))^{\varphi=1},\]
where the $\varphi$-operator on $\Em\otimes_{\Es}W(\CC_K^{\flat})$ to take invariants of is $\varphi_{\Em}\otimes \varphi_{W(\CC_K^{\flat})}.$ The $G_{\infty}$-action is induced by the action on the $W(\CC_K^{\flat})$-component.
\end{definition}

In order to work effectively with $T_{\Es}$, the following auxiliary category is useful. Denote by $\Phi \Mod_{W(\CC_K^{\flat})}$ the category of \textit{\'{e}tale $\varphi$-modules over $W(\CC_K^{\flat}),$} i.e., the category of all finitely generated $W(\CC_K^\flat)$-modules $M$ endowed with an isomorphism $\varphi_{M,\lin}: \varphi^* M \rightarrow M$ or, equivalently, with a $\varphi$-semilinear bijective map $\varphi_{M}: M \rightarrow M$. Then it is well--known (e.g. by \cite[Proposition~4.1.1]{Katz} via d\'{e}vissage) that  $\Phi \Mod_{W(\CC^{\flat})}$ is equivalent to the category $\mathsf{mod}_{\ZZ_p}$ of all finitely generated $\ZZ_p$-modules via the functor $T_{\Phi}:M\mapsto M^{\varphi_M=1},$ with the quasi-inverse given by $T \mapsto M:=T \otimes_{\ZZ_p} W(\CC_K^{\flat})$ (where $\varphi_M=1\otimes \varphi_{W(\CC_K^{\flat})}$). As any equivalences between abelian categories, both these functors are exact. Disregarding the $G_{\infty}$-action, the \'{e}tale realization functor $T_{\Es}$ may be treated as the composite $T_{\Phi} \circ \Phi$ where $\Phi: \Mod_{\Es}^{\varphi} \rightarrow \Phi \Mod_{W(\CC_K^{\flat})}$ is the base-change functor $\Em \mapsto M:=\Em\otimes_{\Es}W(\CC_K^{\flat})$ (with $\varphi_M=\varphi_{\Em}\otimes \varphi_{W(\CC_K^{\flat})}$).

Letting $\oh_{\mathcal{E}}$ denote the $p$-adic completion of $\Es[1/u],$ the map $\Es \rightarrow W(\CC_K^{\flat})$ induces the map $\oh_{\mathcal{E}}\rightarrow W(\CC_K^{\flat})$ which is easily seen to be faithfully flat (both are DVR's with uniformizer $p$). \label{faithful_flat OhE_to_WWCC}

In the literature, several contravariant versions of $T_{\Es}$ are sometimes used. To clarify their relationship, we need a bit more notation. Denote by $\widehat{\oh_{\mathcal{E}}^{\mathrm{ur}}}$ the $p$-adic closure of the integers of a maximal unramified extension of $\oh_{\mathcal{E}}$ in $W(\CC_K^{\flat})[1/p]$, and let $\Es^{\mathrm{ur}}$ denote the intersection $\widehat{\oh_{\mathcal{E}}^{\mathrm{ur}}} \cap \ainf$. It can be shown that $\Es^{\mathrm{ur}}$ is stable under $\varphi$ of $\ainf$. Then we have the following compatibilities. These are well-known but we do not know of an explicit reference, so we include a brief proof.

\begin{lemma}\label{lem:KisinDualT}
For $\Em \in \Mod_{\Es, \fr}^{\varphi, \geq 0},$ one has the following natural isomorphisms of $\ZZ_p[G_{\infty}]$-modules:
$$\hom_{\Es, \varphi}(\Em, \Es^{\mathrm{ur}}) \simeq \hom_{\Es, \varphi}(\Em, \ainf) \simeq \hom_{\Es, \varphi}(\Em, W(\CC_K^{\flat})) \simeq T_{\Es}(\Em)^{\vee} .$$
Here $\hom_{\Es, \varphi}$ denotes the group of homomorphisms of $\Es$-modules $\Em$ endowed with a $\varphi$-semilinear operator $\varphi_{\Em}: \Em \rightarrow \Em$, and $\vee$ denotes taking the $\ZZ_p$-linear dual representation.
\end{lemma}   

\begin{proof}
One clearly has the inclusions $$\hom_{\Es, \varphi}(\Em, \Es^{\mathrm{ur}}) \subseteq \hom_{\Es, \varphi}(\Em, \ainf) \subseteq \hom_{\Es, \varphi}(\Em, W(\CC_K^{\flat})).$$ On the other hand, by \cite[Proposition~1.8.3]{Fontaine90}, the image of $\Em$ in $W(\CC_K^{\flat})$ under a $\Es$-linear map compatible with $\varphi$ always lands in $\Es^{\mathrm{ur}}$, so the inclusions are equalities.

To justify the last isomorphism in the Lemma, note that one has a natural isomorphism $\hom_{\Es, \varphi}(\Em, W(\CC_K^{\flat}))\simeq \hom_{W(\CC_K^{\flat}), \varphi}(\Em\otimes_{\Es}W(\CC_K^{\flat}), W(\CC_K^{\flat}))$ by a $\varphi$-semilinear version of Hom-tensor adjunction. But the latter is just the Hom group in the category $\Phi \Mod_{W(\CC_K^{\flat})}$, and by the above-mentioned equivalence $\Phi \Mod_{W(\CC_K^{\flat})} \stackrel{\sim} \rightarrow \mod_{\ZZ_p}$, it is thus naturally isomorphic to $$\Hom_{\ZZ_p}((\Em\otimes_{\Es}W(\CC_K^{\flat}))^{\varphi=1},W(\CC_K^{\flat})^{\varphi=1})=\Hom_{\ZZ_p}(T_{\Es}(\Em),\ZZ_p)=T_{\Es}(\Em)^{\vee}.$$ 
It is now easy to verify that all the listed isomorphisms are compatible with the natural $G_{\infty}$-actions on the respective $\ZZ_p$-modules.
\end{proof}

The following properties of the functor $T_{\Es}$ are also well-known, but since the proofs are either not written down explicitly, or written for a related but distinct variant (e.g., contravariant functors etc.), we include the proof for reader's convenience.

\begin{proposition}\label{prop:PropertiesOfT}
\begin{enumerate}
\item The functor $T_{\Es}$ is exact.
\item $T_{\Es}$ is compatible with $\otimes$ and $\underline{\Hom}$.
\item $T_{\Es}$ restricts to a functor $T_{\Es}: \Mod_{\Es, \fr}^{\varphi} \rightarrow \Rep_{\ZZ_p, \fr}(G_{\infty})$, where $\Rep_{\ZZ_p, \fr}(G_{\infty})$ is the full subcategory of $\Rep_{\ZZ_p}(G_{\infty})$ consisting of all objects whose underlying $\ZZ_p$-module is finite free. Moreover, $T_{\Es}$ is fully faithful on $\Mod_{\Es, \fr}^{\varphi}$.
\end{enumerate}
\end{proposition}

\begin{proof}[Proof of Proposition~\ref{prop:PropertiesOfT}]
By the above discussion,  both $T_{\Phi}$ and $\Phi$ are exact, so $T_{\Es}$ is also exact; this proves (1). 

To prove (2), let us prove compatibility with $\otimes$ only, since $\underline{\Hom}$ is very similar. Note that we have a natural isomorphism \begin{align*}
\Phi(\Em_1 \otimes \Em_2)&=(\Em_1 \otimes_{\Es} \Em_2)\otimes_{\Es}W(\CC_K^{\flat}) \\ &\simeq (\Em_1 \otimes_{\Es}W(\CC_K^{\flat}))\otimes_{W(\CC_K^{\flat})} (\Em_2\otimes_{\Es}W(\CC_K^{\flat}))=\Phi(\Em_1)\otimes_{W(\CC_K^{\flat})}\Phi(\Em_1),
\end{align*}
compatible with the $\varphi$-operators as well as with the $G_{\infty}$-action. By the previous discussion, we may identify $\Phi(\Em_i)=T_i \otimes_{\ZZ_p}W(\CC_K^{\flat})$ as $\varphi$-modules, where $T_i=T_{\Es}(\Em_i)$ for $i=1,2$. This, in turn, is easily seen to be canonically isomorphic to $(T_1 \otimes_{\ZZ_p} T_2)\otimes_{\ZZ_p}{W(\CC_K^{\flat})},$ again in a manner that is compatible with $\varphi$'s and with $G_{\infty}$-action. Then we have 
$$T_{\Es}(\Em_1\otimes \Em_2)\simeq \left((T_1 \otimes_{\ZZ_p} T_2)\otimes_{\ZZ_p}{W(\CC_K^{\flat})}\right)^{\varphi=1}=T_1 \otimes_{\ZZ_p} T_2,$$
and that these isomorphisms are compatible with the natural $G_{\infty}$-actions on both sides.

When $\Em$ is a free Kisin module, then $\Phi(\Em)$ is a finite free $W(\CC_K^{\flat})$-module. Since it is also isomorphic  to $T_{\Es}\otimes_{\ZZ_p} W(\CC_K^{\flat})$ for $T=T_{\Es}(\Em)$, it follows that $T_{\Es}(\Em)$ is finite free as a $\ZZ_p$-module, proving the freeness part of (3).

The full faithfullness statement in (3) is implicit in \cite{Kisin1}, stated more explicitly in \cite[Lemma 3.4.5]{Liu08}, except for effective Kisin modules only and for the contravariant functor $\hom_{\Es, \varphi}(-, \Es^{\mathrm{ur}}).$ Lemma~\ref{lem:KisinDualT} then implies the full faithfulness of the covariant functor $T_{\Es}$ (since $T \mapsto T^{\vee}$ is an auto-equivalence of $\Rep_{\ZZ_p, \fr}(G_{\infty})$). As will be discussed in \S~\ref{subsec:KisinToBreuil}, there is a natural notion of Breuil--Kisin twist $\Em\mapsto \Em\{s\}$ for $s \in \mathbb{Z}$, such that $T_{\Es}(\Em\{s\})\simeq T_{\Es}(\Em)\otimes \ZZ_p(s).$ Since for a pair of free Kisin modules, $\Em_1, \Em_2,$ one can always find $s$ big enough such that $\Em_1\{-s\}, \Em_2\{-s\}$ are both effective, the result extends from $\Mod_{\Es, \fr}^{\varphi, \geq 0}$ to all of $\Mod_{\Es, \fr}^{\varphi}$. 
\end{proof}

The following proposition and lemma will be important for our computations related to exactness properties of the Kisin functor later on.

\begin{proposition}\label{prop:KerOfTisUTorsion}
For every $\Em \in \Mod_{\Es}^{\varphi},$ $T_{\Es}(\Em)=0$ if and only if $\Em$ is $u^{\infty}$-torsion, i.e., $u^N \Em=0$ for big enough $N$.
\end{proposition}

\begin{proof}
Since the map $\oh_{\mathcal{E}} \rightarrow W(\CC_K^{\flat})$ is faithful flat (see \S\ref{faithful_flat OhE_to_WWCC}),
a Kisin module $\Em$ satisfies $T_{\Es}(\Em)=0$ if and only if $\Em\otimes_{\Es} \oh_{\mathcal{E}}=0$;
that is, if and only if the $p$-adic completion $\Em[1/u]^{\wedge}$ of $\Em[1/u]$ is $0$. This clearly happens when $u^N \Em=0$ for some $N$. To prove the converse, fix a Kisin module $\Em$. By \cite[Proposition 4.3]{BMS1}, there is an exact sequence of Kisin modules 
\begin{center}
\begin{tikzcd}
0 \ar[r] & \Em_{\mathrm{tor}} \ar[r] & \Em \ar[r] & \Em_{\fr} \ar[r] & \overline{\Em} \ar[r] & 0
\end{tikzcd}
\end{center}
where the module $\Em_{\fr}$ is free, $\Em_{\mathrm{tor}}$ is annihilated by a power of $p$, and the module $\overline{\Em}$ is annihilated by a power of $(p, u)$. In particular, $\overline{\Em}$ vanishes after inverting $u$, and the sequence 
\begin{center}
\begin{tikzcd}
0 \ar[r] & \Em_{\mathrm{tor}}[1/u] \ar[r] & \Em[1/u] \ar[r] & \Em_{\fr}[1/u] \ar[r] & 0
\end{tikzcd}
\end{center}
is therefore exact. $\Em_{\fr}[1/u]$ is a free $\Es[1/u]$-module and we therefore have $\Em[1/u]\simeq \Em_{\mathrm{tor}}[1/u] \oplus \Em_{\fr}[1/u]$ as $\Es[1/u]$-modules. Consequently, we have  $\Em[1/u]^{\wedge}\simeq \Em_{\mathrm{tor}}[1/u] \oplus \Em_{\fr}[1/u]^{\wedge}$ (the first summand is $p^k$-torsion for some $k$, hence $p$-complete).

Now, if $\Em[1/u]^{\wedge}=0,$ then we have both $\Em_{\mathrm{tor}}[1/u]=0$ and $\Em_{\fr}[1/u]^{\wedge}=0$, hence also $\Em_{\fr}[1/u]=0$ (since $\Em_{\fr}[1/u]$ is a free $\Es[1/u]$-module). Consequently, we have $\Em[1/u]=0$ and thus, $\Em$ is annihilated by $u^N$ for big enough $N$.
\end{proof}

\begin{lemma}\label{lem:UtorIsPtor}
Let $\Em$ be a Kisin module. If $u^N \Em=0$ for some $N \in \NN$, then $\Em$ is annihilated by a power of $(p, u)$.  
\end{lemma}

\begin{proof}
This is, again, a consequence of the canonical exact sequence 
\begin{center}
\begin{tikzcd}
0 \ar[r] & \Em_{\mathrm{tor}} \ar[r] & \Em \ar[r] & \Em_{\fr} \ar[r] & \overline{\Em} \ar[r] & 0
\end{tikzcd}
\end{center}
from \cite[Proposition 4.3]{BMS1}. If $u^N\Em=0$, then the map $\Em\rightarrow \Em_{\fr}$ is zero, and therefore $\Em \simeq \Em_{\mathrm{tor}}$. Since the latter is annihilated by a power of $p$, the claim follows.
\end{proof}

Full faithfulness of $T_{\Es}$ allows us to consider the following.

\begin{definition}
Denote by $\Rep_{\ZZ_p, \fr}^{\BK}(G_\infty)$ the essential image of the functor $T_{\Es}$ restricted to $\Mod_{\Es, \fr}^{\varphi}$. Since $T_{\Es}$ is fully faithful, it admits a quasi-inverse $\underline{\Em}: \Rep_{\ZZ_p, \fr}^{\BK}(G_\infty) \rightarrow \Mod_{\Es, \fr}^{\varphi}$. We call $\underline{\Em}$ the \textit{Kisin functor}.
\end{definition}
  
For an integer $r \geq 0$, denote by $\Rep_{\ZZ_p, \cris}^{[-r, 0]}(G_K)$ ($\Rep_{\ZZ_p,  \st}^{[-r, 0]}(G_K),$ resp.) the category 
of finite free $\ZZ_p$-modules $T$ endowed with a continuous $G_K$-action such that $T[1/p]$ is a crystalline (semistable, resp.) representation of $G_K$ with Hodge-Tate weights in $[-r, 0]$. The next proposition, stating various properties of $\underline{\Em}$, summarizes some of the main results of \cite{Kisin1}. 

\begin{proposition}\label{prop:KisinBasicProperties}
\begin{enumerate}
\item The restriction functor $\Rep_{\ZZ_p,  \cris}^{[-r, 0]}(G_K) \rightarrow \Rep_{\ZZ_p, \fr}(G_{\infty})$ given by $ \; T \mapsto T|_{G_{\infty}}$, is fully faithful.
\item The restriction functor $\Rep_{\ZZ_p, \st}^{[-r, 0]}(G_K) \rightarrow \Rep_{\ZZ_p, \fr}(G_{\infty})$ lands in $\Rep_{\ZZ_p, \fr}^{\BK}(G_{\infty}).$ Consequently, it makes sense to consider the composite functor (denoted again by $\underline{\Em}$ by abusing notation) $\Rep_{\ZZ_p, \st}^{[-r, 0]}(G_K) \stackrel{\mathrm{res}}\rightarrow \Rep_{\ZZ_p, \fr}^{\BK}(G_{\infty})\stackrel{\underline{\Em}}\rightarrow \Mod_{\Es, \fr}^{\varphi}$. 
\item The functor $\underline{\Em}: \Rep_{\ZZ_p, \st}^{[-r, 0]}(G_K) \rightarrow \Mod_{\Es, \fr}^{\varphi}$ takes values in $\Mod_{\Es, \fr}^{\varphi, \leq r}$. Consequently, there is a fully faithful functor $\underline{\Em}:\Rep_{\ZZ_p, \cris}^{[-r, 0]}(G_K) \rightarrow \Mod_{\Es, \fr}^{\varphi, \leq r}$ characterized by the property that for all $T\in \Rep_{\ZZ_p, \cris}^{[-r,0]}(G_K),$ one has $$T_{\Es}(\underline{\Em}(T))=T|_{G_{\infty}} \text{ in }\Rep_{\ZZ_p, \fr}(G_{\infty}), \text{ functorially in }T.$$
\item{When $r=1$, the functor $\underline{\Em}: \Rep_{\ZZ_p, \cris}^{[-1,0]}(G_K) \rightarrow \Mod_{\Es, \fr}^{\varphi, \leq 1}$ from part (2) is an equivalence of categories, with the quasi-inverse given by $T_{\Es}$ or, more precisely, by the functor $$\Em \mapsto \text{the unique }T \in \Rep_{\ZZ_p, \cris}^{[-1, 0]}(G_K)\text{ with }T|_{G_{\infty}}=T_{\Es}(\Em).$$}
\end{enumerate}
\end{proposition}

\begin{proof}
Part (1) is Theorem~0.2 of \cite{Kisin1}. Parts (2), (3) also follow from \cite{Kisin1}, especially Theorem~0.1, except 
that the \'{e}tale realization functors (hence, also the corresponding functors $\underline{\Em}$) in \textit{loc. cit.} are contravariant. The result is then obtained by precomposing with the duality functor of $\ZZ_p[G_K]$-representations (a more direct proof of the full faithfulness of $\underline{\Em}$ on crystalline lattices in terms of covariant functors is given in \cite[Theorems 1.2, 1.3]{BhattScholze2}). 
Part (4) is a consequence of a similar equivalence between height $\leq 1$ Kisin modules and $p$-divisibe groups; this has been established as Theorem~0.4 of \cite{Kisin1} for $p>2$, and in \cite{LiuBT} for $p=2$.
\end{proof}

To gain compatibility with Breuil modules and their covariant \'{e}tale realization functor, we consider the following twisted variant.

\begin{definition}\label{def:TwistedKisinFunctors}
Fix an integer $r\in \ZZ$.

The \textit{twisted \'{e}tale realization functor} $T_{\Es}^{(r)}: \Mod_{\Es}^{\varphi} \rightarrow \Rep_{\ZZ_p}(G_{\infty})$ is defined as $$T_{\Es}^{(r)}(\Em)=T_{\Es}(\Em)(r), \;\;\Em \in \Mod_{\Es}^{\varphi}.$$
Here $(r)$ denotes the Tate twist $-\otimes_{\ZZ_p}(\ZZ_p(r)|_{G_\infty})$ on $\ZZ_p[G_{\infty}]$-modules.

The \textit{twisted Kisin functor} $\underline{\Em}^{(r)}: \Rep_{\ZZ_p, \fr}^{\BK}(G_{\infty})\rightarrow \Mod_{\Es, \fr}^{\varphi}$ is defined as
$$\underline{\Em}^{(r)}(T)=\underline{\Em}(T(-r)), \;\; T \in \Rep_{\ZZ_p, \fr}^{\BK}(G_{\infty}).$$
Here $(-r)$ stands, again, for the usual Tate twist.

Just like for $\underline{\Em},$ the functors obtained from $\underline{\Em}^{(r)}$ by precomposition with the restriction functor $\Rep_{\ZZ_p, \st}^{[-r, 0]}(G_K)\rightarrow \Rep_{\ZZ_p, \fr}^{\BK}(G_{\infty}) $ or $\Rep_{\ZZ_p,  \cris}^{[-r, 0]}(G_K) \rightarrow \Rep_{\ZZ_p, \fr}^{\BK}(G_{\infty})$ are also denoted by $\underline{\Em}^{(r)}$.
\end{definition}

As a direct consequence of the properties of $T_{\Es}$ and $\underline{\Em}$, we obtain:

\begin{proposition}\label{prop:PropertiesOfTrMr}
\begin{enumerate}
\item The functor $T_{\Es}^{(r)}$ is exact.
\item $T_{\Es}^{(r)}$ restricts to a functor $T_{\Es}^{(r)}: \Mod_{\Es, \fr}^{\varphi} \rightarrow \Rep_{\ZZ_p, \fr}(G_{\infty})$, where $\Rep_{\ZZ_p, \fr}(G_{\infty})$ is the full subcategory of $\Rep_{\ZZ_p}(G_{\infty})$ consisting of all objects whose underlying $\ZZ_p$-module is finite free.
\item For every $\Em \in \Mod_{\Es}^{\varphi},$ we have $T_{\Es}^{(r)}(\Em)=0$ if and only if $\Em$ is $u^{\infty}$-torsion, i.e., $u^N \Em=0$ for big enough $N$.
\item Assume $r \geq 0$. The restriction functor $\Rep_{\ZZ_p,  \cris}^{[0, r]}(G_K) \rightarrow \Rep_{\ZZ_p, \fr}(G_{\infty})$ is fully faithful.
\item Assume $r \geq 0$. The functor $\underline{\Em}^{(r)}:\Rep_{\ZZ_p, \cris}^{[0, r]}(G_K) \rightarrow \Mod_{\Es, \fr}^{\varphi, \leq r}$ is fully faithful, and it is characterized by the property that for all $T\in \Rep_{\ZZ_p, \cris}^{[0,r]}(G_K),$ one has $$T^{(r)}_{\Es}(\underline{\Em}^{(r)}(T))=T|_{G_{\infty}} \text{ in }\Rep_{\ZZ_p, \fr}(G_{\infty}), \text{ functorially in }T.$$
\item{The functor $\underline{\Em}^{(1)}: \Rep_{\ZZ_p, \cris}^{[0,1]}(G_K) \rightarrow \Mod_{\Es, \fr}^{\varphi, \leq 1}$ is an equivalence of categories, with the quasi-inverse given by $T^{(1)}_{\Es}$ or, more precisely, by the functor $$\Em \mapsto \text{the unique }T \in \Rep_{\ZZ_p, \cris}^{[0, 1]}(G_K)\text{ with }T|_{G_{\infty}}=T^{(1)}_{\Es}(\Em).$$}
\end{enumerate}
\end{proposition}

To close this section, we mention explicitly one of the variants of Kisin modules that enhances the pair of functors $\underline{\Em}, T_{\Es}$ to an equivalence of categories with the category of lattices in $G_K$-representations. Namely,  we consider the theory of Kisin $G_K$-modules of Gao \cite{GaoBKGK} (these are called \textit{Breuil--Kisin} $G_K$-modules in \textit{op.cit.}, but in keeping with our convention to separate Breuil--Kisin and Breuil modules clearly, we shorten the name just like we do for general Kisin modules). 
 
\begin{definition}[\cite{GaoBKGK}]
A \textit{(free) Kisin $G_K$-module of height $\leq r$} consists of a free Kisin module $\Em$ of height $\leq r$ together with the datum of a continuous, semilinear $G_K$-action on $\Em_{\inf}=\Em \otimes_{\Es}\ainf$ compatible with $ \varphi_{\Em}$,
such that 
\begin{enumerate}
\item $\Em \subseteq \Em_{\inf}^{G_{\infty}}$,
\item $\Em/u\Em \subseteq (\Em_{\inf}/W(\mathfrak{m})\Em_{\inf})^{G_K}$ via the natural map.
\end{enumerate} 
Denote the category of all free Kisin $G_K$-modules of height $\leq r$ by $\Mod_{\Es, \fr}^{\varphi, G_K, \leq r}$.
\end{definition} 
In the above definition, $W(\mathfrak{m})$ is the kernel of the map $\ainf=W(\oh_{\CC_K}^{\flat})\rightarrow W(\oh_{\CC_K}^{\flat}/\mathfrak{m}),$ where $\mathfrak{m}$ is the maximal ideal of $\oh_{\CC_K}^{\flat}$.
Given a free Kisin $G_K$-module $\Em,$ the isomorphism
$$T_{\Es}(\Em)=(\Em\otimes_{\Es}W(\CC_K^{\flat}))^{\varphi=1}\simeq (\Em_{\inf}\otimes_{\ainf}W(\CC_K^{\flat}))^{\varphi=1}$$
endows $T_{\Es}(\Em)$ with a structure of a $G_K$-module (rather than just $G_{\infty}$-module). Denote the resulting functor by $T_{\Es, G_K}.$

\begin{theorem}[{\cite[Theorem~7.1.7, Proposition~7.1.10]{GaoBKGK}}]\label{thm:BKGKCorrespondence}
\begin{enumerate}
\item The functor $T_{\Es, G_K}$ is an equivalence of categories $T_{\Es, G_K}: \Mod_{\Es, \fr}^{\varphi, G_K, \leq r} \stackrel{\sim}\rightarrow \Rep_{\ZZ_p,\st}^{[-r, 0]}(G_K). $
\item There is an ideal $I \subseteq \ainf$ such that for all $\Em \in \Mod_{\Es, \fr}^{\varphi, G_K, \leq r},$ the representation $T_{\Es, G_K}(\Em)[1/p]$ is crystalline if and only if $$\forall g \in G_K: (g-1)\Em \subseteq I\Em_{\inf}. $$
\end{enumerate}
\end{theorem}

Denote by $\Mod_{\Es, \fr, \cris}^{\varphi, G_K, \leq r}$ the full subcategory of $\Mod_{\Es, \fr}^{\varphi, G_K, \leq r}$ consisting of objects satisfying the condition in Theorem~\ref{thm:BKGKCorrespondence} (2). Composing the functor $T_{\Es, G_K}$ with $-\otimes\ZZ_p(r)$ as before, we therefore obtain the equivalences of categories 
$$T_{\Es, G_K}^{(r)}:\Mod_{\Es, \fr}^{\varphi, G_K, \leq r} \stackrel{\sim}\rightarrow \Rep_{\ZZ_p,\st}^{[0, r]}(G_K), \;\;\; T_{\Es, G_K}^{(r)}:\Mod_{\Es, \fr, \cris}^{\varphi, G_K, \leq r} \stackrel{\sim}\rightarrow \Rep_{\ZZ_p,\cris}^{[0, r]}(G_K).$$
In both cases, let us denote the quasi-inverse by $\underline{\Em}_{G_K}^{(r)}$.

\subsection{Breuil modules and functors to $\ZZ_p$-lattices}\label{subsec:BreuilModules}
In this subsection we recall the definitions of the various categories of modules due to C. Breuil. The main references for this section are \cite{Breuil99, Breuil00, Breuil02}.

\subsubsection{The various categories of modules} In what follows we fix an integer $0\leq r<p-1$. Let us denote by $c_1$ the element $\varphi_1(E(u)) \in S^{\times}$. 

\begin{definition}\label{def:ModSPrime}
    We denote by $\ModPrime(S)^{\varphi, N}_{r}$ the category (of all Breuil modules) with objects $4$-uples $(M, \Fil^r M, \varphi_r, N)$, where: 
    \begin{enumerate}
        \item[1.] $M$ is an $S$-module.
        \item[2.]   $\Fil^r M$ is a $S$-submodule of $M$ such that $\Fil^r S\cdot M\subseteq \Fil^r M$.
        \item[3.]$\varphi_r:\Fil^r M\to M$ is a $\varphi$-semilinear map such that  $\varphi_r(sm)=c_1^{-r}\varphi_r(s)\varphi_r(E(u)^r m)$, for every $s\in \Fil^r S$ and all $m\in M$. 
        \item[4.] $N:M\to M$ is a $W$-linear endomorphism satisfying: 
        \begin{enumerate}
            \item[(i)] $N(sm)=N(s)m+sN(m)$, for all $s\in S, m\in M$.
            \item[(ii)] $E(u) N(\Fil^r M)\subset\Fil^r M$. 
            \item[(iii)] The following diagram commutes:  
            \[\begin{tikzcd}
\Fil^r M \ar{r}{\varphi_r} \ar{d}{E(u)N} &  M\ar{d}{c_1N} \\
\Fil^r M\ar{r}{\varphi_r} & M.
\end{tikzcd}
        \]
        \end{enumerate}
          
    \end{enumerate}
    The morphisms in $\ModPrime(S)^{\varphi, N}_{r}$ are $S$-linear maps $f: M_1\to M_2$ such that $f(\Fil^r M_1)\subset \Fil^r M_2$ and $f$ commutes with $\varphi_r$ and $N$. We also consider the category $\ModPrime(S)^{\varphi}_{r}$ which forgets the derivation $N$ in the previous definition.
   \end{definition}

\begin{example}
From \S~\ref{subsec:PeriodRings}, we have $(A_{\cris}, \Fil^r, \varphi_r), (\widehat{A}_{\st}, \Fil^r, \varphi_r) \in \ModPrime(S)^{\varphi}_{r}$ and $(\widehat{A}_{\st}, \Fil^r, \varphi_r, N) \in  \ModPrime(S)^{\varphi, N}_{r}$.
We also consider a torsion version of these objects: denote
$$A_{\cris,\infty}=A_{\cris}\otimes_{\ZZ_p} \QQ_p/\ZZ_p,\;\;\widehat{A}_{\st,\infty}=\widehat{A}_{\st}\otimes_{\ZZ_p} \QQ_p/\ZZ_p. $$
With all the necessary structures induced by those of $A_{\cris}$ and $\widehat{A}_{\st},$ respectively, one has $A_{\cris,\infty} \in \ModPrime(S)^{\varphi}_{r}$ and $\widehat{A}_{\st,\infty} \in \ModPrime(S)^{\varphi, N}_{r}$ .
\end{example}

We recall that $\ModPrime(S)^{\varphi, N}_{r}$ (resp. $\ModPrime(S)^{\varphi}_{r}$) is an exact category. A short sequence \[0\to M'\to M\to M''\to 0\] in $\ModPrime(S)^{\varphi, N}_{r}$ is called \textit{exact} if the induced sequences $0\to \Fil^r M'\to \Fil^r M\to \Fil^r M''\to 0$ and $0\to M'\to M\to M''\to 0$ of $S$-modules are exact.

Next we recall the definition of the category of \textit{torsion Breuil modules}. 

\begin{definition}
    We denote by $\Mod\FI(S)_r^{\varphi, N}$ the full subcategory of $\ModPrime(S)_r^{\varphi, N}$ whose objects satisfy the following additional properties:
    \begin{enumerate}
        \item[1'.] As $S$-module $M$ is isomorphic to $\displaystyle\bigoplus_{i=1}^l \frac{S}{p^{n_i} S}$, where $l\geq 0$ and $n_1, \ldots, n_l$ are positive integers. 
               \item[5.] $\varphi_r(\Fil^r S)$ generates $M$ as $S$-module. 
    \end{enumerate}
\end{definition}

Lastly, we consider the category of \textit{free Breuil modules}, which are known in the literature as \textit{strongly divisible modules} of weight $\leq r$. 

\begin{definition}\label{def:StronglyDivisible}
  (a)  We denote by $\Mod(S)_{r, \log}^{\varphi, N}$ (resp.
  $\Mod(S)_{r}^{\varphi}$)  the full subcategory of $\ModPrime(S)_r^{\varphi, N}$ (resp. of 
  $\ModPrime(S)_{r}^{\varphi}$) whose objects satisfy the extra properties:
    \begin{enumerate}
        \item[1''.]{\label{BreuilItem111} $M$ is a free $S$-module of finite rank. }
        \item[5.]{\label{BreuilItem5} $\varphi_r(\Fil^r S)$ generates $M$ as $S$-module.} 
        \item[6.] $\Fil^r M\cap pM=p\Fil^r M$. Equivalently, the $S$-module $M/\Fil^r M$ has no $p$-torsion (i.e., $\Fil^r M$ is $p$-saturated in $M$). We will refer to the objects of this category as \textit{strongly divisible modules.}
    \end{enumerate}  
    (b) We denote by $\Mod(S)_{r}^{\varphi, N}$ 
    the subcategory of $\Mod(S)_{r, \log}^{\varphi, N}$ 
    with objects whose derivation $N$ satisfies the regularity condition $N(M)\subset uM$. We will refer to the objects of this category as \textit{crystalline strongly divisible modules.} 
\end{definition}

\begin{definition}\label{Def_U}
    We denote by $U: \Mod(S)_{r, \log}^{\varphi, N}\to \Mod(S)_{r}^{\varphi}$ the forgetful functor that forgets the derivation $N$.   
\end{definition}

\begin{remark}
    It follows by the $p$-saturation property that if $(M, \Fil^r M, \varphi_r, N)$ is an object of $\Mod(S)_{r,\log}^{\varphi, N}$, then for every $n\geq 1$ the module $(M_n, \Fil^r M_n, \overline{\varphi}_r, \overline{N})$ is in $\Mod\FI(S)_{r}^{\varphi, N}$, where $M_n=M/p^n M$, $\Fil^r M_n=\Fil^r M/p^n\Fil^r M$ and $\overline{\varphi}_r, \overline{N}$ are induced by $\varphi_r, N$ by going modulo $p$. In what follows we will simply denote by $M$ the objects of $\Mod(S)_{r,\log}^{\varphi, N}$, and by $M_n$ the induced $p^n$-torsion module. 
\end{remark}

\begin{example}\label{ex1} Let $0\leq r<p-1$ be an integer. 
\begin{enumerate}
    \item We will denote by $\mathbf{1}(r)$ the object $(S, \Fil^r(S), \varphi_r, N)$ of $\Mod(S)_r^{\varphi, N}$. 
    \item We will denote simply by $\mathbf{1}$ the object $(S, S, \varphi, N)$ of $\Mod(S)_r^{\varphi, N}$. 
\end{enumerate}
\end{example}

\subsubsection{Functors to Galois representations} 

\begin{definition}
    The \textbf{contravariant} \textit{\'{e}tale realization functor} is the functor
    \begin{eqnarray*}
        &&T_{\st}^\star: \Mod(S)_{r, \log}^{\varphi, N}\to \Rep_{\ZZ_p,  \st}^{[0, r]}(G_K)\\
        && T_{\st}^\star(M)=\Hom_{\ModPrime(S)_{r,\log}^{\varphi, N}}(M, \widehat{A}_{\st}). 
    \end{eqnarray*}
\end{definition}

The following is a theorem of T. Liu, establishing an earlier conjecture of C. Breuil (\cite[Conjecture 2.2.6]{Breuil02}).

\begin{theorem}\label{thm:Liu} (\cite[Theorem 2.3.5]{Liu08}, \cite[Corollary 2.3.5]{Iovita/Marmora})
    For every $0\leq r\leq p-2$ the functor $T_{\st}^\star$ is an  antiequivalence of categories. Moreover, when restricted to the category $\Mod(S)_{r}^{\varphi, N}$ of crystalline strongly divisible modules of weight $\leq r$, it induces an antiequivalence
    \[T_{\st}^\star:\Mod(S)_{r}^{\varphi, N}\xrightarrow{\sim} \Rep_{\ZZ_p, \cris}^{[0, r]}(G_K).\]
\end{theorem}

In the process of proving this, Liu showed commutativity of the following diagram. 

\begin{equation}\label{diag1}
   \begin{tikzcd}
 \Mod(S)_{r, \log}^{\varphi, N} \ar{r}{T_{\st}^\star} \ar{d}{U} & \Rep_{\ZZ_p}(G_K)_{\st}^{[0,r]} \ar{d}{\res} \\
 \Mod(S)_{r}^{\varphi} \ar{r}{T_{\cris}^\star} & \Rep_{\ZZ_p}(G_\infty)^{[0,r]}.
\end{tikzcd} 
\end{equation}
Here $\res$ is the restriction of the action functor and $T_{\cris}^\star(M)=\Hom_{\ModPrime(S)_{r}^{\varphi}}(M, A_{\cris})$. This commutativity means precisely that the homomorphism 
\begin{equation}\label{st-to-cris-covariant}
    \Hom_{\ModPrime(S)_{r}^{\varphi, N}}(M, \widehat{A}_{\st})\to \Hom_{\ModPrime(S)_{r}^{\varphi}}(M, A_{\cris}) 
\end{equation} induced by the homomorphism  $\widehat{A}_{\st}\to A_{\cris}$ sending $\gamma_i(x)$ to $0$ for $i\geq 1$ is an isomorphism of $\ZZ_p[G_\infty]$-modules. Here we abused notation and used $M$ for the Breuil module with or without monodromy. One first shows that (\ref{st-to-cris-covariant}) an isomorphism of $\ZZ_p$-modules. This is mentioned in \cite[Proof of Proposition 3.5.1]{Liu08} and it follows by an easy d\'{e}vissage argument from \cite[Lemme 2.3.1.1]{Breuil99}. In \cite[Proof of Corollary 3.5.2]{Liu08} Liu shows that in fact this is an isomorphism of $\ZZ_p[G_\infty]$-modules.

We will show next that the functor $T_{\st}^\star$ is exact. To do that we need a detour through the torsion objects. 

\begin{definition}
    Let $\Rep_{\ZZ_p, \tor}(G_K)$ be the category of torsion $\ZZ_p[G_K]$-modules.
    We consider the contravariant functors    
\begin{eqnarray*}
&&T_{\st,\infty}^\star: \Mod\FI(S)_{r}^{\varphi, N}\to \Rep_{\ZZ_p, \tor}(G_K), \;\;\;    T_{\st,\infty}^\star(M)=\Hom_{\ModPrime(S)_r^{\varphi, N}}(M, \widehat{A}_{\st,\infty})\\
&&T_{\cris,\infty}^\star: \Mod\FI(S)_{r}^{\varphi}\to \Rep_{\ZZ_p, \tor}(G_K), \;\;\; T_{\cris,\infty}^\star(M)=\Hom_{\ModPrime(S)_r^{\varphi}}(M, A_{\cris,\infty}).
\end{eqnarray*}
\end{definition} 
We summarize what is known about these functors. 
\begin{enumerate}
    \item[1.] The functor $T_{\st,\infty}^\star$ is fully faithful when $er<p-1$ (\cite[Th\'{e}or\`{e}me 2.3.1]{Caruso05}). In general, it is  not always full. 
    \item[2.] For $M\in \Mod\FI(S)_{r}^{\varphi, N}$ there  is an isomorphism  $T_{\st,\infty}^\star(M)\xrightarrow{\simeq}T_{\cris,\infty}^\star(U(M))$ as $\ZZ_p$-modules (cf.~\cite[Lemme 2.3.1.1]{Breuil99}). 
    \item[3.] The functor $T_{\cris,\infty}^\star$ is exact (\cite[Lemme 2.3.1.3]{Breuil99}). 
\end{enumerate}
Combining the last two facts, we conclude that the functor $T_{\st,\infty}^\star$ is also exact. This yields the following.

\begin{corollary}\label{cor:exactness1}
    The functor $T_{\st}^\star: \Mod(S)_{r, \log}^{\varphi, N}\to \Rep_{\ZZ_p}(G_K)_{\st}^{[0,r]}$ (and its restriction to the category of crystalline strongly divisible modules) is exact. 
\end{corollary} 

\begin{proof}
This follows by a d\'{e}vissage argument. Let $M=(M,\Fil^r M, \varphi_r, N)\in\Mod(S)_{r,\log}^{\varphi, N}$. 
Since by definition $\widehat{A}_{\st}$ and $\Fil^r \widehat{A}_{\st}$ are $p$-complete, we have an isomorphism $\widehat{A}_{\st}\simeq\varprojlim_{n}\widehat{A}_{\st}/p^n\widehat{A}_{\st}$ in $\ModPrime(S)_{r}^{\varphi, N}$. We have a sequence of isomorphisms 
\begin{eqnarray}
T_{\st}^\star(M)=&&\Hom_{\ModPrime(S)_r^{\varphi, N}}(M, \widehat{A}_{\st})=\Hom_{\ModPrime(S)_r^{\varphi, N}}(M, \varprojlim_{n}\widehat{A}_{\st}/p^n\widehat{A}_{\st}) \nonumber \\ 
\simeq && \varprojlim_n \Hom_{\ModPrime(S)_r^{\varphi, N}}(M, \widehat{A}_{\st}/p^n\widehat{A}_{\st}) \nonumber \\ =&& \varprojlim_n \Hom_{\ModPrime(S)_r^{\varphi, N}}(M/p^n M, \widehat{A}_{\st}/p^n\widehat{A}_{\st}) \nonumber \\
\simeq && \varprojlim_{n}\Hom_{\ModPrime(S)_r^{\varphi, N}}(M/p^n M, \widehat{A}_{\st}\otimes (\frac{1}{p^n} \ZZ_p)/\ZZ_p) \nonumber \\
=&& \varprojlim_{n}\Hom_{\ModPrime(S)_r^{\varphi, N}}(M/p^n M, \widehat{A}_{\st,\infty})\nonumber  \\
=&& \varprojlim_n T_{\st,\infty}^\star(M/p^n M). 
\end{eqnarray}

Furthermore, note that the system $\{T_{\st,\infty}^\star(M/p^n M)\}_n$ in the inverse limit is Mittag-Leffler since the modules $T_{\st,\infty}^\star(M/p^n M)$ are of finite $\ZZ_p$-length.

Now let $0\to M_1\xrightarrow{f} M_2\xrightarrow{g} M_3\to 0$ be an exact sequence in $\Mod(S)_{r,\log}^{\varphi, N}$. Recall that this means that the sequences of $S$-modules $0\to M_1\xrightarrow{f} M_2\xrightarrow{g} M_3\to 0$ and $0\to\Fil^r M_1\xrightarrow{f}\Fil^r M_2\xrightarrow{g}\Fil^r M_3\to 0$ are exact. Since each $M_i$ is free as $S$-module, it follows that for every $n\geq 1$ the induced sequence
\begin{equation}\label{ses1}
    0\to M_1/p^n M_1\xrightarrow{f} M_2/p^n M_2\xrightarrow{g} M_3/p^n M_3\to 0
\end{equation} is exact. 
We claim that the sequence 
\[0\to\Fil^r M_1/p^n\Fil^r M_1\xrightarrow{f}\Fil^r M_2/p^n\Fil^r M_2\xrightarrow{g}\Fil^r M_3/p^n\Fil^r M_3\to 0\] is also exact. 
This follows by snake lemma and the fact that multiplication by $p^n$ is injective on $\Fil^r M_3$ (since contained in the free module $M_3$).

Since $T_{\st,\infty}^\star$ is exact, for every $n\geq 1$ the sequence 
\[0\to T_{\st,\infty}^\star(M_3/p^n M_3)\xrightarrow{T_{\st,\infty}^\star(g)} T_{\st,\infty}^\star(M_2/p^n M_2)\xrightarrow{T_{\st,\infty}^\star(f)} T_{\st,\infty}^\star(M_1/p^n M_1)\to 0\] is exact. The Corollary then follows by the Mittag-Leffler property. 
\end{proof}

In this paper we will need covariant versions of the above theorems. 
We consider the functor $\mathbb{D}^r_{\fr}:\Rep_{\ZZ_p, \st}^{[0, r]}(G_K)\to\Rep_{\ZZ_p,  \st}^{[0, r]}(G_K)$ that sends a $\ZZ_p$-representation $T$ to its $r$-twisted $\ZZ_p$-linear dual $\Hom_{\ZZ_p}(T, \ZZ_p)(r)$. Since all objects of $\Rep_{\ZZ_p,  \st}^{[0, r]}(G_K)$ are free as $\ZZ_p$-modules, this is clearly an exact anti-equivalence of categories.
We will abuse notation and denote also by $\mathbb{D}^r_{\fr}$ the exact  anti-equivalence $\mathbb{D}^r_{\fr}:\Rep_{\ZZ_p}(G_\infty)\to\Rep_{\ZZ_p}(G_\infty)$, defined analogously.

\begin{definition}
    We define the \textbf{covariant} \'{e}tale realization functor \[T_{\st,\star}:\Mod(S)_{r, \log}^{\varphi, N}\to \Rep_{\ZZ_p,  \st}^{[0, r]}(G_K)\] to be the composition $T_{\st,\star}=\mathbb{D}_{\fr}^r\circ T_{\st}^\star$. 
 Similarly, we define 
    \[T_{\cris,\star}:\Mod(S)_{r}^{\varphi}\to \Rep_{\ZZ_p}(G_\infty)\] to be the composition $T_{\cris,\star}=\mathbb{D}_{\fr}^r\circ T_{\cris}^\star$. 
\end{definition} 
It follows from Theorem \ref{thm:Liu} and Corollary \ref{cor:exactness1} that $T_{\st,\star}$ and $ T_{\cris,\star}$ are exact functors. In fact, $T_{\st,\star}$ is an equivalence of categories which is exact. 

\begin{definition} \label{def_choice_Mst_}
    We denote by $\mathcal{M}_{\st}: \Rep_{\ZZ_p, \cris}^{[0, r]}(G_K)\rightarrow \Mod(S)_{r}^{\varphi, N}$ (a fixed choice of) the quasi-inverse of $T_{\st,\star}$. 
\end{definition}

\begin{example} Following Example \ref{ex1} we have: $T_{\st,\star}(\mathbf{1})=T_{\st}^\star(\mathbf{1}(r))=\ZZ_p$ and $T_{\st,\star}(\mathbf{1}(r))=T_{\st}^\star(\mathbf{1})=\ZZ_p(r)$. 
    
\end{example}

\begin{remark}
We note that in \cite{Liu08}  the covariant functor $T_{\cris,\star}$ is defined differently; namely, as the composition $T_{\cris,\star}=T_{\cris}^\star\circ \mathbb{D}\Mod^r_{\fr}$, where 
\begin{eqnarray*}
    \mathbb{D}\Mod^r_{\fr}:&& \Mod(S)_r^{\varphi}\to\Mod(S)_r^\varphi\\ && M\mapsto M^\star
\end{eqnarray*}
is an exact antiequivalence that serves as Cartier duality on strongly divisible modules without monodromy defined by X. Caruso in his thesis (\cite{Caruso05Thesis}). For $M\in\Mod(S)_r^{\varphi}$ the dual module $M^\star$ has underlying $S$-module $M^\star:=\Hom_S(M, S)$ and filtration  $\Fil^r(M^\star):=\{f\in M^\star: f(\Fil^r M)\subset\Fil^r S)\}$. For the definition of the Frobenius $\varphi_r$ on $M^\star$ we refer to \cite[Section 4.2]{Liu08}. It turns out that the two definitions of $T_{\cris,\star}$ are equivalent. Namely, the obvious perfect pairing $i: M\times M^\star\to S$ between a module and its dual is compatible with filtrations and Frobenius on both sides if we give $S$ the structure $\mathbf{1}(r)=(S, \Fil^r S, \varphi_r)$. Moreover, it induces a perfect pairing of $\ZZ_p[G_\infty]$-modules
\[T_{\cris}^\star(M)\times T_{\cris}^\star(M^\star)\to T_{\cris}^\star(\mathbf{1}(r)^\star)=T_{\cris}^\star(\mathbf{1})=\ZZ_p(r)\]  (see \cite[Lemma 4.3.1]{Liu08}). This yields a canonical isomorphism \[T_{\cris}^\star(M^\star)\simeq\Hom_{\ZZ_p}(T_{\cris}^\star(M), \ZZ_p(r)),\]  from where the equivalence of the two definitions follows.
\end{remark}

We will need the following more explicit description of $T_{\cris,\star}$. 

\begin{lemma}\label{lem:Tcrisexplicit} 
    Let $M\in\Mod(S)_r^\varphi$. Then there is a canonical isomorphism
    \[T_{\cris,\star}(M)\simeq \Fil^r(M\otimes_S A_{\cris})^{\varphi_r=1}.\]
\end{lemma}
\begin{proof}
This follows by applying \cite[Lemma 4.3.2]{Liu08} for $(M^\star)^\star$ and using the isomorphism $(M^\star)^\star\simeq M$ (cf.~\cite[Theorem 4.2.1]{Liu08}). 
    \end{proof}

\vspace{2pt}

\subsection{Comparison functors between Kisin and Breuil modules}\label{subsec:KisinToBreuil}
\vspace{5pt}
We now review the functor $\mathcal{M}^{(r)}_{\Es}$, connecting the categories of Kisin and Breuil modules. This is the classical functor originally described by Breuil \cite{BreuilUnpub} and further studied by Kisin, Caruso, Liu \cite{Kisin1, Liu08, CarusoLiu} and others. Throughout this section, let us fix an integer $r$ with $0\leq r<p-1$ (we include $r$ in the notation to keep track of the fact that the functor depends on this choice of $r$).

We consider the category $\Mod(S)_{r}^{\varphi}$. Recall from Definition \ref{def:StronglyDivisible} (b) that the objects in this category are triples $(M,\Fil^rM, \varphi_r)$ where $M$ is a free $S$-module of finite rank, $\varphi_r(\Fil^r M)$ generates $M$ as $S$-module, and $\Fil^r M\cap pM=p\Fil^r M$. 

\begin{definition} 
We consider the comparison functor  $\mathcal{M}^{(r)}_{\Es}: \Mod_{\Es, \fr}^{\varphi, \leq r}\rightarrow \Mod(S)_{r}^{\varphi}$ defined as follows. For $\Em \in \Mod_{\Es, \fr}^{\varphi, \leq r},$ we have:
\begin{enumerate}
\item The underlying (finite free) $S$-module of $\mathcal{M}^{(r)}_{\Es}(\Em)$ is the module $S\otimes_{\Es}\varphi_{\Es}^* \Em.$
\item Note that the map $\varphi_{\Em, \lin}: \varphi_{\Es}^* \Em \rightarrow \Em$ induces the map 
$$1\otimes \varphi_{\Em, \lin}: S\otimes_{\Es}\varphi_{\Es}^* \Em \rightarrow S\otimes_{\Es} \Em .$$ Using this map, $\Fil^r\mathcal{M}^{(r)}_{\Es}(\Em)$ is defined as follows:
\[\Fil^r\mathcal{M}^{(r)}_{\Es}(\Em)=\{x \in S\otimes_{\Es}\varphi_{\Es}^* \Em\;|\; 1\otimes \varphi_{\Em, \lin} (x)\in \Fil^r S\otimes_{\Es} \Em \}\,.\]
\item{The operator $\varphi_r: \Fil^r\mathcal{M}^{(r)}_{\Es}(\Em) \rightarrow \mathcal{M}^{(r)}_{\Es}(\Em)$ is given as the composite
\begin{center}
\begin{tikzcd}
\Fil^r\mathcal{M}^{(r)}_{\Es}(\Em)  \ar[r, "1\otimes \varphi_{\Em, \lin}"] & \Fil^rS \otimes_{\Es} \Em \ar[r, "\varphi_{r,S} \otimes 1"] & S\otimes_{\Es, \varphi} \Em= \mathcal{M}^{(r)}_{\Es}(\Em).
\end{tikzcd}
\end{center}}
\end{enumerate}
\end{definition}

\begin{proposition}[{\cite[Proposition 2.1.2, Theorem 2.2.1]{CarusoLiu}}]\label{prop:KisinToBreuilEquivalence}
The functor $\mathcal{M}^{(r)}_{\Es}$  is a well-defined equivalence of categories. Moreover, it is exact with respect to the natural exact structures on both sides.
\end{proposition}

\begin{remark}\label{rem:SNotFlat}
Regarding the exactness statement, we warn the reader that $S$ is not flat over $\Es$: Reducing mod $p$, $\Es$ becomes the DVR $k[[u]]$, whereas in $S/p$, one has $u^{ep}=0$ as a consequence of the relation $E(u)^p=p((p-1)!E(u)^p/p!)$ valid in $S$.

On the other hand, $S$ is ``faithful'' for the tensor product, in the sense that for every finitely generated $\Es$-module $\Em\neq 0,$ one has $S\otimes_{\Es}\Em\neq 0$. Indeed, since $\Em$ is finitely generated, this claim reduces to showing that $S \otimes_{\Es} k \neq 0$ where $k=\Es/(p, u)$, which is true 
because we have $S \otimes_{\Es} k \simeq S/ (p,u) = k$.
\end{remark}

\begin{theorem}\label{thm:QuasiInverseCompatibility}
Given a Kisin module $\Em \in \Mod_{\Es, \fr}^{\varphi, \leq r},$ we have a natural isomorphism $$T_{\Es}^{(r)}(\Em) \stackrel{\sim}\rightarrow T_{\cris, \star}(\mathcal{M}_{\Es}^{(r)}(\Em)). $$
\end{theorem}

\begin{proof}
By \cite[Lemma~3.3.1]{Liu08}, one has a natural isomorphism of $G_{\infty}$-representations
$$\Hom_{\Es, \varphi}(\Em, \ainf) \simeq \Hom_{S, \Fil^r, \varphi_r}(\mathcal{M}^{(r)}_{\Es} (\Em), A_{\cris})\simeq T_{\cris}^{\star}(\mathcal{M}^{(r)}_{\Es} (\Em)).$$
Applying the dual $\mathbb{D}_{\fr}^r$ to both sides now yields the result due to Lemmas \ref{lem:KisinDualT} and \ref{lem:Tcrisexplicit}.
\end{proof}

We use all the established relations between \'{e}tale realization functors to get the following result that allows us to compare the various functors in the opposite direction.

\begin{corollary}\label{cor:CompatibleQuasiInverses}
For $T \in \Rep_{\ZZ_p, \st}^{[0, r]}(G_K),$ there is a natural  isomorphism 
$$U \mathcal{M}_{\st}(T)\simeq \mathcal{M}_{\Es}^{(r)}\underline{\Em}^{(r)}(T)\,,$$
where $U$ is the forgetful functor defined in Definition in \ref{Def_U}.
\end{corollary}

\begin{proof}
Since $T_{\st, \star}(\mathcal{M}_{\st}(T))\simeq T$, the Breuil module $M=U \mathcal{M}_{\st}(T)$ satisfies $T_{\cris, \star}(M) \simeq T|_{G_{\infty}}$ by the discussion after Theorem~\ref{thm:Liu}. At the same time, $M\simeq \mathcal{M}_{\Es}^{(r)}(\Em)$ for some $\Em \in  \Mod_{\Es, \fr}^{\varphi, \leq r}$ by Proposition~\ref{prop:KisinToBreuilEquivalence}, and we have $T_{\Es}^{(r)}(\Em)\simeq T_{\cris, \star}(M) \simeq T|_{G_{\infty}}$ by Theorem~\ref{thm:QuasiInverseCompatibility}. By Proposition~\ref{prop:PropertiesOfTrMr}, this property characterizes $\underline{\Em}^{(r)}(T)$, therefore $\Em \simeq \underline{\Em}^{(r)}(T)$.
\end{proof}

Finally, we use the established equivalences to give a convenient classical explicit description of Tate twists in the context of Kisin modules (cf. the more abstract description in \cite[Example~4.2]{BMS1}).

\begin{example}[Breuil--Kisin twists]
Denote by $c_0$ the element $c_0=E(0)/p\in W^{\times}$. In order to ease the notation, let us in this example identify the Galois representation $\ZZ_{p}(s)$ with its restriction to $G_{\infty}$. 

Consider first the free Kisin module $\Es\{-1\}=\Es t \in \Mod_{\Es}^{\varphi, \leq 1}$, where $\varphi$ on $\Es\{-1\}$ is determined by the prescription $\varphi_{\Es\{-1\}}(t)=(E(u)c_0^{-1})t.$ Then we have $T_{\Es}(\Es\{-1\})=\ZZ_{p}(-1)$ or, equivalently, $T_{\Es}^{(1)}(\Es\{-1\})=\ZZ_{p}$. This follows e.g. from the fact that $\mathcal{M}_{\Es}(\Es\{-1\})=\mathbf{1} \in \Mod(S)_{r}^{\varphi},$ as is explicitly computed in \cite[Exemple 2.2.3]{BreuilUnpub}.

Similarly, setting $\Es\{1\}=\Es\{-1\}^{\vee},$ we have a Kisin module with $T_{\Es}(\Es\{1\})=\ZZ_p(-1)^{\vee}=\ZZ_p(1).$ For $r \in \NN, $ we define 
$$\Es\{-r\}=\Es\{-1\}^{\otimes r},\;\;\;\Es\{r\}=\Es\{1\}^{\otimes r}(\simeq \Es\{-r\}^{\vee})\,.$$
Then we have $\Es\{-r\} \in \Mod_{\Es}^{\varphi, \leq r}, $ and $\Es\{-r\}=\Es t$ with $\varphi_{\Es\{-r\}}(t)=(E(u)c_0^{-1})^{r}t.$ Moreover, since $T_{\Es}^{(r)}(\Es\{-r\})=\ZZ_p$,  Corollary~\ref{cor:CompatibleQuasiInverses} implies that $\mathcal{M}_{\Es}^{(r)}(\Es\{-r\})=\mathbf{1}$.
\end{example}

\vspace{5pt}

\section{Exactness Results}\label{sec:ExactnessResults}

\subsection{Some counterexamples to exactness}

\begin{example}\label{ex:KeyCounterexample}
Let $p>2$ be a prime and let $K=\mathbb{Q}_p[(-p)^{1/(p-1)}].$ Fix the uniformizer $\pi=(-p)^{1/(p-1)}$, and its minimal polynomial $E(u)=u^{p-1}+p$. We consider the category $\Mod_{\Es, \fr}^{\varphi, \leq 1}$. By Proposition~\ref{prop:KisinBasicProperties} (4), 
the category is, in fact, equivalent to $\Rep_{\ZZ_p,  \cris}^{[-1,0]} (G_K)$ via the functors $\Em$ and $T_{\Es}$. In particular, to give an example of a short exact sequence of crystalline $G_K$-lattices that becomes non-exact after applying $\Em,$ it is enough to produce a left exact sequence in $\Mod_{\Es, \fr}^{\varphi, \leq 1}$, not exact on the right, such that it becomes short exact after applying $T_{\Es}$.

To that end, consider the module $\Em=\Es e_1 \oplus \Es e_2$, where $\varphi_{\Em}$ is given by $$\varphi_{\Em}(e_1)=e_1, \;\; \varphi_{\Em}(e_2)=-ue_1+E(u)e_2\,.$$
Note that the Kisin module $\Es\{-1\}=\Es t$ has, in this setting, Frobenius given by $\varphi_{\Es \{-1\}} (t)=E(u)t$. In particular, we have the left exact sequence in $\Mod_{\Es, \fr}^{\varphi, \leq r}$:
\begin{equation}\label{eq:KisinCounterexampleSequence}
\begin{tikzcd}
0 \ar[r]& \Es\{-1\} \ar[r, "\alpha"] & \Em \ar[r, "\beta"] & \Es
\end{tikzcd}
\end{equation}
where $\alpha(t)=-ue_1+pe_2,\;\; \beta(e_1)=p, \;\; \beta(e_2)=u$. The remaining cokernel $\coker \beta$ is isomorphic to $\Es/(p, u)$ as an $\Es$-module, hence it vanishes after $T_{\Es}$. Consequently, the associated sequence
\begin{equation}\label{eq:GaloisCounterexampleSequence}
\begin{tikzcd}
0 \ar[r]& \ZZ_p(-1) \ar[r, "T_{\Es}(\alpha)"] & T_{\Es}(\Em) \ar[r, "T_{\Es}(\beta)"] & \ZZ_p \ar[r] & 0
\end{tikzcd}
\end{equation}
 in $\Rep_{\ZZ_p, \cris}^{[-1,0]} (G_K)$ is short exact.
 \end{example}

\begin{remark}
\begin{enumerate}
\item{Example~\ref{ex:KeyCounterexample} shows that the functor $\underline{\Em}: \Rep_{\ZZ_p, \cris}^{[-1, 0]}(G_K)\rightarrow \Mod_{\Es}^{\varphi}$ is not exact in general. Twisting the sequence of Galois lattices by $\ZZ_p(r)$ then yields an example showing that $\underline{\Em}^{(r)}: \Rep_{\ZZ_p,  \cris}^{[0, r]}(G_K)\rightarrow \Mod_{\Es}^{\varphi}$ is also in general non-exact (for every $r \geq 1$). 
}
\item{Example~\ref{ex:KeyCounterexample} is a slight modification of Liu's example \cite[Example~2.21]{LiuLatticesPhiN}, constructed with a different aim. The modifications of the base field $K$ and of the operator $\varphi_{\Em}$ are made to demonstrate that the extreme terms of the sequence may be taken in the form that is obviously ``Fontaine-Laffaille'', i.e., the associated crystalline $\ZZ_p[G_K]$-lattices $\ZZ_p(-1)$ and $ \ZZ_p$ come from crystalline $\ZZ_p[G_{K_0}]$-lattices by restriction of action (again given by $\ZZ_p(-1)$ and $ \ZZ_p$ in our case). Given that the functor from crystalline $\ZZ_p[G_{K_0}]$-lattices with Hodge-Tate weights in $[0, r]$ to strongly divisible modules in the sense of Fontaine--Laffaille is exact, the example can be interpreted as follows: the lattice $T_{\Es}^{(1)}(\Em)$ is a crystalline $G_K$-lattice whose action cannot be extended to $G_{K_0}$ while maintaining crystallinity.}

\item{As noted earlier, $\Mod_{\Es}^{\varphi, \leq 1}$ is equivalent to the category $p\text{-}\mathsf{Div}_{\oh_K}$ of $p$-divisible groups over $\oh_K$ \cite{Kisin1, LiuBT}. Moreover, in the subsequent text (Theorem~\ref{thm:KisinToBreuilIsBiExact}), we will prove that the exact functor $\mathcal{M}_{\Es}^{(1)}$ from Proposition \ref{prop:KisinToBreuilEquivalence} has exact quasi-inverse. Combining this with \cite[Theor\'{e}me 4.2.1.6]{Breuil00} yields that the equivalence functor $\Mod_{\Es}^{\varphi, \leq 1} \stackrel{\sim}\longrightarrow p\text{-}\mathsf{Div}_{\oh_K}$ is an equivalence of exact categories (at least when $p>2$). Example~\ref{ex:KeyCounterexample} can therefore be reinterpreted as giving a $3$-term complex of $p$-divisible groups that is not short exact, yet the induced sequence of their Tate modules is short exact. This is not too surprising in view of the fact that the Tate module depends only on the generic fiber of the given $p$-divisible group, and it has been observed previously, cf. \cite[Remark~5.2.9]{Scholze_Weinstein}.}
\end{enumerate}
\end{remark}

\begin{corollary}\label{cor:non-exactness}
The quasi-inverse $\mathcal{M}_{\st}: \Rep_{\ZZ_p,  \cris}^{[0, r]}(G_K)\rightarrow \Mod(S)_{r}^{\varphi, N}$ to the functor $T_{\st, \star}$ is in general not left exact. More precisely, there exists a base field $K/\mathbb{Q}_p$ such that for each $r$ with $1\leq r <p-1$, there is a short exact sequence ($\tau$) in $\Rep_{\ZZ_p,  \cris}^{[0, r]}(G_K)$ such that $\mathcal{M}_{\st}(\tau)$ is not even left exact as a sequence of abstract $S$-modules.
\end{corollary}

\begin{proof}
Let the field $K$ be the same field as in Example~\ref{ex:KeyCounterexample}, and consider the sequence of Galois lattices ($\tau$) obtained from (\ref{eq:GaloisCounterexampleSequence}) by 
tensoring with $\ZZ_p(r)$. 
Then $\underline{\Em}^{(r)}(\tau)$ is again the left exact sequence of Kisin modules (\ref{eq:KisinCounterexampleSequence}) as in Example~\ref{ex:KeyCounterexample} (except viewed in the bigger category $\Mod_{\Es, \fr}^{\varphi, \leq r}$). Corollary~\ref{cor:CompatibleQuasiInverses} implies that $U\mathcal{M}_{\st}(\tau)$ is obtained from $\underline{\Em}^{(r)}(\tau)$ by applying the functor $\mathcal{M}_{\Es}^{(r)}$. 

Extending the sequence $\underline{\Em}^{(r)}(\tau)$ by the remaining cokernel, we have the sequence
\begin{center}
\begin{tikzcd}
0 \ar[r]& \Es\{-1\} \ar[r] & \Em \ar[r] & \Es \ar[r] & \Es/(p, u) \ar[r]&0,  
\end{tikzcd}
\end{center}
and after the flat base-change $\Es\otimes_{\Es, \varphi} -$, we obtain the exact sequence of $\Es$-modules
\begin{center}
\begin{tikzcd}[column sep = small]
0 \ar[rr]& & \varphi^*\Es\{-1\} \ar[rr] && \varphi^*\Em \ar[dr]\ar[rr] && \varphi^*\Es \ar[rr] && \varphi^*(\Es/(p, u)) \ar[rr]&&0\\
&&&&& C \ar[dr]\ar[ur] &&&&&\\
&&&& 0 \ar[ur]&& 0 &&&&
\end{tikzcd}
\end{center}
where $C$ is the intermediate kernel/cokernel. The first three terms are finite free $\Es$-modules and for the last term, we have $\widetilde{k}:=\varphi^*(\Es/(p, u))=\Es/(p, u^p)$ as an $\Es$-module. 
Applying $S\otimes_{\Es}-$ to the first short exact sequence yields the long exact sequence
\begin{center}
\begin{tikzcd}
0 \ar[r]& \Tor_1^{\Es}(S, C)\ar[r] & S\otimes_{\Es} \varphi^*\Es\{-1\} \ar[r] & S\otimes_{\Es} \varphi^*\Em \ar[r] & S\otimes_{\Es}C \ar[r]&0,  
\end{tikzcd}
\end{center}
and from the second sequence, we likewise obtain the long exact sequence
\begin{center}
\begin{tikzcd}
0 \ar[r]& \Tor_2^{\Es}(S, \widetilde{k}) \ar[r]& \Tor_1^{\Es}(S, C) \ar[r]
& 0 \ar[r]
\ar[d, phantom, ""{coordinate, name=Z}]
& \Tor_1^{\Es}(S, \widetilde{k})\ar[dlll,
rounded corners,
to path={ -- ([xshift=2ex]\tikztostart.east)
|- (Z) [near end]\tikztonodes
-| ([xshift=-2ex]\tikztotarget.west)
-- (\tikztotarget)}] \\
&S\otimes_{\Es}C \ar[r]
& S\otimes_{\Es}\varphi^*\Es \ar[r]
& S\otimes_{\Es}\widetilde{k} \ar[r]& 0\,.
\end{tikzcd}
\end{center}
Since $\widetilde{k}$ is $p$-torsion, so is $\Tor_2^{\Es}(S, \widetilde{k})$ and hence also $\Tor_1^{\Es}(S, C)$; on the other hand, the latter is a submodule of the free $S$-module $S \otimes_{\Es}\varphi^*\Es\{-1\}$ and thus, we have $\Tor_1^{\Es}(S, C)=0$. Consequently, the first long exact sequence is just short exact since the first term vanishes. 

On the other hand, $\Tor_1^{\Es}(S, \widetilde{k})\neq 0$, since it similarly surjects on the nonzero kernel in the right exact sequence 
\begin{center}
\begin{tikzcd}
S/p \ar[r, "u^p"] & S/p \ar[r] & S \otimes_{\Es}\widetilde{k} \ar[r]& 0,
\end{tikzcd}
\end{center}
obtained from the short exact sequence of $\Es$-modules
\begin{center}
\begin{tikzcd}
0 \ar[r]&\Es/p \ar[r, "u^p"] & \Es/p \ar[r] & \widetilde{k} \ar[r]& 0
\end{tikzcd}
\end{center}
by the base change $S\otimes_{\Es}-$.

Thus, we see that the complex 
\begin{center}
\begin{tikzcd}
0\ar[r] & S\otimes_{\Es}\varphi^*\Es\{-1\} \ar[r] & S\otimes_{\Es}\varphi^*\Em \ar[r] & S\otimes_{\Es}\varphi^*\Es \ar[r] & 0
\end{tikzcd}
\end{center}
obtained from $\underline{\Em}^{(r)}(\mu)$ by the functor $\mathcal{M}_{\Es}^{(r)}$, has nonzero cohomology in the middle term $(\simeq \Tor_1^{\Es}(S, \widetilde{k})\neq 0)$ and in the next term ($\simeq S\otimes_{\Es}\widetilde{k}\neq 0$, cf. Remark~\ref{rem:SNotFlat}). In particular, the sequence is not left exact.
\end{proof}

\subsection{Exactness results for $1$-extensions}

We now give some positive results. We start with the fact that unlike for Breuil modules, the Kisin functor is left exact. This is the content of \cite[Lemma~2.19]{LiuLatticesPhiN}, stated in the contravariant version and for the case of the category $\Mod_{\Es}^{\varphi, \leq r}$. Rather than dualizing the functor, we give an alternative ``geometric'' proof (akin to extending a vector bundle over a puncture in a plane) that applies directly in the case of covariant functors.
Recall that $\Rep_{\ZZ_p, \fr}^{\BK}(G_{\infty})$ denotes the essential image of the functor $T_{\Es}$ restricted to free Kisin modules. 

\begin{proposition}\label{prop:LeftExactKisinModules}
For every integer $r \in \ZZ$, the functor $\underline{\Em}^{(r)}:\Rep_{\ZZ_p, \fr}^{\BK}(G_{\infty})\rightarrow  \Mod_{\Es, \fr}^{\varphi}$ is left exact. That is, given a short exact sequence 
\begin{equation}\tag{$\tau$}
\begin{tikzcd}
0\ar[r]& T' \ar[r] & T \ar[r] & T'' \ar[r] & 0
\end{tikzcd}
\end{equation}
in $\Rep_{\ZZ_p, \fr}^{\BK}(G_{\infty})$, the corresponding sequence of Kisin modules
\begin{center}
\begin{tikzcd}
0\ar[r]& \Em' \ar[r, "\alpha"] & \Em \ar[r, "\beta"] & \Em'' 
\end{tikzcd}
\end{center}
obtained by applying $\underline{\Em}^{(r)}$ is left exact as a sequence of $\Es$-modules. Moreover, $\coker \beta$ is annihilated by $(p, u)^k$ for big enough $k$.
\end{proposition}

\begin{proof}
 Since the functors $\underline{\Em}^{(r)}, \underline{\Em}^{(0)}=\underline{\Em}$ are obtained from one another by precomposition with the exact functor $-\otimes_{\ZZ_p}\ZZ_{p}(\pm r),$ it suffices to prove the claim for $\underline{\Em}$. It is further enough to show exactness for $U\underline{\Em}$ where $U$ is the (exactness-preserving and -reflecting) forgetful functor $\Mod_{\Es}^{\varphi} \rightarrow \Mod_{\Es}$ to the abelian category of abstract $\Es$-modules. 

The complex
\begin{equation}\tag{$\mu$}
\begin{tikzcd}
0\ar[r]& \Em' \ar[r, "\alpha"] & \Em \ar[r, "\beta"] & \Em'' \ar[r]& 0, 
\end{tikzcd}
\end{equation}
obtained by applying $\underline{\Em}$ to ($\tau$), recovers the short exact sequence ($\tau$) by applying the exact functor $T_{\Es}$. By Proposition~\ref{prop:KerOfTisUTorsion} and Lemma~\ref{lem:UtorIsPtor}, the cohomologies of ($\mu$) are Kisin modules annihilated by a power of $(p, u)$ (which, in particular, proves the ``Moreover'' claim). Consequently, the functors 
\begin{align*}
 T \mapsto U\underline{\Em}(T)[1/u],\quad
 T \mapsto U\underline{\Em}(T)[1/p],\quad
 T \mapsto U\underline{\Em}(T)[1/pu]
\end{align*} 
are exact functors from $\Rep_{\ZZ_p, \fr}^{\BK}(G_K)$ to $\Mod_{\Es}$. In total, we have an exact functor 
\begin{align*}
\underline{\Em}^{\rightarrow \leftarrow}: \Rep_{\ZZ_p, \fr}^{\BK}(G_{\infty})& \longrightarrow \Mod_{\Es}^{\rightarrow \leftarrow},\\ 
 T& \longmapsto \left(U\underline{\Em}(T)[1/u] \stackrel{\subseteq}{\rightarrow} U\underline{\Em}(T)[1/pu] \stackrel{\supseteq}{\leftarrow}U\underline{\Em}(T)[1/p] \right)
\end{align*}
where $\Mod_{\Es}^{\rightarrow \leftarrow}$ denotes the (diagram) category of cospans of $\Es$-modules. Since every finite free $\Es$-module $F$ satisfies $F=F[1/u] \cap F[1/p]$ (with intersection taken in $F[1/pu]$) and $U\underline{\Em}$ takes values in finite free $\Es$-modules, it follows that $U\underline{\Em} $ is naturally equivalent to the composite 
$$\underline{\Em}:\Rep_{\ZZ_p, \fr}^{\BK}(G_{\infty})\stackrel{\underline{\Em}^{\rightarrow \leftarrow}}\longrightarrow \Mod_{\Es}^{\rightarrow \leftarrow} \stackrel{\lim}\longrightarrow \Mod_{\Es},$$
where $\lim$ is the fiber product functor. Since this functor is left exact (in the usual sense), it follows that $U\underline{\Em},$ hence $\underline{\Em}$, is left exact in the sense of the proposition. 
\end{proof}

 For the purposes of the following lemma, let us set $p^{\infty}$ to be $0$ (and, of course, $p^0=1$). Let us denote by $v_p$ the $p$--adic valuation on $W,$ written additively (this way, $v_p(p^n)=n$ is valid for all $n \in \mathbb{N}\cup \{\infty\}$).

\begin{lemma}\label{lem:KeyLemma}
Let $I \subseteq \Es$ be an ideal with the property $\Es\varphi(I) \supseteq I$. Then $I=(p^n)$ for some $n \in \NN \cup \{\infty\}$.
\end{lemma}

\begin{proof}
Consider the smallest $n \in \NN \cup \{\infty\}$ such that $p^n \in I$. 
Let us show that $I=(p^n)$. Suppose the contrary, i.e., that there is a power series $f(u)=\sum_i a_i u^i$ with $ f(u) \in I \setminus (p^n)$. In particular, some of the terms of $f(u)$ are not divisible by $p^n$, that is, $v_p(a_j)=m<n$ for some $j, m\in \mathbb{N}$. Among all the possible choices of $(f(u), m, j),$ we make the choice so that:
\begin{enumerate}
\item $m$ is the smallest possible, 
\item among all possible choices of $(f(u), j)$ with the $m$ as in (1), we make the choice so that $j$ is the smallest possible. 
\end{enumerate} 

First, we claim that $j>0$. Indeed, $j=0$ would imply that $f(u)$ is of the form $$f(u)=\widetilde{a_0}p^m+\sum_{i\geq 1}\widetilde{a_i}p^mu^i=p^m(\widetilde{a_0}+\sum_{i\geq 1} \widetilde{a_i}u^i),$$
where $\widetilde{a_0} \in W^{\times}$ and $\widetilde{a_i} \in W$ for each $i$. But the expression in the bracket is a unit since $\widetilde{a_0}$ is a unit and $u \in \mathrm{rad}(\Es),$ the Jacobson radical of $\Es$. 
This means that $p^m \in I$ with $m<n,$ contradicting the choice of $n$.  

With this setup, let $N$ be an integer big enough so that $jp/N<1$ (e.g., $N=pj+1$ works). Under the surjective homomorphism
\begin{align*}
\Es & \rightarrow \Es/(u^N-p)\simeq W[p^{1/N}],\;\; u \longmapsto p^{1/N}, 
\end{align*}
one easily sees that $I$ is mapped onto the ideal $(p^{m+j/N})$ (with $f(u)$ mapped to an element of the form $p^{m+j/N}\cdot(\text{unit})$). On the other hand, the ideal $\Es \varphi(I)$ is mapped onto the ideal $(p^{m+jp/N})$, since for every $g(u)\in I,$ the $p$-adic valuation of all the $W$-coefficients of $\varphi(g(u))$ remains the same but all the powers of the $u$-terms are multiplied by $p$ (and, once again, $\varphi(f(u))$ is mapped to an element of the form $p^{m+jp/N}\cdot(\text{unit})$). But now the strict inclusion $(p^{m+jp/N}) \subsetneq (p^{m+j/N})$ is in contradiction with $\Es\varphi(I)\supseteq I$. 

This shows that no such power series $f(u)\in I \setminus (p^n)$ can exist. Thus, $I=(p^n).$   
\end{proof}

We can now proceed with the proof of the main theorem. 
Let us use the notation $\Rep_{\ZZ_p, \fr}^{\BK, [0, r]}(G_{\infty})$ for the essential image of $\Mod_{\Es, \fr}^{\varphi, \leq r}$ under $T_{\Es}^{(r)}$.  

\begin{theorem}\label{thm:ExactKisinModules}
Fix an integer $r\in \NN$. Given a short exact sequence in $\Rep_{\ZZ_p, \fr}^{\BK, [0, r]}(G_{\infty})$ of the form 
\begin{center}
\begin{tikzcd}
0\ar[r]& T \ar[r] & L \ar[r] & \ZZ_p \ar[r] & 0\,,
\end{tikzcd}
\end{center}
where $\ZZ_p$ is the rank $1$ trivial Galois module, the corresponding sequence of Kisin modules
\begin{center}
\begin{tikzcd}
0\ar[r]&\underline{\Em}^{(r)}(T) \ar[r, "\alpha"] & \underline{\Em}^{(r)}(L) \ar[r, "\beta"] & \underline{\Em}^{(r)}(\ZZ_p) \ar[r] & 0
\end{tikzcd}
\end{center}
is also short exact.
\end{theorem}

\begin{remark}\label{rem:ExactKisinModlesComments}
The (exact) restriction functor $T\mapsto T|_{G_{\infty}}$ maps $\Rep_{\ZZ_p, \st}^{[0,r]}(G_K)$ (hence, also $\Rep_{\ZZ_p,  \cris}^{[0,r]}(G_K)$) into $\Rep_{\ZZ_p, \fr}^{\BK, [0, r]}(G_{\infty})$. As a consequence, the conclusion of Theorem~\ref{thm:ExactKisinModules} is valid also for the functors $\underline{\Em}^{(r)}:\Rep_{\ZZ_p, \bullet}^{[0,r]}(G_K)\rightarrow \Mod_{\Es, \fr}^{\varphi, \leq r}$ with $\bullet=\cris$ or $\st$. These are the cases of greatest interest
 (cf.~Theorem \ref{thm:MainThmKisin} in the introduction), 
 but we formulate Theorem~\ref{thm:ExactKisinModules} in greater generality for the purpose of wider applicability. 

It is also worth pointing out that the conclusion of Theorem~\ref{thm:ExactKisinModules} remains valid if the functor $\underline{\Em}^{(r)}$ is replaced by $\underline{\Em}$ (as long as the source category is $\Rep_{\ZZ_p, \fr}^{\BK, [0, r]}(G_{\infty})$). Indeed, passing between the two functors amounts to applying the twist by $\Es\{\pm r\},$ which is an exact operation. 
\end{remark}

\begin{proof}[Proof of Theorem~\ref{thm:ExactKisinModules}]
Denote $\Em=\underline{\Em}^{(r)}(L)$. Note that $\underline{\Em}^{(r)}(\ZZ_p)=\underline{\Em}(\ZZ_p(-r))=\Es\{-r\},$ the rank one Kisin module with generator $t$ and Frobenius satisfying $\varphi_{\Es\{-r\}}(t)=(c_0^{-1}E(u))^r t$, with $c_0=E(0)/p \in W^{\times}.$ By Proposition~\ref{prop:LeftExactKisinModules}, it is enough to show that the map $\beta: \Em \rightarrow \Es\{-r\}$ is surjective.

The image of $\beta$ is a nonzero submodule of $\Es\{-r\}=\Es t$, hence of the form $It$ for some ideal $0 \neq I\subseteq \Es$. By compatibility of $\beta$ with $\varphi$'s, the inclusion $\Es\varphi_{\Em}(\Em)\supseteq E(u)^r \Em$ yields $\Es\varphi_{\Es\{-r\}}(It)\supseteq E(u)^r It$. Since $\varphi_{\Es\{-r\}}(t)=(\text{unit})\cdot E(u)^rt$, this latter inclusion can be rewritten as $\Es\varphi(I)E(u)^rt \supseteq E(u)^r It$. Cancelling $E(u)^rt$ on both sides, we obtain the inclusion of ideals $\Es\varphi(I) \supseteq I$. By Lemma~\ref{lem:KeyLemma}, $I=(p^n)$ for some $n \in \mathbb{N}$ ($n \neq \infty$ since $I \neq 0$).

Suppose for contradiction that $n>0$. Then the sequence of Kisin modules can be extended to the exact sequence
\begin{center}
\begin{tikzcd}
0 \ar[r] & \underline{\Em}^{(r)}(T) \ar[r] &  \underline{\Em}^{(r)}(L) \ar[r] & \Es\{-r\} \ar[r] & \Es\{-r\}/p^n \Es\{-r\} \ar[r] & 0.
\end{tikzcd}
\end{center}
Applying the exact functor $T_{\Es}^{(r)},$ we obtain the sequence of $\ZZ_p[G_{\infty}]$-modules 
\begin{center}
\begin{tikzcd}
0 \ar[r] & T \ar[r] &  L \ar[r] & \ZZ_p \ar[r] & \ZZ/p^n\ZZ \ar[r] & 0.
\end{tikzcd}
\end{center}
This is a contradiction, since the map to $\ZZ_p$ was surjective. Thus, $n=0$, hence $I=\Es$ and $\beta$ is surjective.
\end{proof}

\begin{corollary}\label{cor:ExactBreuilModules}
Let $r$ be an integer satisfying $0\leq r<p-1$. Given a short exact sequence in $\Rep_{\ZZ_p,  \cris}^{[0,r]}(G_K)$ ($\Rep_{\ZZ_p,  \st}^{[0,r]}(G_K)$, resp.) of the form 
\begin{center}
\begin{tikzcd}
0\ar[r]& T \ar[r] & L \ar[r] & \ZZ_p \ar[r] & 0\,,
\end{tikzcd}
\end{center}
where $\ZZ_p$ is the rank $1$ trivial Galois module, the corresponding sequence of strongly divisible modules 
\begin{center}
\begin{tikzcd}
0\ar[r]& \mathcal{M}_{\st}(T) \ar[r] & \mathcal{M}_{\st}(L) \ar[r] & \mathbf{1} \ar[r] & 0
\end{tikzcd}
\end{center}
is also short exact with respect to the exact structure on $\Mod(S)_r^{\varphi, N}$ (on $\Mod(S)_{r,\log}^{\varphi, N}$, resp.).
\end{corollary}

\begin{proof}
From Theorem~\ref{thm:ExactKisinModules},  Corollary~\ref{cor:CompatibleQuasiInverses} and Proposition~\ref{prop:KisinToBreuilEquivalence}, we obtain that the sequence 
\begin{center}
\begin{tikzcd}
0\ar[r]& U\mathcal{M}_{\st}(T) \ar[r] & U\mathcal{M}_{\st}(L) \ar[r] & \mathbf{1} \ar[r] & 0
\end{tikzcd}
\end{center}
in $\Mod(S)_r^{\varphi}$ is short exact (in the sense of the exact category structure, i.e., including exactness of $\Fil^r$'s). But the forgetful functor $U$ does not affect the exactness of the sequence one way or another, that is, the exactness does not depend on the monodromy operators $N$ in any way. The claim thus follows. 
\end{proof}

\begin{theorem}\label{Th:isomorphism_ext}
Let $r$ be an integer satisfying $0\leq r<p-1$.
Given $T \in \Rep_{\ZZ_p, \cris}^{[0,r]}(G_K),$ one has $$H^1_f(K, T)\simeq \Ext^1_{\Mod(S)_r^{\varphi, N}}(\mathbf{1}, \mathcal{M}_{\st}(T)),$$
 where the Ext group on the right is the Yoneda Ext in an exact category.
\end{theorem}

\begin{proof}
By \cite[\S 3]{BlochKato}, the left-hand side of the isomorphism to be proved is naturally isomorphic to $\Ext^1_{\Rep_{\ZZ_p, \cris}^{[0,r]}(G_K)}(\ZZ_p, T),$ classifying crystalline extensions of the trivial module $\ZZ_p$ by $T$. Likewise, $\Ext^1_{\Mod(S)_r^{\varphi, N}}(\mathbf{1}, \mathcal{M}_{\st}(T))$ classifies extensions (in $\Mod(S)_r^{\varphi, N}$) of $\mathcal{M}_{\st}(T)$ by $\mathbf{1}$. Given such an extension, applying the exact functor $T_{\st, \star}$ produces a crystalline extension of $T$ by $\ZZ_p$. Conversely, if one starts with an extension of $T$ by $\ZZ_p$, Corollary~\ref{cor:ExactBreuilModules} guarantees that application of $\mathcal{M}_{\st}$ gives an extension of $\mathcal{M}_{\st}(T)$ by $\mathbf{1}$. Due to $\mathcal{M_{\st}}$ being an equivalence of categories, these correspondences of extensions respect the usual equivalence relations of extensions. Consequently, we have an isomorphism $\Ext^1_{\Mod(S)_r^{\varphi, N}}(\mathbf{1}, \mathcal{M}_{\st}(T))\simeq \Ext^1_{\Rep_{\ZZ_p,  \cris}^{[0,r]}(G_K)}(\ZZ_p, T)$. This proves the claim.
\end{proof}

Using essentially the same proof, one obtains the following semistable version of the same claim.

\begin{theorem}\label{Th:isomorphism_ext_st}
Let $r$ be an integer satisfying $0\leq r<p-1$.
Given $T \in \Rep_{\ZZ_p,  \st}^{[0,r]}(G_K),$ one has $$H^1_{\st}(K, T)\simeq \Ext^1_{\Mod(S)_{r, \log}^{\varphi, N}}(\mathbf{1}, \mathcal{M}_{\st}(T)),$$
 where the Ext group on the right is the Yoneda Ext group in the exact category $\Mod(S)_{r, \log}^{\varphi, N}$.
\end{theorem}

\begin{remark}\label{remarkGap}
 Theorem  \ref{Th:isomorphism_ext} fills a gap in the proof \cite[Proposition 4.1.4]{Iovita/Marmora}. Indeed, in that proof, the authors used 
 the isomorphism (using the notation of \emph{loc.cit})  
 $$\Ext^1_{\Mod(S)^r}(\mathbf{1}, M) \xrightarrow{\sim} \Ext^1_{\Rep_{\ZZ_p}(G_K)_{\mathrm{crys}}^{[0,r]}}(\ZZ_p, T_{\st}(M))
 \qquad \text{(by Cor. 2.3.5)}.$$
Actually, Corollary 2.3.5 in \cite{Iovita/Marmora} implies the existence and the injectivity of the map, since $T_{\st}$ is an exact equivalence, but considering that any quasi-inverse of $T_{\st}$ is not exact, as we have seen in Corollary \ref{cor:non-exactness}, for the surjectivity we need  Theorem \ref{Th:isomorphism_ext}. 
\end{remark}

The same argument as in the proof of Theorem~\ref{Th:isomorphism_ext} also applies in the context of Kisin modules, only starting directly from Theorem~\ref{thm:ExactKisinModules} instead of Corollary~\ref{cor:ExactBreuilModules}. As mentioned in 
the introduction (see Remark \eqref{RemarkIntro2} after Theorem \ref{thm:MainThmBreuil}),  
we formulate these consequences explicitly in the case of Kisin $G_K$-modules.

\begin{theorem}\label{thm:BKGKComputesCohomology}
Fix an integer $r \geq 0$.
\begin{enumerate}
\item For every $T \in \Rep_{\ZZ_p, \st}^{[0, r]}(G_K),$ we have $$H^1_{\st}(K, T)\simeq \Ext^1_{\Mod_{\Es, \fr}^{\varphi, G_K, \leq r}}(\Es\{-r\}, \underline{\Em}^{(r)}_{G_K}(T))\,.$$
\item For every $T \in \Rep_{\ZZ_p, \cris}^{[0, r]}(G_K),$ we have $$H^1_f(K, T)\simeq \Ext^1_{\Mod_{\Es, \fr, \cris}^{\varphi, G_K, \leq r}}(\Es\{-r\}, \underline{\Em}^{(r)}_{G_K}(T))\,.$$
\end{enumerate}
\end{theorem}

\begin{remark}\label{rem:VariantsOfMainResult}
In this remark, we discuss some mild extensions or analogues of the main results of this section.
\begin{enumerate}
\item Just like in the case of Theorem~\ref{thm:ExactKisinModules}, taking the twist $\underline{\Em}^{(r)}_{G_K}$ instead of $\underline{\Em}_{G_K}$ in Theorem~\ref{thm:BKGKComputesCohomology} is, strictly speaking, not necessary (if one drops the ``$\leq r$'' condition in the description of the category of Kisin $G_K$-modules). The choice of $\underline{\Em}^{(r)}_{G_K}$ is, however, natural since it translates the problem to computing extensions of \textit{effective} Kisin $G_K$-modules. On the other hand, the requirement that $T$ has Hodge--Tate weights in the range $[0,r]$ (i.e., non-negative) is crucial. 
Indeed, Example \ref{ex:KeyCounterexample} shows that when $T=\ZZ_p(-1)$, the Bloch-Kato Selmer group $H^1_f(K,T)$ can no longer be computed by a group of $1$-extensions of Breuil-Kisin modules. We note that in this case, if  $V=\QQ_p(-1)$ is the associated crystalline representation of $G_K$, then $H^1_f(K,V)=0$ (see \cite[Example 3.9]{BlochKato}), and hence the counterexample to exactness arises from a torsion class in $H^1_f(K,\ZZ_p(-1))$. 

\item From the proof of Theorem~\ref{thm:ExactKisinModules}, it is clear that the key property of the Kisin module $\Es\{-r\}$ used in the proof is that the Frobenius operator acts on the generator $t$ by $\varphi(t)=E(u)^r \cdot \mathrm{(unit)} \cdot t$. In particular, modifying the unit does not change the outcome of the Theorem. Consequently, all the $\Ext$-isomorphisms hold when the trivial representation $\ZZ_p$ is replaced by an unramified character.
\item As alluded to in Introduction, Theorem~\ref{thm:BKGKComputesCohomology} holds when Gao's notion of Kisin $G_K$-modules is replaced by another category of Kisin modules equivalent via an enhancement of the functors $T_{\Es}, \underline{\Em}$ to the category of lattices in crystalline or semistable representations, such as the variants \cite{LiuPhiGHat, DuLiuPhiGHat, YaoSemistableLattices}, \cite[Appendix~F]{EmertonGee}. In fact, passing between these categories via the equivalence functors to semistable lattices, one observes that the underlying Kisin modules between corresponding objects are always the same. Consequently, all these Kisin-type categories are equivalent to each other as exact categories.
\end{enumerate}
\end{remark}

\subsection{Formulation via prismatic $F$-crystals}\label{subsec:Fcrystals}

Let us now reformulate the main result in terms of prismatic $F$-crystals of Bhatt and Scholze \cite{BhattScholze2} and thus, prove Theorem~\ref{thm:MainThmPrismatic}. Recall that the absolute prismatic site $\spf(\oh_K)_{\Prism}$ is the category whose objects are bounded prisms $(A, I)$ endowed with a map $\oh_K \to A/I$, and whose morphisms are given by the opposites of morphisms of prisms $(A, I)\to (B, J)=(B, IB)$ over $\oh_K$. The topology is given by ($(p, I)$-completely) flat maps of prisms. The resulting site is equipped with the structure sheaf $\oh_{\Prism}:(A, I) \mapsto A$, yielding a ringed site $(\spf(\oh_K)_{\Prism}, \oh_{\Prism}).$ The sheaf of rings $\oh_{\Prism}$ further comes with a distinguished sheaf of ideals $\mathcal{I}_{\Prism} \subseteq \oh_{\Prism}$ given by $\mathcal{I}_{\Prism}((A, I))=I$. 

Recall that a prismatic $F$-crystal on $\spf(\oh_K)$ is a vector bundle $\mathcal{E}$ on $(\spf(\oh_K)_{\Prism}, \oh_{\Prism})$ equipped with an isomorphism $\varphi_{\mathcal{E}}:\varphi^* \mathcal{E}[1/\mathcal{I}_{\Prism}]\stackrel{\sim}\rightarrow \mathcal{E}[1/\mathcal{I}_{\Prism}]$ (where for a sheaf $\mathcal{F}$ on $(\spf(\oh_K)_{\Prism}, \oh_{\Prism})$, $\varphi^*\mathcal{F}$ is given by $(\varphi^* \mathcal{F})(A, I)=\varphi_A^*(\mathcal{F}(A, I))$, with $\varphi_A: A\to A$ being the Frobenius lift of the prism $(A, I)$). 
By \cite[Theorem~1.2]{BhattScholze2}, there is an equivalence of categories given by an (exact) \'{e}tale realization functor $T_{\oh_K}: \mathsf{Vect}^{\varphi}(\spf(\oh_K)_{\Prism}, \oh_{\Prism}) \stackrel{\sim}\rightarrow \Rep_{\ZZ_p, \cris}(G_K)$ where the left-hand side is the category of prismatic $F$-crystals. Denote by $D_{\Prism}$ the quasi-inverse. The Breuil-Kisin prism $(\Es, E(u))$ is naturally an object of $\spf(\oh_K)_{\Prism}$, and the functor of sections $\ev_{(\Es, (E(u)))}:\mathcal{F}\mapsto \Em =\mathcal{F}(\Es, (E(u)))$ lands in $\Mod_{\Es, \fr}^{\varphi}$. Moreover, the composition $\mathrm{ev}_{(\Es, (E(u)))}\circ D_{\Prism}$ is the functor $\underline{\Em}$ (see \cite[\S 7]{BhattScholze2}).
\begin{theorem}\label{thm:ExactnessPrismatic}
Let $r \in \mathbb{N}$. Given a short exact sequence
\begin{equation}\tag{$\tau$}
\begin{tikzcd}
0\ar[r]& T \ar[r] & L \ar[r] & \ZZ_p\ar[r] & 0
\end{tikzcd}
\end{equation}
in $\Rep_{\ZZ_p,  \cris}^{[0, r]}(G_K),$  the corresponding sequence of prismatic $F$-crystals 
\begin{equation}\tag{$\delta$}
\begin{tikzcd}
0\ar[r]& D_{\Prism}(T) \ar[r] & D_{\Prism}(L) \ar[r] & \oh_{\Prism}\ar[r] & 0
\end{tikzcd}
\end{equation}
is again short exact. Consequently, we have 
$$H^1_f(K, T)\simeq \Ext^1_{\mathsf{Vect}^{\varphi}(\spf(\oh_K)_{\Prism}, \oh_{\Prism})}(\oh_{\Prism}, D_{\Prism}(T)).$$
\end{theorem}

\begin{proof}
It is enough to show that for every prism $(A, I) \in \spf(\oh_K)_{\Prism}$, the sequence
\begin{equation}\tag{$\delta_{(A, I)}$}
\begin{tikzcd}
0 \ar[r]& D_{\Prism}(T)(A, I) \ar[r]& D_{\Prism}(L)(A, I) \ar[r]& A \ar[r]& 0 
\end{tikzcd}
\end{equation} obtained by evaluating $(\delta)$ at $(A, I)$ is a short exact sequence of $A$-modules. Note that such a sequence is then automatically split (as a sequence of $A$-modules). Since $\mathrm{ev}_{(\Es, (E(u)))}\circ D_{\Prism}=\underline{\Em}$, Theorem~\ref{thm:ExactKisinModules} and Remark~\ref{rem:ExactKisinModlesComments} (1) imply that $(\delta_{(\Es, (E(u)))})$ is exact. Let $(A, I) \in \spf(\oh_K)_{\Prism}$ be a general prism. By \cite[Example 2.6 (1)]{BhattScholze2}, there is a flat cover (given as the opposite of) $(A, I) \to (B, IB)$ in $\spf(\oh_K)_{\Prism}$ and a map of prisms $(\Es, (E(u))) \to (B, IB)$ whose opposite is a map in $\spf(\oh_K)_{\Prism}$. In particular, it follows that $(\delta_{(B, IB)})$ is exact, as it is canonically identified with the base change of the split-exact sequence $(\delta_{(\Es, (E(u)))})$. Since $(\delta_{(B, IB)})$ is, likewise, obtained from $(\delta_{(A, I)})$ by base change, the fact that $(\delta_{(A, I)})$ is exact follows by $(p, I)$-completely faithfully flat descent. 
\end{proof}

\begin{remark}\label{rem:PrismaticComments}
Several comments about the above result are in order.
\begin{enumerate}
\item In the proof, we established a seemingly stronger statement that evaluation of the sequence $(\delta)$ at every object of the site yields a short exact sequence. This is, in fact, equivalent to exactness of $(\delta)$, since the prismatic structure sheaf $\oh_{\Prism}$ (hence also every locally free $\oh_{\Prism}$-module) has vanishing higher cohomology for $\ev_{(A, I)}$ for every bounded prism $(A, I)$ by \cite[Corollary~3.12]{BhattScholze}.
\item Theorem~\ref{thm:ExactnessPrismatic} has a direct analogue in the case of semistable representations and $F$-crystals on an appropriate log-prismatic site, as in \cite[Definition~5.0.6]{DuLiuPhiGHat}. Since the proof is completely analogous, we leave the details to the reader. \label{rem:PrismaticComments2}
\end{enumerate}
\end{remark}

The canonical nature of prismatic $F$-crystals allows one to extend the exactness results to theories over different prisms. Let us show an application in terms of Wach modules, which we briefly recall (see e.g. \cite[\S III]{BergerWachM} or \cite[\S 3]{AbhinandanBlochKato} for more details).

Assume that $K$ is absolutely unramified, so that $K=K_0=W[1/p]$. The Wach prism is given by the ring $\AA=A_K^+=W[[q-1]]$ (where $q$ is a formal variable), the Frobenius lift $\varphi_{\AA}$ extending that of $W=W(k)$ by $\varphi_{\AA}(q)=q^p$, and the distinguished ideal $([p]_q)$ where $[p]_q=1+q+\dots +q^{p-1}$. It is further endowed with a $W$-linear, continuous action of $\Gamma=\mathrm{Gal}(K(\mu_{p^{\infty}})/K)$ determined by the rule $g(q)=q^{\chi(g)}$ where $\chi$ denotes the cyclotomic character. A Wach module is then a finite free $\AA$-module $M$ endowed with a semilinear $\Gamma$-action which becomes trivial on the quotient $M/(q-1)M$, and a $\Gamma$-equivariant, $\AA$-linear isomorphism $\varphi_M:\varphi_{\AA}^*M[1/[p]_q]\rightarrow M[1/[p]_q]$. Denote by $\Mod_{\AA}^{\varphi,\Gamma}$ the category of Wach modules. By a result of Berger \cite[\S III.4]{BergerWachM}, there is an equivalence between $\Mod_{\AA}^{\varphi,\Gamma}$ and $\Rep_{\ZZ_p, \cris}(G_K)$. The functor from Wach modules to Galois lattices can be taken as the functor $T_{\AA}: (M \mapsto M\otimes_{\AA}\CC_K^{\flat})^{\varphi=1}$, and just like $T_{\Es}$, it is an exact functor. Denote the quasi-inverse functor by $\mathbf{N}_{\AA}$.

Let $M$ be a Wach module and let $V=T_{\AA}(M)[1/p]$ be the corresponding crystalline representation. In \cite{AbhinandanBlochKato}, Abhinandan constructs a complex $\mathcal{S}^{\bullet}(M)$ with the property that $H^i(\mathcal{S}^{\bullet}(M))[1/p] \simeq H^i_f(K,V)$ for all $i$. As a consequence of Theorem~\ref{thm:ExactnessPrismatic}, we have:

\begin{corollary}\label{cor:WachModulesSyntomic}
Assume that $K=K_0$, i.e., that $K$ is absolutely unramified. Let $T$ be a lattice in a crystalline $G_K$-representation such that all the Hodge-Tate weights of $T[1/p]$ are non-negative. Let $M$ be a Wach module such that $T_{\AA}(M)=T$. Then the 
Wach module
syntomic complex $\mathcal{S}^{\bullet}(M)$ satisfies $H^i(\mathcal{S}^{\bullet}(M))=H^i_f(K,T)$ for $i=0, 1$.
\end{corollary}

\begin{proof}
The case $i=0$ is \cite[Lemma~4.5]{AbhinandanBlochKato}. Similarly, Proposition~4.6 and the proof of Corollary~4.7 of \cite{AbhinandanBlochKato} show that there is an isomorphism $H^1(\mathcal{S}^{\bullet}(M))\simeq \Ext^1(\AA, M)$ where the Ext is the group of extensions of Wach modules. To finish the proof, it is enough to show that $\Ext^1(\AA, M)$ is isomorphic to 
$\Ext^1_{\Rep_{\ZZ_p, \cris}(G_K)}(\ZZ_p, T)\simeq H^1_f(K, T)$.

There is a well-defined, injective map $\Ext^1(\AA, M) \to \Ext^1_{\Rep_{\ZZ_p, \cris}(G_K)}(\ZZ_p, T)$ induced by $T_{\AA}$, since this is an equivalence of categories and an exact functor. To show surjectivity, we again need to show that for a short exact sequence of the form $0 \to T \to L \to \ZZ_p \to 0$ in $\Rep_{\ZZ_p, \cris}(G_K)$, the corresponding sequence of Wach modules is also short exact.

To that end, note that by \cite[Theorem~1.2]{AbhinandanFcrystals}, the evaluation functor $\mathrm{ev}_{(\AA, ([p]_q))}: \mathcal{F} \mapsto \mathcal{F}(\AA, ([p]_q))$ induces an equivalence of categories $\mathsf{Vect}^{\varphi}(\spf(\oh_K)_{\Prism}, \oh_{\Prism}) \stackrel{\sim}\rightarrow \Mod_{\AA}^{\varphi,\Gamma},$ in a manner such that $\mathrm{ev}_{(\AA, ([p]_q))} \circ D_{\Prism}=\mathbf{N}_{\AA}$ (cf. \cite[\S 1.5]{AbhinandanFcrystals}). Strictly speaking, the equivalence is established with the category of \textit{analytic} prismatic $F$-crystals $\mathsf{Vect}^{\varphi, \mathrm{an}}(\spf(\oh_K)_{\Prism}, \oh_{\Prism})$ in order to accomodate relative setting, but there is an equivalence of categories $\mathsf{Vect}^{\varphi}(\spf(\oh_K)_{\Prism}, \oh_{\Prism})\to\mathsf{Vect}^{\varphi, \mathrm{an}}(\spf(\oh_K)_{\Prism}, \oh_{\Prism})$ compatible with \'{e}tale realization functors \cite[Proposition~3.8]{ReineckeGuo}. Now the exactness claim follows directly from Theorem~\ref{thm:ExactnessPrismatic} and Remark~\ref{rem:PrismaticComments} (1). \end{proof}

\vspace{2pt}

\subsection{An auxiliary exactness result}

In the following, we establish exactness of the quasi-inverse to the functor $\mathcal{M}_{\Es}^{(r)}: \Mod_{\Es, \fr}^{\varphi, \leq r} \rightarrow \Mod(S)^{\varphi}_r. $ That is, we show that $\mathcal{M}_{\Es}^{(r)}$ is an equivalence of exact categories.

\begin{theorem}\label{thm:KisinToBreuilIsBiExact}
Fix an integer $r$ with $0 \leq r < p-1$, and denote by $\underline{\Em}^{(r)}_{S}: \Mod(S)^{\varphi}_{r}\rightarrow \Mod_{\Es, \fr}^{\varphi, \leq r}$ the quasi-inverse to $\mathcal{M}_{\Es}^{(r)}.$ Then $\underline{\Em}^{(r)}_{S}$ is exact.
\end{theorem}

\begin{proof}
Consider a short exact sequence 
\begin{center}
\begin{tikzcd}
0\ar[r]& M_1 \ar[r] & M_2 \ar[r] & M_3 \ar[r] & 0
\end{tikzcd}
\end{center}
in $\Mod(S)^{\varphi}_{r}$. To apply $\underline{\Em}^{(r)}_{S}$ is to apply the \'{e}tale realization $T_{\cris, \star}$ followed by $\underline{\Em}^{(r)}$ (up to natural isomorphism). Since $T_{\cris, \star}$ is exact and $\underline{\Em}^{(r)}$ is left exact, we conclude that the sequence 

\begin{center}
\begin{tikzcd}
0\ar[r]& \underline{\Em}^{(r)}_{S}(M_1) \ar[r] & \underline{\Em}^{(r)}_{S}(M_2) \ar[r] & \underline{\Em}^{(r)}_{S}(M_3)
\end{tikzcd}
\end{center} 
is left exact as a sequence of $\Es$-modules. Suppose that it is not exact, i.e., the sequence can be extended to an exact sequence of Kisin modules 
\begin{center}
\begin{tikzcd}
0\ar[r]& \underline{\Em}^{(r)}_{S}(M_1) \ar[r] & \underline{\Em}^{(r)}_{S}(M_2) \ar[r] & \underline{\Em}^{(r)}_{S}(M_3)\ar[r] & \mathfrak{C} \ar[r] & 0
\end{tikzcd}
\end{center} 
with $\mathfrak{C}\neq 0$ finitely generated (and, in fact, torsion). Applying the functor $S\otimes_{\Es}\varphi^*(\cdot)$ yields a complex of $S$-modules
\begin{center}
\begin{tikzcd}
0\ar[r]& U(M_1) \ar[r] & U(M_2) \ar[r] & U(M_3)\ar[r] & S\otimes_{\Es}\varphi^*\mathfrak{C} \ar[r] & 0
\end{tikzcd}
\end{center} 
which is exact on the right (as any tensor product), i.e., the sequence of $S$-modules 
\begin{center}
\begin{tikzcd}
U(M_2) \ar[r] & U(M_3)\ar[r] & S\otimes_{\Es}\varphi^*\mathfrak{C} \ar[r] & 0
\end{tikzcd}
\end{center} 
is exact. Moreover, $S \otimes_{\Es} \varphi^* \mathfrak{C} \neq 0$ by Remark~\ref{rem:SNotFlat}. This is a contradiction, since the first three terms of the complex come from the original short exact sequence by applying the forgetful functor, and the map $U(M_2)\rightarrow U(M_3)$ is therefore surjective.
\end{proof}

The above theorem implies that the category $\Mod(S)_r^\varphi$ is exact. Indeed, $\underline{\Em}_{S}^{(r)}$ is an exact equivalence, and we already established that the category $\Mod_{\Es, \fr}^{\varphi, \leq r}$ is exact (see \ref{rem:Kisinexactcat}). One can verify easily that the category $\Mod(S)_r^{\varphi, N}$ is also exact by verifying the Axioms of an exact category (cf.~\cite[\S A.3]{Positselski}. The only thing one needs to check is that $\Mod(S)_r^{\varphi, N}$ is closed under push-outs of admissible monos and pullbacks of admissible epis.

\vspace{2pt}

\section{Applications}\label{sec:Applications}
In this last section we discuss some applications of our exactness theorems \ref{thm:ExactKisinModules}, \ref{Th:isomorphism_ext} and \ref{thm:KisinToBreuilIsBiExact}.

\subsection{Tensor Product of strongly divisible modules}

We use the connection of Breuil modules and Kisin modules to give a notion of tensor products of Breuil modules. We introduce this operation indirectly, in a way that makes it well-behaved but, on the other hand, makes the construction somewhat inexplicit.

\begin{definition}\label{def:TensorProductQuasiSDMod}
Consider a pair of integers $r_1, r_2 \geq 0$ with $r_1+r_2<p-1.$ 
The tensor product functor $\otimes: \Mod(S)_{r_1}^{\varphi} \times \Mod(S)_{r_2}^{\varphi}\rightarrow \Mod(S)_{r_1+r_2}^{\varphi}$ is defined by the prescription
$$M_1\otimes M_2=\mathcal{M}_{\Es}(\underline{\Em}^{(r_1)}_S(M_1)\otimes \underline{\Em}^{(r_2)}_S(M_2)), \;\; M_i \in \Mod(S)_{r_i}^{\varphi},\; i=1,2.$$
\end{definition}

\begin{remark}
The above definition is clearly the ``correct one'' in the sense that for all $M_1, M_2$ as above, one has 
\begin{align*}
T_{\cris, \star}(M_1 \otimes M_2)& = T_{\Es}^{(r_1+r_2)}(\underline{\Em}^{(r_1)}_{S}(M_1)\otimes \underline{\Em}^{(r_2)}_{S}(M_2))\\
&= T_{\Es}^{(r_1)}(\underline{\Em}^{(r_1)}_{S}(M_1))\otimes T_{\Es}^{(r_2)}(\underline{\Em}^{(r_2)}_{S}(M_2))
=T_1 \otimes_{\ZZ_p} T_2
\end{align*}
where $T_i=T_{\cris, \star}(M_i), \; i=1,2.$ That is, $\otimes$ is compatible with \'{e}tale realization. Notice that the second equality follows by Proposition \ref{prop:PropertiesOfT}.

It is easy to see that the underlying $S$-module of $M_1 \otimes M_2$ is $M_1 \otimes_S M_2$. However, it does not seems obvious how to describe $\Fil^{r_1+r_2}(M_1\otimes M_2)$ directly in terms of $\Fil^{r_i}M_i$'s. Clearly it is a submodule containing the image of $\Fil^{r_1}M_2\otimes_{S}\Fil^{r_2}M_2$ and $\Fil^{r_1+r_2}(S)\cdot M_1\otimes M_2$, but it is not even clear whether $\Fil^r$ is the sum of these two submodules in general (e.g., due to $p$-saturatednesss considerations). Similarly, the operator $\varphi_r$ remains somewhat inexplicit, although it is easy to see how it necessarily behaves on the above-mentioned sum.\end{remark}

\begin{definition}\label{def:TensorProductSDMod}
Consider a pair of integers $r_1, r_2 \geq 0$ with $r_1+r_2<p-1.$ 
The tensor product functor $\otimes: \Mod(S)_{r_1, \log}^{\varphi, N} \times \Mod(S)_{r_2, \log}^{\varphi, N}\rightarrow \Mod(S)_{r_1+r_2, \log}^{\varphi, N}$ is defined as follows.
\ Given $(M_i, \Fil^{r_i}M_i, \varphi_{r_i}, N_i) \in \Mod_{r_i, \log}^{\varphi, N}(S),\; i=1,2,$ we set
$$M_1 \otimes M_2=\mathcal{M}_{\st}(T_1\otimes_{\ZZ_p} T_2),$$
where $T_i=T_{\st, \star}(M_i),\; i=1, 2.$
\end{definition}

\begin{remark}\label{rem:TensorProductCompatibility}
  For $M_1, M_2$ as in Definition~\ref{def:TensorProductSDMod}, we have the natural isomorphisms
 \begin{align*}
 T_{\cris, \star}(U(M_1\otimes M_2))&=T_{\st, \star}(M_1\otimes M_2)|_{G_{\infty}}=(T_1\otimes T_2)_{G_{\infty}}\\ &=T_1|_{G_{\infty}}\otimes_{\ZZ_p}T_2|_{G_{\infty}}=T_{\cris, \star}(U(M_1)\otimes U(M_2)),
 \end{align*}
 where $U$ is the functor forgetting the monodromy operator, and $U(M_1)\otimes U(M_2)$ denotes the tensor product in $\Mod(S)^{\varphi}_{r_1+r_2}$ as defined in \ref{def:TensorProductQuasiSDMod}. Since $T_{\cris, \star}$ is fully faithful, this implies that $U(M_1\otimes M_2)$ is naturally isomorphic to $U(M_1)\otimes U(M_2)$.

 In other words, Definition~\ref{def:TensorProductSDMod} supplements Definition~\ref{def:TensorProductQuasiSDMod} by specifying a monodromy operator on $U(M_1)\otimes U(M_2)$ in a manner compatible with \'{e}tale realization (to lattices with full $G_K$-action).

We note that if $N_i$ denotes the monodromy operator on $M_i,$ $i=1,2,$ then the monodromy operator $N$ on $M_1\otimes M_2$ is given by $N_1 \otimes 1+1\otimes N_2,$ as expected. This can be shown by a similar functor-yoga via the categories of filtered $(\varphi, N)$-modules over $K_0$ and $S_{K_0}$. We will not need this fact, and so we leave the details to the reader.  
\end{remark}

\begin{remark}
Recall that a strongly divisible module $M$ is called \textit{filtered free} if $\Fil^rM$ has the special form $$\Fil^rM=\bigoplus_i SE(u)^{t_i}e_i+\Fil^r(S)M$$ for some $t_i \in \{ 1, 2, \dots r \},$ where $\{e_i\}_i$ is a free basis of $M$. In \cite[\S~5.4]{GazakiTateDuality}, the second named author defined tensor product of Breuil modules in the case that both modules were filtered free, using the ``natural'' definition of filtrations and operations. In particular, the filtration on $M_1\otimes M_2$ is given by the sum of the image of $\Fil^{r_1}M_1 \otimes_S \Fil^{r_2}M_2$ and $\Fil^{r_1+r_2}(S)\cdot M_1 \otimes_S M_2$. It can be shown that this version of tensor product coincides with the one given in Definition~\ref{def:TensorProductSDMod} for filtered free strongly divisibile modules by viewing them as lattices in filtered $(\varphi, N)$-modules over $S_{K_0}$. In fact, it can be shown that this is the case when we assume that only one of the strongly divisible modules $M_1, M_2$ is filtered free. In this article Definition \ref{def:TensorProductSDMod} will be enough for our purposes, so we do not include a proof of this fact.
\end{remark}

\begin{theorem}\label{thm:tensorprodexact}
The tensor product of strongly divisible modules is exact. That is, suppose that $r_1, r_2$ are two integers with $0 \leq r_1, r_2 \leq r_1+r_2<p-1$, and consider $M_1 \in \Mod(S)^{\varphi, N}_{r_1, \log}$. Then for every short exact sequence $$0 \to M_2' \to M_2 \to M_2'' \to 0$$ in $\Mod(S)^{\varphi, N}_{r_2, \log}$, the sequence $$0 \to M_1 \otimes M_2' \to M_1 \otimes M_2 \to M_1 \otimes M_2'' \to 0$$
is a short exact sequence of strongly divisible modules. (The same holds for the version without monodromy operators.)
\end{theorem}

\begin{proof}
The forgetful functor $U: \Mod(S)_{r_1+r_2, \log}^{\varphi, N}\to \Mod(S)_{r_1+r_2}^{\varphi}$ preserves and reflects exactness, so in view of Remark~\ref{rem:TensorProductCompatibility}, it is enough to consider the version without monodromy operators. Now the Theorem is a consequence of the obvious fact that the tensor product of free Kisin modules is exact, together with exactness of $\mathcal{M}_{\Es}^{(r)}$ and its quasi-inverse (Theorem~\ref{thm:KisinToBreuilIsBiExact}). 
\end{proof}

\subsection{An application to Abelian Varieties} In this section we will prove Theorem \ref{thm:abvarsintro} as stated in the introduction. We start with a homological algebra construction.

\subsubsection{A cup product structure on $\Ext$ groups}
In this section we fix integers $r_1, r_2\geq 0$ such that $r_1+r_2<p-1$. To avoid confusion with the different weights we will denote by $\mathbf{1}_{r_i}$ the quadruple $(S, S, \phi, N)$ viewed as an object of $\Mod(S)_{r_i}^{\varphi, N}$. 

\begin{lemma}
   Using the definition of tensor product of strongly divisible modules given in \ref{def:TensorProductSDMod} there is an equality $\mathbf{1}_{r_1}\otimes \mathbf{1}_{r_2}=\mathbf{1}_{r_1+r_2}$. 
\end{lemma}
\begin{proof}
This is a consequence of the fact that $\ZZ_p\otimes\ZZ_p=\ZZ_p$ in $\Rep_{\ZZ_p, \st}^{[0,r_1+r_2]}(G_K)$. Alternatively, recall that $\Fil^{r_1+r_2}(\mathbf{1}_{r_1}\otimes \mathbf{1}_{r_2})$ contains $\Fil^{r_1}(\mathbf{1}_{r_1})\otimes_S \Fil^{r_2}(\mathbf{1}_{r_2})=S\otimes_S S=S$ and the Frobenius on this piece is defined diagonally. 
\end{proof}

The goal of this section is to prove the following theorem. 
\begin{theorem}\label{thm:cupproduct}
  For $i=1,2$ let $M_i\in \Mod(S)_{r_i}^{\varphi, N}$ be strongly divisible modules. There is a cup product structure 
\[\Ext^1_{\Mod(S)_{r_1}^{\varphi, N}}(\mathbf{1}_{r_1},M_1)\otimes \Ext^1_{\Mod(S)_{r_2}^{\varphi, N}}(\mathbf{1}_{r_1},M_2)\to\Ext^2_{\Mod(S)_{r_1+r_2}^{\varphi, N}}(\mathbf{1}_{r_1+r_2},M_1\otimes M_2),\] where $M_1\otimes M_2$ is the tensor product defined in \ref{def:TensorProductSDMod}.  
\end{theorem}
Here $\{\Ext^j_{\Mod(S)_{r_i}^{\varphi, N}}\}_{j\geq 0}$ are the Yoneda $\Ext$ functors in the exact category $\Mod(S)_{r_i}^{\varphi, N}$. We refer to \cite[Proposition A.13]{Positselski} for details. Morally, most standard properties of $\Ext$-groups in an abelian category carry through to the exact category setting if one imposes the correct definitions. We will only recall those properties needed to us, omitting some of the technical details.  

For us an exact category $\mathcal{C}$ is an additive category equipped with short exact sequences of objects (often called \textit{admissible triples}) $0\to X'\to X\to X''\to 0$ verifying the Axioms described in \cite[\S A.3]{Positselski}. 

\begin{definition}\label{def:exactseq} (cf.~\cite[\S A.7]{Positselski})
    Let $\mathcal{C}$ be an exact category. A complex $X^\bullet=\{ \cdots \to X^{i-1}\xrightarrow{d^{i-1}} X^i\xrightarrow{d^i} X^{i+1}\to\cdots\}$ in $\mathcal{C}$ is said to be exact if there exist objects $Z^i$ in $\mathcal{C}$ and short exact sequences $0\to Z^i\to X^i\to Z^{i+1}\to 0$ such that for every $i$ the differential $X^i\xrightarrow{d^i} X^{i+1}$ is equal to the composition $X^i\to Z^{i+1}\to X^{i+1}$. 
\end{definition}

\begin{remark}
 Applying the above definition to the category $\Mod(S)_r^{\varphi, N}$, we see that a complex    $\cdots\to  M_{i-1}\xrightarrow{d^{i-1}} M_i\xrightarrow{d^i} M_{i+1}\to\cdots$ of strongly divisible modules is exact in $\Mod(S)^{\varphi, N}_r$ if  and only if  $\ker(d^i), \img(d^i)$ lie in the category for all $i$ and satisfy $\img(d^{i-1})=\ker(d^i)$, and additionally, we have short exact sequences 
 $0\to\Fil^r(\ker(d^i))\to \Fil^r M_i\to\Fil^r(\img(d^i))\to 0$ for all $i$. 
\end{remark}

\subsection*{Composition of extensions} 
For the construction of cup product we will need the notion of composition of extensions, which we recall below.

\begin{definition} (cf. \cite[\S7, 4.Prop.~3~pg.~118]{BourAlg10} or \cite[Chapter III, \S5 (5.1) pg.~82]{MacLane})
    Let $\mathcal{C}$ be an exact category. Let $A, B, C\in\mathcal{C}$. Then there is a well-defined \textit{composition} of extensions \[\Ext_\mathcal{C}^n(C,A)\times \Ext^m_\mathcal{C}(B,C)\to \Ext^{n+m}_\mathcal{C}(B, A),\]
defined as follows. Let $\alpha: 0\to A\to X_1\to\cdots \to X_n\to C\to 0$ and $\beta: 0\to C\to Y_1\to\cdots\to Y_{m}\to B\to 0$ be (classes of) exact sequences. We abbreviate notation denoting $\alpha: 0\to A\to \underline{X}\to C\to 0$, $\beta: 0\to C\to \underline{Y}\to B\to 0$.  We consider the concatenation 

\[
\begin{tikzcd}[column sep = small]
0 \ar[rr]& & A \ar[rr] && \underline{X} \ar[dr] && \underline{Y} \ar[rr] && B \ar[rr]&&0\\
&&&&& C \ar[dr]\ar[ur] &&&&&\\
&&&& 0 \ar[ur]&& 0 &&&&\\
\end{tikzcd}
\]

Then $\alpha\circ\beta$ is to defined to be the class of the $n+m$-extension 
\[\alpha\circ\beta: 0\to A\to X_1\to \cdots\to X_n\to Y_1\to\cdots\to  Y_m\to B\to 0.\]
\end{definition}

Combining this with Definition \ref{def:exactseq} we see that a sequence $0\to X_1\to\cdots\to X_n\to 0$ in $\mathcal{C}$ is exact if we can write it as a composition of $1$-extensions. 

\begin{proof} (of Theorem \ref{thm:cupproduct})
We consider (classes of) $1$-extensions, \begin{eqnarray*}
    \alpha: && 0\to M_1\xrightarrow{\iota_1} X_1\xrightarrow{g_1} \mathbf{1}_{r_1}\to 0,\\
    \beta: && 0\to M_2\xrightarrow{\iota_2} X_2\xrightarrow{g_2} \mathbf{1}_{r_2}\to 0
\end{eqnarray*}
 in $\Mod(S)_{r_1}^{\varphi, N}$ and $\Mod(S)_{r_2}^{\varphi, N}$ respectively. It follows by Proposition \ref{thm:tensorprodexact} that the sequences
\begin{eqnarray*}
   \alpha\otimes M_2:  && 0\to M_1\otimes M_2\xrightarrow{\iota_1\otimes 1} X_1\otimes M_2\xrightarrow{g_1\otimes 1} \mathbf{1}_{r_1}\otimes M_2\to 0\\
   \mathbf{1}_{r_1}\otimes \beta: && 0\to \mathbf{1}_{r_1}\otimes M_2\xrightarrow{1\otimes \iota_2} \mathbf{1}_{r_1}\otimes X_2\xrightarrow{1\otimes g_2}\mathbf{1}_{r_1}\otimes \mathbf{1}_{r_2}\to 0   \end{eqnarray*}
are exact in $\Mod(S)_{r_1+r_2}^{\varphi, N}$. Then we define $\alpha\cup\beta$ to be the composition $(\alpha\otimes M_2)\circ(\mathbf{1}_{r_1}\otimes \beta)$. That is, the class of the $2$-extension 
\[\alpha\cup\beta:  0\to M_1\otimes M_2\xrightarrow{\iota_1\otimes 1} X_1\otimes M_2\xrightarrow{(1\otimes\iota_2)\circ(g_1\otimes 1)} \mathbf{1}_{r_1}\otimes X_2\xrightarrow{1\otimes g_2}\mathbf{1}_{r_1+r_2}\to 0.\]
\end{proof}

\begin{remark}
One needs to check that the above construction is: 

(a) well-defined, i.e. it does not depend on the choice of representatives for $\alpha, \beta$; 

(b)  bilinear with respect to the Baer sum.

Both these are routine checks by adjusting the analogous results for abelian to exact categories.  To give a flavor of this type of proof, see the proof of the following Lemma \ref{lem:welldefinedcup}. 
\end{remark}

Note that we could have defined $\alpha\cup\beta$ to be instead the class of the composition $(M_1\otimes\beta)\circ(\alpha\otimes\mathbf{1}_{r_2})$. The following lemma takes care of this ambiguity. 

\begin{lemma}\label{lem:welldefinedcup}
    Consider the set-up in the proof of Theorem \ref{thm:cupproduct}. The $2$-extensions $(M_1\otimes\beta)\circ\alpha$ and $(\alpha\otimes M_2)\circ \beta$ are equivalent in $\Mod(S)_{r_1+r_2}^{\varphi, N}$.
\end{lemma}
\begin{proof}
    We recall (cf.~\cite[Proposition A.13]{Positselski}) that two $2$-extensions $0\to A\to X_1\to X_2\to B\to 0$ and $A\to Y_1\to Y_2\to B\to 0$ in an exact category $\mathcal{C}$ are Yoneda equivalent if there exists another $2$-extension $ 0\to A\to Z_1\to Z_2\to B\to 0$ in $\mathcal{C}$ such that we have a commutative diagram 
    \[
\begin{tikzcd}[column sep = small]
0 \ar{r} & A\ar{r} \ar{d}{=} & X_1 \ar{r} \ar{d} & X_2 \ar{r} \ar{d} & B \ar{r} \ar{d}{=} & 0 \\
0 \ar{r} & A\ar{r} & Z_1 \ar{r} & Z_2 \ar{r} & B \ar{r} & 0 \\
0 \ar{r} & A\ar{r} \ar{u}{=} & Y_1 \ar{r} \ar{u} & Y_2 \ar{r} \ar{u} & B \ar{r} \ar{u}{=} & 0.
\end{tikzcd} 
\] 
We will construct such a diagram for the $2$-extensions $(M_1\otimes\beta)\circ\alpha$ and $(\alpha\otimes M_2)\circ \beta$. To abbreviate notation we will denote by $X_1(r_2)$ the strongly divisible module $X_1\otimes\mathbf{1}_{r_2}$, and similarly $X_2(r_1)$. We also relabel the morphisms as follows: $0\to M_1\otimes M_2\xrightarrow{\lambda_1}X_1\otimes M_2\xrightarrow{f_1}X_2(r_1)\xrightarrow{\gamma_1}\mathbf{1}_{r_1+r_2}\to 0$ and similarly for the extension $(\alpha\otimes M_2)\circ \beta$. We consider the diagram: 

\[\begin{tikzcd}
0 \ar{r} & M_1\otimes M_2\ar{r}{\lambda_1} \ar{d}{=} & X_1\otimes M_2 \ar{r}{f_1} \ar{d}{j_1} &  X_2(r_1) \ar{r}{\gamma_1} \ar{d}{(0,1)}  & \mathbf{1}_{r_1+r_2} \ar{r} \ar{d}{=} & 0 \\
0 \ar{r} & M_1\otimes M_2\ar{r}   & X_1\otimes X_2 \ar{r}{\mu}  & X_1(r_2)\oplus X_2(r_1)\ar{r}{\gamma_2\oplus\gamma_1}   & \mathbf{1}_{r_1+r_2} \ar{r}   & 0 \\
0 \ar{r} & M_1\otimes M_2\ar{r}{\lambda_2} \ar{u}{=} & M_1\otimes X_2 \ar{r}{f_2} \ar{u}{j_2}  & X_1(r_2) \ar{r}{\gamma_2} \ar{u}{(1,0)} & \mathbf{1}_{r_1+r_2} \ar{r} \ar{u}{=} & 0,
\end{tikzcd} 
\] where $j_1$ (resp. $j_2$) is the inclusion $X_1\otimes M_2\hookrightarrow X_1\otimes X_2$ (resp. $M_1\otimes X_2\hookrightarrow X_1\otimes X_2$). The map $\mu=(\mu_1,\mu_2)$ is given coordinate-wise by $\mu_1(x_1\otimes x_2)=g_2(x_2)\cdot x_1$, and $\mu_2(x_1\otimes x_2)=g_1(x_1)\cdot x_2$. To make this definition precise, recall that $g_1: X_1\to S$ is $S$-linear map. Since $X_2$ is an $S$-module, the definition $\mu_2(x_1\otimes x_2)=g_1(x_1)\cdot x_2$ makes sense, and similarly for $\mu_1$. 

It is easy to check the commutativity of the diagram. Additionally, the homomorphisms $\mu, \gamma_2\oplus\gamma_1$ are filtered. To see this, note that $X_1(r_2)$ is precisely the quadruple $(\tilde{X}_2,\Fil^{r_1+r_2}\tilde{X}_2,\tilde{\varphi}_{r_1+r_2},\tilde{N})$, where $\tilde{X}_2=X_2$, $\Fil^{r_1+r_2}\tilde{X}_2=\Fil^{r_2}X_2$, $\tilde{\varphi}_{r_1+r_2}=\varphi_{r_2}$ on $\Fil^{r_2}X_2$ etc. 

It remains to check that the sequence $0\to M_1\otimes M_2\otimes X_1\otimes X_2\otimes X_1(r_2)\oplus X_2(r_1)\to \mathbf{1}_{r_1+r_2}$ is exact in $\Mod(S)_{r_1+r_2}^{\varphi, N}$. This follows because it is the composition of the short exact sequences $0\to M_1\otimes M_2\to X_1\otimes X_2\to X_1(r_2)\oplus M_2(r_1)\to 0$ and $0\to X_1(r_2)\oplus M_2(r_1)\to 0\to X_1(r_2)\oplus X_2(r_1)\to\mathbf{1}_{r_1+r_2}\to 0$. 
\end{proof}

\subsubsection{The cup product map on abelian varieties}
Let $A_1, A_2$ be abelian varieties over $K$ with good reduction with $p$-adic Tate modules $T_1=T_p(A_1)$ and $T_2=T_p(A_2)$. Note that $T_1, T_2$ are objects of $\Rep_{\ZZ_p,\cris}^{[0,1]}(G_K)$. For $i=1,2$ let $\mathcal{M}_{\st}(T_i)$ be the corresponding strongly divisible module, an object of $\Mod(S)_{1}^{\varphi, N}$. 

For every $n\geq 1$ the Kummer sequence for $A_i$ gives a connecting homomorphism 
$\delta_i: A_i(K)/p^n\to H^1(K, A_i[p^n])$, which induce homomorphism 
\[\delta_i:A(K)\to H^1(K, T_i).\]
It follows by \cite[p.~352]{BlochKato} that the image of $\delta_i$ is precisely $H^1_f(K, T_i)$. The results of the previous sections yield the following Corollary.

\begin{corollary}\label{cor:abelianvars}
    The homomorphism $\delta_1\cup\delta_2: A_1(K)\otimes A_2(K)\to H^2(K, T_1\otimes T_2)$ induced by the cup product factors through $\Ext^2_{\Mod(S)_{2}^{\varphi, N}}(\mathbf{1}_2,\mathcal{M}_{\st}(T_1)\otimes \mathcal{M}_{\st}(T_2))$. 
\end{corollary}

\begin{proof}
    Theorem \ref{Th:isomorphism_ext} gives isomorphisms $H^1_f(K, T_i)\simeq \Ext^1_{\Mod(S)_{1}^{\varphi, N}}(\mathbf{1}_1, \mathcal{M}_{\st}(T_i))$ for $i=1,2$.  We thus have a commutative diagram

\[
\begin{tikzcd}[column sep = small]
A_1(K)\otimes A_2(K) \ar{d}{\delta_1\otimes\delta_2} \ar{r}{\delta_1\otimes\delta_2} & H^1_f(K, T_1)\otimes H^1_f(K, T_2)\ar{d}{\simeq}\\
H^1(K,T_1)\otimes H^1(K,T_2) \ar{d}{\cup} & \Ext^1_{\Mod(S)_{1}^{\varphi, N}}(\mathbf{1}_1, \mathcal{M}_{\st}(T_1))\otimes \Ext^1_{\Mod(S)_{1}^{\varphi, N}}(\mathbf{1}_1, \mathcal{M}_{\st}(T_2))\ar{d}{\cup}\\
H^2(K, T_1\otimes T_2) \ar{d}{=} & \Ext^2_{\Mod(S)_{2}^{\varphi, N}}(\mathbf{1}_2, \mathcal{M}_{\st}(T_1)\otimes \mathcal{M}_{\st}(T_2))\ar{d}{\simeq}\\
H^2(K, T_1\otimes T_2) \ar{d}{=} & \Ext^2_{\Mod(S)_{2}^{\varphi, N}}(\mathbf{1}_2, \mathcal{M}_{\st}(T_1\otimes T_2))\ar{d}{T_{\st,\star}}\\
\Ext^2_{\Rep_{\ZZ_p}(G_K)}(\ZZ_p, T_1\otimes T_2)   & \ar{l}\Ext^2_{\Rep_{\ZZ_p,\cris}^{[0,2]}(G_K)}(\ZZ_p, T_1\otimes T_2), 
\end{tikzcd} 
\] 
where the map $\cup$ on $\Ext$ groups is the cup product defined in Theorem \ref{thm:cupproduct}. The bottom horizontal map is induced by the fully faithful exact inclusion of categories $\Rep_{\ZZ_p,\cris}(G_K)\hookrightarrow \Rep_{\ZZ_p}(G_K)$. 
The commutativity of the diagram follows because for a $\ZZ_p[G_K]$-representation $T$ we have isomorphisms $H^i(K, T)\simeq\Ext^i_{\Rep_{\ZZ_p}(G_K)}(\ZZ_p, T)$ and using this description of Galois cohomology the cup product $H^1(K,T_1)\otimes H^1(K,T_2)\to H^2(K, T_1\otimes T_2)$ is defined analogously to \ref{thm:cupproduct}. 
     \end{proof}

\begin{remark}\label{rem:cyclemap}

As mentioned in the introduction, we don't expect an analogue of Theorem \ref{thm:MainThmBreuil} to hold for $2$-extensions. In particular for a semistable $\ZZ_p$-lattice $T$ with Hodge-Tate weights in $[0,r]$, we expect that the map $\Ext^2_{\Mod(S)_{r,\log}^{\varphi, N}} (\mathbf{1},\mathcal{M}_{\st}(T))\to\Ext^2_{\Rep_{\ZZ_p,\st}^{[0,r]}(G_K)}(\ZZ_p, T)$ might fail to be surjective, even though we can verify it is injective. 
One reason that Corollary \ref{cor:abelianvars} is interesting is because unlike the case of $1$-extensions, $\Ext^2_{\Rep_{\ZZ_p,\cris}(G_K)}(\ZZ_p, -)$ might not be a subfunctor of $H^2(K, -)$. 
Indeed, by \cite[Corollary~A.14]{Positselski},  a $2$--extension 
\[
\begin{tikzcd}[column sep = small]
0 \ar[rr]& & A \ar[rr, "\alpha"] && X \ar[rr] && Y \ar[rr, "\beta"] && B \ar[rr]&&0\end{tikzcd}
\]
in an exact category $\mathcal{C}$ is in the trivial class of $\Ext^2_{\mathcal{C}}(B, A)$ if and only if the diagram above can be extended to a diagram
\[
\begin{tikzcd}[column sep = small]
&&&&& Z \ar[rd, "b"] &&&&&\\ 
0 \ar[rr]& & A \ar[rr, "\alpha"] && X \ar[rr]\ar[ru, "a"] && Y \ar[rr, "\beta"] && B \ar[rr]&&0\end{tikzcd}
\]
where $a$ is an admissible monomorphism and $b$ is an admissible epimorphism, such that the sequences $0\rightarrow A\stackrel{a\alpha}\rightarrow Z\stackrel{b}\rightarrow Y\rightarrow 0$ and $0\rightarrow X\stackrel{a}\rightarrow Z\stackrel{\beta b}\rightarrow B\rightarrow 0$ are both short exact in $\mathcal{C}$. The fact that crystalline representations are not closed under extensions therefore suggests existence of non-trivial crystalline $2$-extensions that become trivial as extensions of Galois representations. A somewhat distinct, but likely closely related, discussion expressed in the language of prismatic $F$-gauges appears in \cite[\S 6.7]{BhattFGauge}. 

As mentioned in the introduction, we suspect that a counterexample to injectivity might specifically arise in the case when $T_1, T_2$ are Tate modules of elliptic curves with good supersingular reduction, and we hope to use Corollary~\ref{cor:abelianvars} as a first step toward proving this.

\end{remark}

\appendix
\section{An exactness property for $F$-gauges}\label{appendix}

In this appendix, we improve Theorem~\ref{thm:MainThmPrismatic} to the case of reflexive coherent prismatic $F$-gauges \cite{BhattFGauge}. This answers positively a question posed to us by Bhargav Bhatt. 

\subsection{Exactness of filtered Breuil--Kisin modules}\label{subsec:KisinFilteredArgument}
To get the exactness result at the level of $F$-gauges, we again establish a Kisin-module version first. For that purpose, the following definition is needed.

\begin{definition}\label{def:NygaardFiltration}
Let $(\Em, \varphi_{\Em})$ be a free Kisin module. The \textit{Nygaard filtration} $\Fil^i \varphi^*\Em \subseteq \varphi^*\Em$ is defined as follows:

$$\Fil^i \varphi^*\Em=\{x \in \varphi^*\Em\;|\;\varphi_{\Em, \lin}(x)\in E(u)^i\Em\},\;\;\; i \in \ZZ\,.$$
\end{definition}

It will be convenient to identify $\varphi^* \Em$ with an $\Es$-submodule $\widetilde{\Em}=\varphi_{\Em, \lin}(\varphi^* \Em)$ of $\Em[1/E(u)]$ via the injective map $\varphi_{\Em, \lin}: \varphi^*\Em \to \Em[1/E(u)]$. Under this identification, $\Fil^i \varphi^*\Em$ corresponds to the intersection $\Ef^i \Em=\widetilde{\Em} \cap E(u)^i\Em \subseteq \Em$. In particular, given a Kisin module $\Em$ corresponding to a Galois lattice with Hodge--Tate weights in $[-b,-a]$, we have
$$E(u)^{a}\Em \supseteq \widetilde{\Em}\supseteq E(u)^{b}\Em$$ and consequently,  we have the following: when $i \leq a$, we have $\widetilde{\Em}\subseteq E(u)^i\Em$ and hence $\Ef^i\Em=\widetilde{\Em}$ (i.e., $\Fil^i \varphi^*\Em=\varphi^*\Em$). Conversely, if $i \geq b$, one has $\widetilde{\Em}\supseteq E(u)^i\Em$. Under the identification of $\varphi^*\Em$ with $\widetilde{\Em}$ as above, this translates to the equality $\Ef^i\Em=E(u)^i\Em$. Note that these identifications of filtered $\Es$-modules $(\varphi^*\Em, \{\Fil^i\varphi^*\Em\}_i)\leftrightarrow (\widetilde{\Em}, \{\Ef^i\Em\}_i)$ are functorial (in maps of Kisin modules). 

\begin{theorem}\label{thm:FGaugeKisinExactness}
For a short exact sequence of crystalline lattices with non-negative Hodge--Tate weights
\begin{equation}\tag{$\tau$}
\begin{tikzcd}
0\ar[r]& T \ar[r] & L \ar[r] & \ZZ_p \ar[r] & 0,
\end{tikzcd}
\end{equation}
 the associated sequence of Kisin modules
\begin{equation}\tag{$\mu$}
\begin{tikzcd}
0\ar[r]& \underline{\Em}(T) \ar[r] & \underline{\Em}(L) \ar[r] & \Es \ar[r] & 0,
\end{tikzcd}
\end{equation}
as well as the sequences (for all $i \in \ZZ$)
\begin{equation}\tag{$\phi_i$}
\begin{tikzcd}
0\ar[r]& \Fil^i \varphi^*\underline{\Em}(T) \ar[r] & \Fil^i \varphi^*\underline{\Em}(L) \ar[r] & \Fil^i \varphi^*\Es \ar[r] & 0
\end{tikzcd}
\end{equation}
are short exact.
\end{theorem}

\begin{proof}
The sequence ($\mu$) is exact by Theorem~\ref{thm:MainThmKisin}. We also know that the sequences  ($E(u)^i\mu$) and ($\varphi^* \mu$) (hence, also ($\widetilde{\mu}$), i.e., the sequence obtained from ($\mu$) by applying the operation $\Em\mapsto \widetilde{\Em}$ termwise) are short exact since ($\mu$) is split-exact as a sequence of abstract $\Es$-modules. Consequently, the sequences $(\Ef^i \mu)$, hence also $ (\phi_i)$, are left exact for all $i \in \ZZ$ due to the intersection description of $\Ef^i\Em$'s. To prove exactness, it thus remains to prove surjectivity of the map on the right for all $i \in \ZZ$. Since all the Galois lattices have weights in the range $[0, r]$ for some $r \geq 0$ (equal to the maximum Hodge--Tate weight of $T$, and also of $L$), we immediately obtain this exactness in the case when $i<-r$ or $i\geq 0$, as this amounts to exactness of ($\varphi^* \mu$) and of ($E(u)^i\mu$), respectively.

It remains to treat the case when $-r\leq i<0$. In this case, we still have $\Fil^i \varphi^*\Es=\varphi^*\Es$, i.e., the equality $\Ef^i \Es=\widetilde{\Es}=\Es$. Due to the Hodge--Tate weights of $L$ being non-negative, $\widetilde{\underline{\Em}(L)}$ contains $\underline{\Em}(L)$ and hence, so does $\Ef^i \underline{\Em}(L)=\widetilde{\underline{\Em}(L)} \cap E(u)^i \underline{\Em}(L)$. Since the map $\underline{\Em}(L)\to \Es$ is surjective, the same is true of the map $\Ef^i \underline{\Em}(L)\to \Es= \Ef^i \Es$, which finishes the proof.
\end{proof}

\subsection{Consequences for $F$-gauges}

In order to reinterpret Theorem~\ref{thm:FGaugeKisinExactness}, we need to first recall and somewhat reassemble  background material on prismatic $F$-gauges to the form suitable for our purposes. To do this, we follow \cite{BhattFGauge}, \cite{Pentland} and \cite{DatPham}. In particular, everything in this section, except for the final conclusions (Theorem~\ref{thm:ExactnessFGauges} and Corollary~\ref{cor:BlochKatoBySyntomic}), is known to experts and jointly covered by \cite{BhattFGauge, Pentland, DatPham, BhattLurie2}.

\subsubsection{Coherent $F$-gauges} Recall that a \textit{prismatic $F$-gauge on $\oh_K$} is a quasi-coherent sheaf (or, more generally, complex of such sheaves) on the syntomification of $\spf(\oh_K)$, the formal stack $\oh_K^{\syn}$ due to Drinfel'd \cite{DrinfeldPrismatization}, further developed by Bhatt--Lurie in \cite[Definition~6.1.1]{BhattFGauge}. The stack $\oh_K^{\syn}$ is obtained from the Nygaard stack $\oh_K^{\mathcal{N}}$ (or, filtered prismatization of $\spf(\oh_K)$) by gluing certain open substacks. Namely, there is a stack $\oh_K^{\Prism}$ (called prismatization of $\spf(\oh_K)$) and a pair of open embeddings $j_{\dr}, j_{\HT}:\oh_K^{\Prism} \to \oh_K^{\mathcal{N}}$ such that $\oh_K^{\mathrm{syn}}$ is obtained by gluing the image $j_{\dr}(\oh_K^{\Prism})$ with $j_{\HT}(\oh_K^{\Prism})$ along $j_{\HT}j_{\dr}^{-1}$. Denote by $j_{\mathcal{N}}$ the resulting \'{e}tale quotient map $\oh_K^{\mathcal{N}}\to\oh_K^{\mathrm{syn}}$, and by $j_{\Prism}$ the resulting inclusion of $\oh_K^{\Prism}$ to $\oh_K^{\syn}$. An $F$-gauge is then an object $E$ of the derived category $D_{qc}(\oh_K^{\mathrm{syn}})$ of quasicoherent sheaves on $\oh_K^{\mathrm{syn}}$.

To specify \textit{coherent} $F$-gauges, let us connect the picture to Breuil--Kisin theory. By \cite[Remark~5.5.19, Example~5.5.20]{BhattFGauge}, there is a flat cover of stacks $$\rho_{\Es}^{\mathcal{N}}: \spf(\mathrm{Rees}(E(u)^{\bullet}\Es))/\Gm= \spf(\Es[t_-, t_+]/(t_-t_+-E(u)))/\Gm\to \oh_K^{\mathcal{N}}\,.$$
\begin{remark}\label{rem:ReesConventions}
In the notation of \cite{BhattFGauge}, $t_-$ is called $u$ and $t_+$ is called simply $t$. We change the notation to $t_+, t_-$ since $u$ is reserved for the indeterminate in the Breuil--Kisin prism $\Es$. The ring $\mathrm{Rees}(E(u)^{\bullet}\Es)=\Es[t_-, t_+]/(t_-t_+-E(u))$ is graded, with $\deg(t_+)=1$ and $\deg(t_-)=-1$, hence the $\pm$ notation. The quotient by $\Gm$ has the effect that the category of coherent sheaves on $\spf(\mathrm{Rees}(E(u)^{\bullet}\Es))/\Gm$ is equivalent to the category finitely generated graded $\Es[t_-, t_+]/(t_-t_+-E(u))$-modules. 
\end{remark}

Composing $\rho_{\Es}^{\mathcal{N}}$ with $j_{\mathcal{N}}$ then yields a flat cover of $\oh_K^{\mathrm{syn}}$, denoted by $\rho_{\Es}^{\mathrm{syn}}$. 
Then $E\in D_{qc}(\oh_K^{\mathrm{syn}})$ is said to be coherent if and only if $(\rho_{\Es}^{\mathrm{syn}})^*E$ is a coherent $\mathrm{Rees}(E(u)^\bullet \Es)$-module (cf. \cite[\S 6.6.3]{BhattFGauge}). The resulting category $\coh(\oh_K^{\mathrm{syn}})$ is then an abelian category, in fact, the heart 
of the canonical $t$-structure on the related category $\perf(\oh_K^{\syn})$ of perfect $F$-gauges (cf. \cite[Definition~2.15]{Pentland}, \cite[Remark 5.5.19 and footnote 72 on p. 117]{BhattFGauge}). 

\begin{remark}\label{rem:ExactnessAfterFlatCover}
Given that  $\rho_{\Es}^{\mathrm{syn}}$
is a flat cover, the pullback functor preserves exactness and maps nonzero coherent $F$-gauges to nonzero modules over the Rees ring. In particular, a sequence $0 \to E_1 \to E_2 \to E_3 \to 0$ of coherent $F$-gauges is short exact if and only if the sequence $0 \to (\rho_{\Es}^{\mathrm{syn}})^*E_1 \to (\rho_{\Es}^{\mathrm{syn}})^*E_2 \to (\rho_{\Es}^{\mathrm{syn}})^*E_3 \to 0$ is short exact.
\end{remark}

\subsubsection{Relation to $F$-crystals and crystalline lattices} A key property of $\oh_K^{\Prism}$ is that vector bundles on $\oh_K^{\Prism}$ are precisely given by prismatic crystals in vector bundles (not yet $F$-crystals) on the ringed absolute prismatic site $(\spf(\oh_K)_{\Prism}, \oh_{\Prism})$. This follows from \cite[Proposition~8.15]{BhattLurie2}. To describe $F$-crystals, consider the Hodge--Tate ideal sheaf $\mathcal{I}$ on $\oh_K^{\Prism}$ (cf. \cite[Construction 5.1.18]{BhattFGauge}), a stacky avatar of the distinguished prismatic ideal $\mathcal{I}_{\Prism}$. A prismatic $F$-crystal in vector bundles (in perfect complexes, resp.) is then given by the datum of a vector bundle $G$ on $ \oh_K^{\Prism}$ (perfect complex on $\oh_K^{\Prism}$, resp.) together with an isomorphism of sheaves $\varphi^*G[1/\mathcal{I}] \simeq G[1/\mathcal{I}]$ in $D(\spf(\oh_K)_{\Prism}, \oh_{\Prism}[1/\mathcal{I}])$. Here $\varphi$ is a map on $\oh_K^{\Prism}$ lifting the Frobenius map on $(\oh_K^{\Prism})_{p=0}$, defined as in \cite[Remark 5.3.11]{BhattFGauge}.

There is a functor assigning to a coherent $F$-gauge $E$ a prismatic $F$-crystal in perfect complexes. Concretely, there is a ``structure map'' $\pi: \oh_K^{\En} \to \oh_K^{\Prism}$ defined in \cite[Definition~5.3.10 (3)]{BhattFGauge}. It is a retraction of $j_{\dr},$ i.e., we have $\pi j_{\dr}=\mathrm{id}_{\oh_K^{\Prism}}$, and we further have $\pi j_{\HT}=\varphi.$ By \cite[Remark~6.3.4]{BhattFGauge}, there is a natural span 
$$\varphi^*j_{\dr}^* M \leftarrow  \varphi^*\pi_* M \to j_{\HT}^*M$$
of $\mathcal{I}$-isogenies (i.e., maps that become an isomorphism after inverting $\mathcal{I}$), where $M=j_{\En}^*E$. Since $j_{\En}j_{\HT}=j_{\En}j_{\dr}=j_{\Prism}$, this gives an isomorphism $\varphi^*j_{\Prism}^*E[1/\mathcal{I}]\to j_{\Prism}^*E[1/\mathcal{I}]$, that is, the structure of a prismatic $F$-crystal in perfect complexes on $j_{\Prism}^*E$. Denote the result of this construction by $G(E)$.

The subcategory of $\coh(\oh_K^{\syn})$ we are interested in is the category $\coh^{\refl}(\oh_K^{\syn})$ of coherent reflexive sheaves, that is, coherent sheaves $E$ such that the natural map $E \to E^{\vee \vee}$ to their double-dual (for a natural internal Hom in $\coh(\oh_K^{\syn})$, \cite[Definition~2.17]{Pentland}) is an isomorphism. An alternative (original) definition is given in \cite[\S6.6.3]{BhattFGauge}, but by \cite[Lemma 3.2]{Pentland}, the two descriptions are equivalent. By \cite[Theorem~6.6.13]{BhattFGauge}, the functor $G$ induces an equivalence between $\coh^{\refl}(\oh_K^{\syn})$ and $\vect^{\varphi}(\spf(\oh_K)_{\Prism}, \oh_{\Prism})$. Composing this with the equivalence $\vect^{\varphi}(\spf(\oh_K)_{\Prism}, \oh_{\Prism}) \simeq \Rep_{\ZZ_p, \cris}(G_K)$ from Section~\ref{subsec:Fcrystals}, we get an equivalence 
$$\begin{tikzcd}
\coh^{\refl}(\oh_K^{\syn}) \ar[r, shift left=2, "T_{\syn}", "\sim"'] & \ar[l, shift left=2, "D_{\syn}", "\sim"'] \Rep_{\ZZ_p, \cris}(G_K)\,.
\end{tikzcd}$$

\begin{remark}
For $E \in \coh^{\refl}(\oh_K^{\syn}),$ the underlying crystal of the $F$-crystal $G(E)$ is simply the restriction $j_{\Prism}^*E$ along the open immersion $j_{\Prism}$. It follows that the functor $G$ is exact, since exactness of sequences of prismatic $F$-crystals depends only on the underlying crystal structure. The functor $T_{\syn}$ is therefore a composition of a pair of exact functors, in particular, it is itself exact. As expected, the inverse functors are not exact in general (cf. \cite[p. 146 after Theorem~6.6.13]{BhattFGauge}).
\end{remark}

\subsubsection{Relation to Breuil--Kisin theory} Finally, we relate the functor $\underline{\Em}$ with $D_{\syn},$ and deduce the main exactness theorem. By \cite[Construction~3.10]{BhattLurie2} (see also \cite[Construction~3.5]{DatPham}), there is a map $\rho_{\Es}^{\Prism}: \spf(\Es)\to \oh_K^{\Prism}$. By \cite[Proposition~8.15]{BhattLurie2}, for a prismatic crystal in vector bundles $\mathcal{E}$, $(\rho_{\Es}^{\Prism})^*\mathcal{E}$ is the sheaf on $\spf{\Es}$ corresponding to the $\Es$-module $\mathcal{E}((\Es, E(u)\Es))$ (where the latter $\mathcal{E}$ is interpreted as a sheaf on the absolute prismatic site of $\oh_K$). The same applies to $F$-crystals; that is, for a prismatic $F$-crystal $\mathcal{E}$ attached to a crystalline lattice $T$, the pullback $(\rho_{\Es}^{\Prism})^*\mathcal{E}$ is functorially identified with the associated Kisin module $\underline{\Em}(T)$ (cf. \cite[p.817]{DatPham}).

To understand an analogous pullback of the $F$-gauge $D_{\syn}(T)$ by $\rho_{\Es}^{\syn},$ let us recall the Rees dictionary. Given a $\ZZ$-filtered $(\Es, E(u)^{\bullet}\Es)$-module $(M, \Fil^{\bullet}M)$ (i.e., a module such that $E(u)^a\Fil^bM\subseteq \Fil^{a+b}M$ and $\bigcup_i \Fil^i M=M$), setting $R(M)=\bigoplus_{i \in \ZZ}\Fil^i M t_+^{-i}$ yields a $t_+$-torsion free $\Es[t_-, t_+]/(t_-t_+-E(u))$-module, with $t_-$ acting by $E(u)t_+^{-1}$.  Conversely, given a graded module $N$ over $\Es[t_-, t_+]/(t_-t_+-E(u))$, a corresponding filtered $\Es$-module is given by $S(N)=(N[1/t_+])_0$, with $\Fil^iS(N)$ given as the image of $N_{-i}$ under the map $N_{-i} \to (N[1/t_+])_0$ given by $ x \mapsto xt_+^i$. Here ${(-)}_j$ denotes taking the degree $j$ part of a graded module. 
This is an equivalence of the category of filtered modules and the category of $t_+$-torsion-free graded $\Es[t_-, t_+]/(t_-t_+-E(u))$ modules. Crucially for our purposes, we have the following immediate lemma: 

\begin{lemma}\label{lem:ReesExactness}
In the situation above, let us call a sequence  
\begin{equation}\label{display:FilteredSES}\tag{$*$}
0 \to (M_1, \Fil^\bullet M_1) \to (M_2, \Fil^\bullet M_2) \to (M_3, \Fil^\bullet M_3) \to 0
\end{equation}
of filtered $(\Es, E(u)^{\bullet})$-modules short exact if the induced sequence $0 \to \Fil^iM_1 \to \Fil^iM_2\to \Fil^iM_3 \to 0$ is exact for every $i \in \ZZ$. Then ($*$) is short exact if and only if the sequence $0 \to R(M_1)\to R(M_2)\to R(M_3)\to 0$ is a short exact sequence of (graded) $\Es[t_-, t_+]/(t_-t_+-E(u))$-modules.
\end{lemma} 

Let us note that under the above correspondence, the graded module $R((M, \Fil^{\bullet}M))$ is finitely generated if and only if $M$, hence also each $\Fil^i M$, is finitely generated $\Es$-module, and furthermore $\Fil^i M=M$ for all $i\ll 0$ and $\Fil^{j+k} M=E(u)^j \Fil^{k} M$ for some fixed $k$ and all $j\gg 0$ (the former being a consequence of $\Es$ being noetherian). In particular, this is the case for the $\Es$-modules $(\varphi^*\Em, \Fil^\bullet \varphi^*\Em)$ considered in Section~\ref{subsec:KisinFilteredArgument}.

Given $(M, \Fil^{\bullet}M)$ as above, denote further by $\widetilde{R}((M, \Fil^{\bullet}M))$ the coherent sheaf on $\spf(\mathrm{Rees}(E(u)^{\bullet}\Es))/\Gm$ associated with the graded module $R((M, \Fil^{\bullet}M))$ as per Remark~\ref{rem:ReesConventions}. In this notation, the final key input is the following.
\begin{lemma}[{\cite[Lemma~3.31]{DatPham}}]\label{lem:PullbackToSyntomicBKcover}
Let $T$ be a crystalline $G_K$-lattice. Then we have the functorial identification 
$$(\rho_{\Es}^{\syn})^*D_{\syn}(T) \simeq \widetilde{R}((\varphi^*\underline{\Em}(T), \Fil^\bullet\varphi^*\underline{\Em}(T))),$$
where $\Fil^\bullet$ is the Nygaard filtration on the Kisin module $\underline{\Em}(T)$ as in Definition~\ref{def:NygaardFiltration}.
\end{lemma}

Finally, we can import Theorem~\ref{thm:FGaugeKisinExactness} to the language of $F$-gauges:

\begin{theorem}\label{thm:ExactnessFGauges}
Given a short exact sequence of crystalline lattices with non-negative Hodge--Tate weights
\begin{equation}\tag{$\tau$}
\begin{tikzcd}
0\ar[r]& T \ar[r] & L \ar[r] & \ZZ_p \ar[r] & 0,
\end{tikzcd}
\end{equation}
the corresponding sequence of coherent reflexive $F$-gauges
\begin{equation}\tag{$\delta$}
\begin{tikzcd}
0\ar[r]& D_{\syn}(T) \ar[r] & D_{\syn}(L) \ar[r] & \oh \ar[r] & 0,
\end{tikzcd}
\end{equation}
with $\oh=D_{\syn}(\ZZ_p),$
is a short exact sequence in $\coh(\oh_K^{\syn})$.
\end{theorem}

\begin{proof}
In order to verify the exactness of ($\delta$), by Remark~\ref{rem:ExactnessAfterFlatCover} it is enough to show the exactness of $(\rho_{\Es}^{\syn})^*\text{(}\delta\text{)}$. By Remark~\ref{rem:ReesConventions} and Lemma~\ref{lem:PullbackToSyntomicBKcover}, we equivalently need to show that the sequence of graded $\Es[t_-, t_+]/(t_-t_+-E(u))$-modules
\begin{equation*}
\begin{tikzcd}
0 \ar[r] &  R((\varphi^*\underline{\Em}(T), \Fil^\bullet\varphi^*\underline{\Em}(T))) \ar[r]  &  R((\varphi^*\underline{\Em}(L), \Fil^\bullet\varphi^*\underline{\Em}(L)))
  \ar[d, phantom, ""{coordinate, name=Z}]
  \ar[d,
    rounded corners,
    to path={ -- ([xshift=2ex]\tikztostart.east)
              |- (Z) [near end]\tikztonodes
              -| ([xshift=-2ex]\tikztotarget.west)
              -- (\tikztotarget)}] 
  \\
 & & R((\varphi^*\underline{\Em}(\ZZ_p), \Fil^\bullet\varphi^*\underline{\Em}(\ZZ_p)))  \ar[r] & 0
\end{tikzcd}
\end{equation*}
is short exact. By Lemma~\ref{lem:ReesExactness}, this directly translates to exactess of the filtered modules 
\begin{equation*}
\begin{tikzcd}
0 \ar[r] &  (\varphi^*\underline{\Em}(T), \Fil^\bullet\varphi^*\underline{\Em}(T)) \ar[r]  &  (\varphi^*\underline{\Em}(L), \Fil^\bullet\varphi^*\underline{\Em}(L))
  \ar[d, phantom, ""{coordinate, name=Z}]
  \ar[d,
    rounded corners,
    to path={ -- ([xshift=2ex]\tikztostart.east)
              |- (Z) [near end]\tikztonodes
              -| ([xshift=-2ex]\tikztotarget.west)
              -- (\tikztotarget)}] 
  \\
 & & (\varphi^*\underline{\Em}(\ZZ_p), \Fil^\bullet\varphi^*\underline{\Em}(\ZZ_p))  \ar[r] & 0
\end{tikzcd}
\end{equation*}
(exactness understood as in Lemma~\ref{lem:ReesExactness}). Finally, this exactness follows directly from Theorem~\ref{thm:FGaugeKisinExactness}.
\end{proof}

We end this appendix on the prevalent theme of this paper: the exactness statement translates into a presentation for Bloch--Kato Selmer groups.

\begin{corollary}\label{cor:BlochKatoBySyntomic}
Let $T$ be a lattice in a crystalline representation with non-negative Hodge--Tate weights. Then the syntomic cohomology $\H^1(\oh_K^{\syn}, D_{\syn}(T))$ is canonically isomorphic to $H^1_f(K, T)$.
\end{corollary}

\begin{proof}
By \cite[Proof of Proposition 4.3]{Pentland}, one has $$\H^1(\oh_K^{\syn}, D_{\syn}(T)) \simeq \ext^1_{\coh(\oh_K^{\syn})}(\oh, D_{\syn}(T)).$$
Both $\oh, D_{\syn}(T)$ are reflexive coherent $F$-gauges, and by \cite[Proof of Lemma~1.2]{YuMin}, reflexive coherent $F$-gauges are stable under extensions. Thus, the extensions classified by $\H^1(\oh_K^{\syn}, D_{\syn}(T))$ are all short exact sequences consisting only of reflexive coherent $F$-gauges, moreover such that their attached crystalline lattices all have non-negative Hodge--Tate weights. In particular, the equivalence of categories given by $D_{\syn}$ and $T_{\syn}$ preserves the exactness of all these short exact sequences in both directions. This proves the claim.
\end{proof}

\begin{remark}
Let us recall that $\H^1(\oh_K^{\syn}, D_{\syn}(T))[1/p]\simeq H^1_f(K, T[1/p])$, as was proved by Bhatt--Lurie when $K=\QQ_p$ \cite[Remark~6.7.4]{BhattFGauge} and extended to general $K/\QQ_p$ by Pentland \cite[Proposition~1.4]{Pentland}. Corollary~\ref{cor:BlochKatoBySyntomic} shows that this rational relationship can be yet again improved to an integral one, assuming that all Hodge--Tate weights of $T$ are non-negative. In view of Lemma~\ref{lem:PullbackToSyntomicBKcover}, Example~\ref{ex:KeyCounterexample} once again shows that this result does not extend, in general, to lattices with negative Hodge--Tate weights.
\end{remark}

\bibliography{references}
\bibliographystyle{amsalpha}

\end{document}